\documentclass[a4paper]{article}
\usepackage{amssymb}
\usepackage{amsmath}
\usepackage{amsthm}
\usepackage{color}
\usepackage{amscd}
\usepackage{graphicx}
\usepackage{pdflscape}
\usepackage{longtable,booktabs}
\usepackage{float}
\usepackage[dvipsnames]{xcolor}
\usepackage{placeins}

\newtheorem{thm}{Theorem}[section]
\newtheorem{proposition}[thm]{Proposition}
\newtheorem{lemma}[thm]{Lemma}
\newtheorem{algorithm}[thm]{Algorithm}

\newtheorem{cor}[thm]{Corollary}
\newtheorem{remark}[thm]{Remark}
\newtheorem{notation}[thm]{Notation}
\newtheorem{definition}[thm]{Definition}

\newtheorem{secfourconvention}[thm]{Convention}
\newtheorem{filter}[thm]{Filter}
\newtheorem{conjecture}[thm]{Conjecture}
\newcommand{\GL}{\mathrm{GL}}
\newcommand{\SL}{\mathrm{SL}}
\newcommand{\Sing}{\mathrm{Sing}}

\newcommand{\Supp}{\mathrm{Supp}}
\newcommand{\rank}{\mathrm{rank}}
\newcommand{\Sym}{\mathrm{Sym}}
\newcommand{\C}{\mathbb{C}}
\newcommand{\diag}{\mathrm{diag}}
\newcommand{\G}{\mathbb{G}}
\newcommand{\PP}{\mathbb{P}}
\newcommand{\wt}{\mathrm{wt}}
\newcommand{\Tor}{\mathrm{Tor}}
\newcommand{\HF}{\mathrm{HF}}
\newcommand{\Ann}{\mathrm{Ann}}

\newcommand{\nf}{\phi^{\mathrm{nf}}}

\newcommand{\Sp}{\mathrm{Sp}}
\newcommand{\Q}{\mathbb{Q}}
\newcommand{\Z}{\mathbb{Z}}
\newcommand{\R}{\mathbb{R}}

\begin{document}

\title{Boundary of the Moduli Space of Stable Cubic Fivefolds}
\author{Yasutaka Shibata}
\date{}
\maketitle

\begin{abstract}
We give a description of the strictly semistable boundary of the GIT moduli space of cubic fivefolds. Specifically, we classify all $21$ boundary components and determine the \emph{codimension-one wall adjacency} relation among them in the sense of Kirwan wall-crossing. The resulting adjacency graph is computed explicitly: it has $21$ vertices and $56$ edges (in particular, it has no isolated vertices). For each boundary component, we compute the minimal exponent of a general member and prove that it equals $7/3$, which is the critical value appearing in the stability criterion for cubic fivefolds.

Our approach is explicit: we exhibit $21$ normal forms representing the general polystable member in each boundary component and analyze the corresponding singularities. In particular, we identify exactly six analytic types of isolated singularities occurring on the boundary---\emph{extremal cubic fivefold singularities}---each given by a quasi-homogeneous normal form (in some cases with moduli) and each having minimal exponent $7/3$. We also treat the non-isolated cases: we classify the possible positive-dimensional singular loci---lines, conics, complete intersections of two quadrics, and an elliptic quartic curve---and compute their minimal exponents.
\end{abstract}

\section*{Introduction}

This paper studies the Geometric Invariant Theory (GIT) compactification of the moduli space of cubic fivefolds, i.e.\ cubic hypersurfaces $X\subset \mathbb{P}^{6}$, via the natural action of $\SL(7)$ on
\[
\mathbb{P}\bigl(\Sym^{3}\mathbb{C}^{7}\bigr)
\quad\text{and the quotient}\quad
\mathbb{P}\bigl(\Sym^{3}\mathbb{C}^{7}\bigr)// \SL(7).
\]
Our goal is to describe not only the stable locus but also the boundary contributed by strictly semistable cubics,
i.e.\ the image of the strictly semistable locus in the GIT quotient.
While the case of cubic fourfolds has been intensively studied from both classical and modern perspectives,
a comparably explicit description of the GIT boundary for cubic fivefolds has not been fully developed:
in particular, one would like concrete normal forms for closed orbits, an explicit description of the boundary families,
and a systematic analysis of the singularities that govern the geometry of degenerations.

From the GIT viewpoint, the boundary arises from the strictly semistable locus.
Understanding it naturally breaks into three tasks:
(i) identify polystable points (closed orbits) and construct explicit normal forms representing them;
(ii) understand the boundary families they generate, together with the corresponding parameter spaces and dimensions; and
(iii) extract geometrically meaningful information by analyzing the singular loci $\Sing(X)$ of these boundary representatives.
The main contribution of this paper is to carry out these steps as explicitly as possible for cubic fivefolds,
and to provide strong theoretical and computational evidence for the remaining structural assertions.

\medskip
\noindent\textbf{Irreducible components and normal forms.}
We first construct explicit polystable (closed-orbit) normal forms and the corresponding boundary families
$\Phi_k$ ($1\le k\le 21$), together with their parameter spaces and dimensions.
We then prove the non-inclusion among these $21$ families (Theorem~\ref{non-iclusion});
in particular, the families $\Phi_k$ are pairwise distinct and give exactly the $21$ irreducible components of the strictly semistable locus;
see Theorem~A.
In any case, the construction provides a concrete coordinate-level model for the boundary representatives and their degenerations.

\medskip
\noindent\textbf{Singular loci on the boundary.}
Next, we determine the singular locus of a general closed-orbit representative $X_k$ in each boundary family.
We show that $\Sing(X_k)$ consists of a one-dimensional part (when present) together with finitely many isolated singular points.
Moreover, the positive-dimensional part is a union of low-degree curves---namely lines, smooth conics, complete intersections of two quadrics,
or an elliptic quartic---whereas isolated singularities occur only in finitely many families.
Crucially, on these general representatives, all isolated boundary singularities are quasi-homogeneous and fall into precisely six analytic types.
These statements are summarized in Theorem~C.

\medskip
\noindent\textbf{The critical minimal exponent threshold and extremal singularities.}
A second theme of the paper is that these six isolated singularity types are not accidental by-products of the classification:
they capture the \emph{exact boundary} of stability in a sharp sense coming from singularity theory.
Building on Park's framework relating (local) minimal exponents to GIT stability for hypersurfaces,
we show that every isolated boundary singularity appearing here has local minimal exponent equal to
\[
\frac{7}{3}=\frac{n+1}{d}\qquad\text{for }(n,d)=(6,3),
\]
that is, it realizes the equality case at the critical threshold.
In Park's terminology, this is precisely the borderline regime where semistability can still occur while stability is no longer forced.
Motivated by this, we package these six analytic types as a distinguished class and refer to them as
\emph{extremal cubic fivefold singularities}; see Definition~\ref{def:extremal-sing} and Theorem~C.

\medskip
\noindent\textbf{Global minimal exponents on the boundary.}
The extremal value $\frac{7}{3}$ is not confined to isolated points.
We compute minimal exponents along all positive-dimensional singular loci that occur on strictly semistable closed orbits
(lines, conics, complete intersections of two quadrics, and an elliptic quartic), and show that possible drops occur only at finitely many
explicitly described special points.
As a consequence, the global minimal exponent of a general strictly semistable closed-orbit cubic fivefold in each boundary component is equal to the same critical threshold $7/3$ (see Theorem D).

\medskip
\noindent\textbf{Adjacency of boundary families.}
Finally, we analyze codimension-one wall adjacency relations among the boundary families arising from the closed-orbit normal forms.
In a setting with many boundary families, their \emph{connectivity} and \emph{proximity} control the local structure of the moduli space
and are closely tied to Kirwan-type wall-crossing.
Using Kirwan’s stratification together with an explicit slice-maximality computation on the weight set $I=\mathbb{Z}^7_{\ge 0}(3)$,
we determine the full codimension-one wall-adjacency graph among the closed strata
$\{\Phi_k\}_{k=1}^{21}\subset \Sigma^{\mathrm{ss}}$.
Contrary to a sparse picture, the resulting adjacency graph is nontrivial: it has $21$ vertices and $56$ edges (Theorem~E and Theorem~\ref{thm:adjacency}),
and in particular it has no isolated vertices.
For clarity, we distinguish genuine wall adjacencies (mutual slice matches, drawn as solid edges in Figure~\ref{fig:adjacency-codim1})
from \emph{one-sided} slice matches (drawn as dotted arrows), which do not represent codimension-one adjacency but rather reflect
higher-codimension incidence among faces.
Altogether, this provides an explicit combinatorial skeleton for the boundary, and serves as a foundation for further investigations of local deformation models
and related compactifications.

\medskip
\noindent\textbf{Idea of the proofs.}
Our approach combines standard GIT techniques---$1$-parameter-subgroup limits, computations of stabilizers and centralizers,
and normalizations via residual group actions---with explicit calculations tailored to cubic fivefolds.
For each boundary family we construct concrete normal forms and read off the structure of the singular locus directly from these models.
The minimal exponent statement is obtained by applying Park's criterion to our explicit list of isolated singularities and verifying the critical equality
$\widetilde{\alpha}_x(X)=\frac{7}{3}$.
The adjacency results follow from comparing degenerations along Kirwan strata and tracking how the boundary normal forms specialize.
Finally, the non-inclusion among the boundary families is established through a successive sieve of geometric and algebraic filters, utilizing the semicontinuity of dimensions, apolar Betti numbers, singular-locus Hilbert functions, minimal exponents, and Steenbrink spectra.

\medskip
\noindent\textbf{Organization.}
The paper provides a concrete ``dictionary'' for the GIT boundary of cubic fivefolds: the complete list of components, explicit closed-orbit normal forms, the associated singularity types, the critical minimal-exponent characterization, and the full adjacency graph.
We hope that this dictionary will serve as a basis for subsequent work on the geometry of degenerations of cubic fivefolds, local deformation theory near boundary points, and comparisons with other notions of stability.

\medskip

\noindent
\textbf{Theorem A (Boundary families and components).}
We construct $21$ explicitly described $\SL(7)$-invariant families
\[
\Phi_1,\Phi_2,\ldots,\Phi_{21}\ \subset\ \PP(W)^{ss}\setminus \PP(W)^s
\qquad (W=\Sym^{3}\C^{7}),
\]
together with explicit closed-orbit (polystable) representatives in normal form; see Table~2
(Section~\ref{sec:closed-orbit}).
Moreover, we prove the non-inclusion among these $21$ families (Theorem~\ref{non-iclusion}).
In particular, the families $\Phi_k$ are pairwise distinct and give exactly the $21$ irreducible components of
$\PP(W)^{ss}\setminus \PP(W)^s$.

\medskip

\noindent
\textbf{Theorem B (Closed-orbit test and normal-form construction).}
Let $G=\SL(7)$ act on $W=\Sym^{3}\C^{7}$.
For each boundary family $\Phi_k$ $(k=1,\dots,21)$,
choose a $1$-PS limit with stabilizer $H\subset G$ and consider the induced effective action of
the centralizer $C_G(H)$ on $W^H$.
Then closed orbits occurring in $\Phi_k$ can be represented by elements $\phi_k\in W^H$,
and $\phi_k$ can be normalized to a normal form by the residual $C_G(H)$-action.
Moreover, the closedness of $C_G(H)\cdot\phi_k$ (hence of $G\cdot\phi_k$) is decided uniformly by a dichotomy:
in the toric case by the convex-hull criterion for weights, and otherwise by the
Casimiro--Florentino criterion via a positive linear identity.
In particular, Table~2 lists explicit closed-orbit normal forms associated with all $\Phi_k$.

\medskip

\noindent
\textbf{Theorem C (Singularities on the boundary).}
Let $\nf_k$ be a general closed--orbit representative associated with $\Phi_k$ (Table~2) and set
$X_k:=V(\nf_k)\subset \PP^6$.
Then $\Sing(X_k)$ is explicitly determined (Table~3): it is a union of low--degree curves
(lines, smooth conics, complete intersections of two quadrics, or an elliptic quartic), possibly together with isolated points.
Isolated points occur exactly for $k\in\{1,5,7,8,9,11,16,19,21\}$; they are quasi--homogeneous and fall into six analytic types
(Definitions~5.1--5.6), including wild examples, and satisfy the extremal equality
\[
\widetilde{\alpha}_x(X_k)=\frac{7}{3}=\frac{n+1}{d}\qquad\text{for }(n,d)=(6,3).
\]

\medskip

\noindent
\textbf{Theorem D (Global minimal exponents on the boundary).}
For each $k\in\{1,\dots,21\}$ let $X_k:=V(\nf_k)\subset \PP^6$ be a general closed--orbit cubic fivefold.
Then the \emph{global minimal exponent} of $X_k$,
\[
\widetilde{\alpha}(X_k)\ :=\ \min_{x\in \Sing(X_k)} \widetilde{\alpha}_x(X_k),
\]
is equal to the extremal threshold value
\[
\widetilde{\alpha}(X_k)=\frac{7}{3}\qquad\text{for all }k=1,\dots,21.
\]

More precisely, if $\Sing(X_k)$ contains isolated points, then by Theorem~C each such point $x$
satisfies $\widetilde{\alpha}_x(X_k)=7/3$, hence $\widetilde{\alpha}(X_k)=\frac{7}{3}$.
If $\Sing(X_k)$ contains positive--dimensional components (lines, conics, $(2,2)$-complete intersections,
or an elliptic quartic), then $\widetilde{\alpha}_x(X_k)$ is constant at a general point of each irreducible
component, while it may drop at finitely many special points; in every case the minimum along each component
is attained at an explicitly listed special point and equals $7/3$.
Consequently, across the entire singular locus one has
\[
\min_{x\in \Sing(X_k)}\widetilde{\alpha}_x(X_k)=\frac{7}{3}.
\]
In particular, a general strictly semistable closed-orbit cubic fivefold in each of the 21 irreducible boundary components realizes the equality case
\[
\widetilde{\alpha}(X_k)=\frac{n+1}{d}\qquad\text{for }(n,d)=(6,3).
\]

\medskip

\noindent
\textbf{Theorem E (Codimension-one wall adjacency via wall-crossing).}
Write $\Sigma^{\mathrm{ss}}:=\PP(W)^{\mathrm{ss}}\setminus \PP(W)^{\mathrm{s}}$ for the strictly semistable locus.
In the sense of Kirwan wall-crossing (Definition~\ref{def:adjacency}), the codimension-one wall-adjacency graph
among the closed strata $\{\Phi_k\}_{k=1}^{21}\subset \Sigma^{\mathrm{ss}}$ has $21$ vertices and $56$ edges.
Equivalently, for each $k$ the neighbor set
\[
N(k):=\{\,j\neq k \mid \Phi_k\ \text{is adjacent to}\ \Phi_j\,\}
\]
is given explicitly in Table~\ref{tab:adjacency-codim1} (and visualized in Figure~\ref{fig:adjacency-codim1}).
All other unordered pairs $\{\Phi_i,\Phi_j\}$ are not adjacent.
In particular, the adjacency graph has \emph{no isolated vertices}.
See Theorem~\ref{thm:adjacency} and \cite{Kir84,DH98,Tha96}.

\medskip

\subsection*{Ideas and methods}

Our starting point is a convex-geometric analysis of Hilbert--Mumford weights for the natural
$\SL(7)$-action on $W=\Sym^{3}(\C^{7})$ \cite{MFK94, Dol03}.
Fix a maximal torus $T\subset \SL(7)$, and for a cubic $f\in W$ write $\Supp(f)\subset I = \mathbb{Z}^7_{(3)} \cap \mathbb{Z}^7_{\ge0}$ for its
set of exponent vectors. We enumerate the \emph{maximal $T$-strictly semistable supports}
$I(r)_{\ge 0}$ whose convex hull contains the barycenter $\eta=(3/7,\dots,3/7)$.
An algorithmic search over the faces of $\mathrm{Conv}(I)$ containing $\eta$ produces $23$ candidates up
to permutation; exactly one is $T$-unstable and is discarded, leaving $22$ families with respect
to $T$. Passing to the full $\SL(7)$-action, there is a single residual identification
$f_{21}\sim f_{22}$. We prove that no further identifications occur: namely, we establish the
non-inclusion among the $21$ boundary families (Theorem~\ref{non-iclusion}; see Section~\ref{sec:non-inclusion}).
This leads to the desired list of $21$ $\SL(7)$-inequivalent boundary families
(Algorithm~2.2, Proposition~3.4, and Theorem~\ref{21-list}).

For each family, we obtain closed $\SL(7)$-orbits by taking appropriate one-parameter subgroup limits
and applying Luna's centralizer reduction \cite{Lun75}. If the centralizer is a torus, we apply the convex-hull criterion;
if it is non-toric, we use the Casimiro--Florentino criterion \cite{CF12}. This uniform dichotomy produces explicit
closed-orbit normal forms, from which the family dimensions follow (Section~4; see also Table~2).

\subsection*{Organization of the paper}

Section~1 recalls the numerical criterion for (semi)stability and reformulates it in convex--geometric terms via supports and weight half-spaces.
Section~2 enumerates the maximal $T$-strictly semistable supports for a fixed maximal torus $T\subset \SL(7)$ (Algorithm~2.2), yielding $22$ candidates up to permutation.
Section~3 passes from the $T$-data to the full $\SL(7)$-action, removes the single residual identification, and obtains $21$ $\SL(7)$-inequivalent boundary families $\Phi_1,\dots,\Phi_{21}$ (with generic forms $f_1,\dots,f_{21}$), together with the non-inclusion statement deferred to Section~8.
Section~4 constructs explicit closed-orbit (polystable) representatives $\nf_k$ for each family via $1$-PS limits and Luna's centralizer reduction, certifies polystability by either the convex-hull criterion (toric centralizers) or the Casimiro--Florentino criterion (non-toric centralizers), and computes the dimensions of the corresponding components (summarized in Table~2).
Section~5 determines the singular scheme $\Sing(X_k)$ of each closed-orbit cubic $X_k:=V(\nf_k)\subset \mathbb{P}^6$, describing its positive-dimensional part (low-degree curves) and listing the isolated singularities; in particular, the isolated boundary singularities fall into exactly six quasi-homogeneous analytic types (summarized in Tables~3--4).
Section~6 studies minimal exponents: it first proves that each of the six isolated boundary singularity types has local minimal exponent $7/3$, and then computes the global minimal exponent $\tilde\alpha(X_k)$ for every $k=1,\dots,21$ by a case-by-case analysis along $\Sing(X_k)$.
Section~7 determines the adjacency relations among the closed strata $\{\Phi_k\}$ via wall-crossing in Kirwan’s stratification, showing that the adjacency graph has exactly 56 edges.
Section~8 proves the non-inclusion among the $21$ boundary families (Theorem~3.5) by a sieve of monotone necessary conditions (``filters''), culminating in fine geometric and algebraic obstructions based on limits of $1$-cycles, Steenbrink spectra of isolated singularities, and generic stabilizer tori that remove the final candidates.
Section~9 collects concluding remarks and outlines several future directions, including deformation/K-moduli questions for the extremal singularities and Hodge-theoretic and integrable-systems perspectives.

We present below a table comparing cubic threefolds and fourfolds.

\begin{longtable}{p{0.97\textwidth}}
\caption{Cubic threefolds and fourfolds: moduli dimension, GIT boundary components, and stability-singularity profile.}\\
\textbf{Cubic threefolds in $\mathbb{P}^4$} (See \cite{Yok02}.)\\
\emph{Moduli dimension.}
Number of monomials $\binom{5+3-1}{3}=35$, projective dimension $34$; 
$\dim \mathrm{PGL}_5=24$. Thus $\dim \mathcal{M}^{\mathrm{GIT}}=34-24=\mathbf{10}$.\\[0.25em]
\emph{Components of the GIT boundary (polystable closed orbits).}
Two irreducible components:
\begin{itemize}
\item[(i)] a $\mathbb{P}^1$-family $\{\phi_{\alpha,\beta}\}$ with parameter $[\alpha:\beta]$, and
\item[(ii)] the isolated point represented by $\phi=vwz+x^3+y^3$.
\end{itemize}
\emph{Singularity profile (stability).}
Stable $\Longleftrightarrow$ only double points of type $A_n$ with $n\le 4$.
Semistable $\Longleftrightarrow$ only $A_n$ $(n\le 5)$, $D_4$, or $A_\infty$ double points.
Inside the $\mathbb{P}^1$-component, the special member with $\alpha^2=4\beta$ is the secant threefold
(the singular locus is the rational normal curve).\\[0.25em]
\emph{Adjacency/closures.}
Both boundary components consist of closed orbits; the $\mathbb{P}^1$-component and the isolated point
are distinct boundary components (there is no specialization of one to the other).\\

\medskip

\textbf{Cubic fourfolds in $\mathbb{P}^5$} (See \cite{Laz09, Yok08, Huy23}.)\\
\emph{Moduli dimension.}
Number of monomials $\binom{6+3-1}{3}=56$, projective dimension $55$;
$\dim \mathrm{PGL}_6=35$. Thus $\dim \mathcal{M}^{\mathrm{GIT}}=\mathbf{20}$.\\[0.25em]
\emph{Boundary (closed orbit) families and their dimensions.}
There are six types, denoted $[C.1]$-$[C.6]$, of respective dimensions $1,2,3,1,1,0$.
A convenient set of normal forms is:
\begin{itemize}
  \item[$\,$\textbf{[C.1]}] $u\,q_1(w,x,y,z)+v\,q_2(w,x,y,z)$ with $V(u,v,q_1,q_2)$ smooth.
  \item[$\,$\textbf{[C.2]}] $u(xy+xz+yz+\alpha z^2)+v^2x+w^2y+2vwz$ (generic $\alpha$).
  \item[$\,$\textbf{[C.3]}] $uy^2+v^2z+l_1(w,x)\,uz+2\,l_2(w,x)\,vy+c(w,x)$ with $l_2^2\nmid c$, $l_1\nmid c$.
  \item[$\,$\textbf{[C.4]}] $uvw+c(x,y,z)$ with $V(u,v,w,c)$ smooth.
  \item[$\,$\textbf{[C.5]}] $\alpha\,uy^2+v^2z+w^2x-uxz+2vwy$ ($\alpha\neq 0$).
  \item[$\,$\textbf{[C.6]}] $uvw+xyz$.
\end{itemize}
\emph{Stability via singularities.}
A cubic fourfold with only isolated simple (ADE) singularities is GIT stable.
Conversely, non-stability occurs if any of the following conditions holds:
$\,$(1) $\mathrm{Sing}(X)$ contains a conic;
(2) $\mathrm{Sing}(X)$ contains a line;
(3) $\mathrm{Sing}(X)$ contains the intersection of two quadrics;
(4) $X$ has a double point of rank $\le 2$;
(5) a rank 3 double point with a hyperplane section whose singular locus is a line with ranks $\le 2$ along it;
(6) a rank 3 double point whose tangent-cone singular locus is a $2$-plane in $X$.\\[0.25em]
\emph{Adjacency (specialization) among boundary strata.}
If we denote the families $[C.1]$ to $[C.6]$ by $C_1, S_2, V_3, C_4, C_5, P_6$, then
\[
P_6\subset \overline{S_2}\cap \overline{V_3},\qquad
P_6\in \overline{C_1}\cap \overline{C_4}\cap \overline{C_5}.
\]
\medskip

\textbf{Cubic fivefolds in $\mathbb{P}^6$}\\
\emph{Moduli dimension.}
Number of monomials $\binom{7+3-1}{3}=84$, projective dimension $83$;
$\dim \mathrm{PGL}_7=48$. Thus $\dim \mathcal{M}^{\mathrm{GIT}}=\textbf{35}$.\\[0.25em]
\emph{Components of the GIT boundary (strictly semistable locus).}

\begin{remark}[Cubic threefolds and fourfolds rephrased via minimal exponents]
The classical GIT classifications of cubic threefolds and cubic fourfolds admit a succinct
reinterpretation in terms of the (global) minimal exponent $\widetilde{\alpha}$ in the sense of Park.

\smallskip
\noindent\textbf{Cubic threefolds.}
Let $X \subset \mathbb{P}^{4}$ be a cubic threefold. The critical value in Park's criterion is
$\frac{n+1}{d}=\frac{5}{3}$.
Park's theorem gives the sufficient conditions
$\widetilde{\alpha}(X)>\frac{5}{3}\Rightarrow X$ is GIT stable and $\widetilde{\alpha}(X) \ge \frac{5}{3}\Rightarrow X$ is GIT semistable
\cite[Theorem~A]{Par25}.
In fact, the explicit GIT analysis of cubic threefolds (see \cite{Yok02}) shows that

\begin{itemize}
\item[] $X$ is stable $\iff \widetilde{\alpha}(X) >\frac{5}{3}$
\item[] $X$ is semistable and not GIT-equivalent to the chordal cubic $\mathcal{T}$ $\iff \widetilde{\alpha}(X) \ge \frac{5}{3}$.
\end{itemize}

The remaining strictly semistable closed orbit is the chordal cubic $\mathcal{T}$, which is GIT polystable and satisfies
$\widetilde{\alpha}(\mathcal{T})=\frac{3}{2}$ \cite[Remark~7.6(2)]{Par25}.

\smallskip
\noindent\textbf{Cubic fourfolds.}
Let $X \subset \mathbb{P}^{5}$ be a cubic fourfold. Here $\frac{n+1}{d}=2$.
The classical results of Yokoyama--Laza (see \cite{Yok08,Laz09}) imply that
a cubic fourfold with only isolated simple (ADE) singularities is GIT stable, hence in particular
$\widetilde{\alpha}(X) > 2$.
Moreover, away from the one-parameter family $\chi$ of GIT polystable cubics,
semistability is characterized by the threshold $\widetilde{\alpha}(X) \ge 2$:
\begin{itemize}
\item[] $X$ semistable and not GIT-equivalent to a member of $\chi$ 
\item[] $\Rightarrow \widetilde{\alpha}(X) \ge 2 \Rightarrow$ $X$ semistable.
\end{itemize}
\cite[Remark~7.6(3)]{Par25}.

The exceptional family $\chi$ forms the polystable locus with $\widetilde{\alpha}(X)<2$:
the secant to the Veronese surface $\omega\in\chi$ has $\widetilde{\alpha}(\omega)=\frac{3}{2}$, while for all other
$X\in\chi\setminus\{\omega\}$ one has $\widetilde{\alpha}(X)=\frac{11}{6}$ \cite[Remark~7.6(3)]{Par25}.

\smallskip
These lower-dimensional cases provide useful prototypes for cubic fivefolds, where the critical value is
$\frac{n+1}{d}=\frac{7}{3}$ and our boundary isolated singularities realize the equality case at this threshold.
\end{remark}

\end{longtable}

\section*{Scripts used in this paper}
The scripts used in this paper are publicly available at \cite{Shi26}. 
In particular, the archive includes the scripts for Sections~\ref{algorithm}, \ref{sec:5}, \ref{sec:adjacency}, \ref{sec:non-inclusion}.

\section*{Use of AI Tools}
The author used ChatGPT and Gemini as auxiliary tools in several parts of the project. In Section~4, they were used to help explore and organize computations leading to candidate normal forms. In Section~5, they were used in the same way to help derive and organize candidate normal forms for isolated singularities. In Section~6, they were used to support the organization and exploration of the minimal exponent computations (e.g., structuring case-by-case checks and assisting with the interpretation and presentation of intermediate computational outputs). These tools were also used to support coding related to these computations (for example, drafting and debugging scripts). Finally, under the author's supervision, they were used for English language editing and for improving exposition, organization, and internal consistency of the manuscript through iterative discussions. All mathematical statements and results were ultimately checked and verified by explicit calculations carried out by the author, who takes full responsibility for the content.

\section*{Acknowledgments}

The author is grateful to S.\,G.~Park for valuable correspondence and insightful comments.
The author would like to thank Professor T.~Terasoma for his encouragement, advice, and valuable suggestions.
He is deeply grateful to his parents for their support during difficult times.
The author dedicates this work to the memory of Professor E.~Horikawa.

\section{Numerical criterion for cubic fivefolds}\label{num}
In this section, we review the numerical criterion for stability or semistability of cubic fivefolds. We use the following notations.

\begin{itemize}
\item Let $\mathbb{C}[x_{0},\cdots,x_{6}]_{3}$ be the set of homogeneous polynomials of degree $3$.

\item For a vector $\mathbf{x} \in \mathbb{Q}^7$, $\mathrm{wt}(\mathbf{x})=\sum_{k=0}^{6} x_{k}$ is called the weight of $\mathbf{x}$.

\item We define $\mathbb{Z}^{7}_{\ge 0}=\{\mathbf{x}=(x_{0},x_{1},\cdots,x_{6})  \in \mathbb{Z}^7 | x_{k} \ge 0 (k=0,1,\cdots,6) \}$,\\

$\mathbb{Z}^7_{(d)}=\{ \mathbf{x} \in \mathbb{Z}^7 | \mathrm{wt}(\mathbf{x})=d \}$,\\

$I=\mathbb{Z}^7_{(3)} \cap \mathbb{Z}^7_{\ge 0}$ and it is simply called the simplex.

\item For $\mathbf{r} \in \mathbb{Q}^{7}$, we define $I(\mathbf{r})_{\ge 0}=\{ \mathbf{i} \in I | \mathbf{r} \cdot \mathbf{i} \ge 0 \}$, $I(\mathbf{r})_{> 0}=\{ \mathbf{i} \in I | \mathbf{r} \cdot \mathbf{i} > 0 \}$ and $I(\mathbf{r})_{=0}=\{ \mathbf{i} \in I | \mathbf{r} \cdot \mathbf{i} = 0 \}$, here $\cdot$ denotes the standard inner product of vectors.

\item For a polynomial $f=\sum_{\mathrm{wt}(\mathbf{i})=3}a_{\mathbf{i}}x^{\mathbf{i}} \in \mathbb{C}[x_{0},\cdots,x_{6}]_{3}$, we define the support of $f$ by $\mathrm{Supp}(f)=\{ \mathbf{i} \in I | a_{\mathbf{i}} \neq 0\}$

\item We set $\eta=(3/7,3/7,3/7,3/7,3/7,3/7,3/7) \in \mathbb{Q}^7$ and it is called the barycenter of the simplex $I$.

\item A vector $\mathbf{r} \in \mathbb{Z}^{7}$ is said to be reduced when there is no integer $\alpha$ such that $|\alpha| \ge 2$ and $\frac{1}{\alpha}\mathbf{r} \in \mathbb{Z}^{7}$
\end{itemize}

We fix a maximal torus $T$ of $\SL(7)$. Consider a one-parameter subgroup (1-PS for short) $\lambda : \mathbb{G}_{m} \to \SL(7)$ whose image is contained in $T$. For a suitable basis of $\mathbb{C}^7$, $\lambda$ can be expressed as a diagonal matrix $\mathrm{diag}(t^{r_{0}},t^{r_{1}},\cdots,t^{r_{6}})$ where $t \neq 0$ is a parameter of $\mathbb{G}_{m}$. Let us choose and fix such basis. Then $\lambda$ corresponds to an element $\mathbf{r}=(r_{0},r_{1},\cdots,r_{6})$ in $\mathbb{Z}^{7}_{(0)}$. We can regard an element of $\mathbb{Z}^7_{(0)}$ as a 1-PS of $T$.

\begin{definition}
Let $s$ be a subset of $I$. We say that $s$ is not stable (resp. unstable) with respect to $T$ when $s \subseteq I(\mathbf{r})_{\ge 0}$ (resp. $s \subseteq I(\mathbf{r})_{> 0}$) for some 1-PS $\mathbf{r}$.
For $0 \neq f \in \mathbb{C}[x_{0},\cdots,x_{6}]_{3}$, we say that $f$ is not stable (resp. unstable) with respect to $T$ when $\mathrm{Supp}(f) \subseteq I$ is not stable (resp. unstable) with respect to $T$. 
\end{definition}

\begin{figure}[tbp]
  \centering
  \includegraphics[width=.7\linewidth]{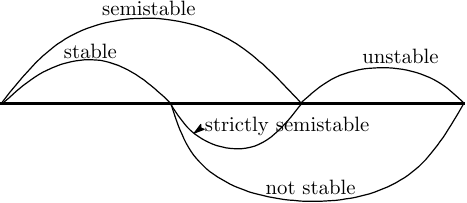}
  \caption{various concepts of stability}
  \label{fig:one}
\end{figure}

The following theorem is the numerical criterion for stability via the language of convex geometry.

\begin{thm}\label{numerical}
The cubic fivefold defined by $f \in \mathbb{C}[x_{0},\cdots,x_{6}]_{3}$ is not stable (resp. unstable) if and only if there exists an element $\sigma \in \SL(7)$ such that $f^{\sigma}$ is not stable (resp. unstable) with respect to $T$. 

In particular, $f$ is strictly semistable if and only if 
\begin{description}
\item[(1)] There exist $\sigma \in \SL(7)$ such that $f^{\sigma}$ is not stable with respect to $T$, and

\item[(2)] For any $\sigma \in \SL(7)$, $f^{\sigma}$ is semistable with respect to $T$.
\end{description}

\end{thm}
\begin{proof}
See Theorem 9.1 of \cite{Dol03}.
\end{proof}

\section{Maximal strictly semistable cubic fivefolds with respect to the maximal torus~$T$}\label{algorithm}

In this section, we list the irreducible components corresponding to strictly semistable cubic fivefolds. For this purpose, we list all strictly semistable cubic fivefolds with respect to the maximal torus $T$. 
To solve this problem, we will consider the set of maximal strictly semistable subsets of $I$. The order in the set of subsets of $I$ is given by inclusion. 
For this purpose, we list the set of all maximal elements of $\mathcal{S}=\{ I(\mathbf{r})_{\ge 0} | \mathbf{r} \in \mathbb{Z}_{(0)}^{7} \}$.

We solve this problem computationally. We need an algorithm which enables us to obtain them in finitely many steps. Before giving such an algorithm, we remark that $I(\mathbf{r})_{\ge 0}$ and $I(\mathbf{r'})_{\ge 0}$ might be the same for two different vectors $\mathbf{r},\mathbf{r'} \in \mathbb{Z}^7_{(0)}$. 

\begin{lemma}
Let $I(\mathbf{r})_{\ge 0}$ be a maximal element of $\mathcal{S}$, where $\mathbf{r} \in \mathbb{Z}^7_{(0)}.$ Then there exist $5$ elements $\mathbf{x}_{1},\mathbf{x}_{2},\cdots,\mathbf{x}_{5} \in I$ and a vector $\mathbf{r'} \in \mathbb{Z}^7_{(0)}$ such that they satisfy the following three conditions:
\begin{description}
\item[(1)]  The vector subspace $W$ of $\mathbb{Q}^7$ spanned by $\mathbf{x}_{1},\cdots,\mathbf{x}_{5},\eta$ over $\mathbb{Q}$ has dimension $6$

\item[(2)] The vector $\mathbf{r'}$ is orthogonal to the subspace $W$ of $\mathbb{Q}^7$.

\item[(3)] $I(\mathbf{r})_{\ge 0}=I(\mathbf{r'})_{\ge 0}$  
\end{description}
\end{lemma}

\begin{proof}
Let us put $C=I(r) \cup \{\eta \}$. We consider the convex-hull $\check{C}$ of $C$ in $\mathbb{Q}^7$. Let $F$ be a face of $\check{C}$ containing the point $\eta$. There is a normal vector $\mathbf{r}'$ of $F$ in $\mathbb{Z}^7_{(0)}$ such that $\check{C} \subseteq \{\mathbf{x} \in \mathbb{Q}^7 | \mathbf{r}' \cdot \mathbf{x} \ge 0 \}$. We have $\mathrm{wt}(\mathbf{r'})=0$ because the hyperplane defined by $\{ \mathbf{x} \in \mathbb{Q}^{7} | \mathbf{r}' \cdot \mathbf{x}=0 \}$ passes through the point $\eta$. By the definition of the faces of a convex set in $\mathbb{Q}^7$, we can take 5 points $\mathbf{x}_{1},\mathbf{x}_{2},\cdots,\mathbf{x}_{5}$ from the set $I \cap F$ such that $\mathbf{x}_{1},\mathbf{x}_{2},\cdots,\mathbf{x}_{5},\eta$ are linearly independent over $\mathbb{Q}$. In general we have $I(\mathbf{r})_{\ge 0} \subseteq I(\mathbf{r}')_{\ge 0}$, and by the assumption that $I(\mathbf{r})_{\ge 0}$ is maximal in $\mathcal{S}$, we conclude that $I(\mathbf{r})_{\ge 0}=I(\mathbf{r}')_{\ge 0}$. 
\end{proof}

By this lemma, we can determine the set of maximal elements of $\mathcal{S}$ up to permutations of coordinates in finite steps using the following algorithm.

\begin{algorithm}
Let $\mathcal{F}$ be the set of all ordered $5$-tuples
$\mathbf{x}=(\mathbf{x}_{1},\ldots,\mathbf{x}_{5})$ of \emph{pairwise distinct} points of $I$.
Fix a total order on $\mathcal{F}$.
As initial data, set $\mathcal{S}'=\varnothing$ and let
$\mathbf{x}=(\mathbf{x}_{1},\ldots,\mathbf{x}_{5})$ be the minimum element of $\mathcal{F}$.
We modify $\mathcal{S}'$ by the following algorithm.
\begin{itemize}

\item Step 1.
Let $W\subset \mathbb{Q}^{7}$ be the $\mathbb{Q}$-subspace spanned by
$\mathbf{x}_{1},\ldots,\mathbf{x}_{5},\eta$.
If $\dim_{\mathbb{Q}}(W)=6$, then take a \emph{reduced} normal vector
$\mathbf{r}=(r_{0},\ldots,r_{6})\in \mathbb{Z}^{7}_{(0)}\setminus\{0\}$ of $W$
(i.e.\ $\mathbf{r}\cdot \mathbf{w}=0$ for all $\mathbf{w}\in W$)
and go to Step~2; otherwise, go to Step~5.

\item Step 2.
If $\mathbf{r}$ satisfies $r_{0} \ge \cdots \ge r_{6}$ or $r_{0} \le \cdots \le r_{6}$,
then go to Step~3; otherwise, go to Step~5.

\item Step 3.
If $r_{0} \ge \cdots \ge r_{6}$ (resp.\ $r_{0} \le \cdots \le r_{6}$),
add $I(\mathbf{r})_{\ge 0}$ (resp.\ $I(-\mathbf{r})_{\ge 0}$) to $\mathcal{S}'$
and go to Step~4.

\item Step 4.
Delete all elements of $\mathcal{S}'$ that are not maximal in $\mathcal{S}'$
(with respect to inclusion), and go to Step~5.

\item Step 5.
Replace the element $\mathbf{x}$ with the next element if $\mathbf{x}$ is not the maximum element,
and go to Step~1. Otherwise, stop the algorithm.

\end{itemize}
\end{algorithm}

We note that Step 2 removes the $S_{7}$ symmetry on the variables $x_{0},\cdots,x_{6}$.
We also note that Step 4 is not essential but serves as technical measure to save memory. After running this algorithm with the aid of a computer, we find 23 elements $I(\mathbf{r}_{1})_{\ge 0},\cdots,I(\mathbf{r}_{23})_{\ge 0}$ in $\mathcal{S}'$, where $\mathbf{r}_{k}=(r_{0},\cdots,r_{6}) \in \mathbb{Z}^{7}_{(0)}$ is a reduced vector with $r_{0} \ge \cdots \ge r_{6}$. When we compute the convex-hulls of $I(\mathbf{r}_{1})_{\ge 0},\cdots,I(\mathbf{r}_{23})_{\ge 0}$ in $\mathbb{Q}^{7}$, only one of the convex-hulls of $I(\mathbf{r}_{k})_{\ge 0}$ does not contain $\eta$. 
We denote it as $I(\mathbf{r}_{23})_{\ge 0}$. As $I(\mathbf{r}_{23})_{\ge 0}$ is unstable with respect to $T$, we remove it from the list. Thus, we can conclude that there are 22 maximal strictly semistable cubic fivefolds for the fixed maximal torus $T$. Because of this algorithm, we have the following proposition.

\begin{proposition}\label{vectors}
The set $\mathcal{M}=\{ I(\mathbf{r}_{1})_{\ge 0},\cdots,I(\mathbf{r}_{22})_{\ge 0} \}$ is given as follows.

\begin{tabular}[t]{|c|c|}
\hline
$\mathbf{r}_{1}=(8,3,2,-1,-2,-4,-6)$ & $\mathbf{r}_{2}=(6,4,1,-1,-2,-3,-5)$ \\ \hline
$\mathbf{r}_{3}=(4,2,1,-1,-1,-2,-3)$ & $\mathbf{r}_{4}=(3,2,1,0,-1,-2,-3)$ \\ \hline
$\mathbf{r}_{5}=(4,2,1,0,-1,-2,-4)$ & $\mathbf{r}_{6}=(5,3,2,1,-1,-4,-6)$ \\ \hline
$\mathbf{r}_{7}=(6,4,2,1,-2,-3,-8)$ & $\mathbf{r}_{8}=(4,1,1,0,-2,-2,-2)$ \\ \hline
$\mathbf{r}_{9}=(2,2,0,0,-1,-1,-2)$ & $\mathbf{r}_{10}=(2,1,0,0,-1,-1,-1)$ \\ \hline
$\mathbf{r}_{11}=(2,0,0,0,0,-1,-1)$ & $\mathbf{r}_{12}=(3,2,1,1,-1,-2,-4)$ \\ \hline
$\mathbf{r}_{13}=(2,1,1,0,-1,-1,-2)$ & $\mathbf{r}_{14}=(2,2,0,-1,-1,-1,-1)$ \\ \hline
$\mathbf{r}_{15}=(2,1,1,0,0,-2,-2)$ & $\mathbf{r}_{16}=(2,1,0,0,0,-1,-2)$ \\ \hline
$\mathbf{r}_{17}=(1,1,1,0,0,-1,-2)$ & $\mathbf{r}_{18}=(1,1,0,0,0,-1,-1)$ \\ \hline
$\mathbf{r}_{19}=(2,2,2,0,-1,-1,-4)$ & $\mathbf{r}_{20}=(1,1,1,1,0,-2,-2)$ \\ \hline
$\mathbf{r}_{21}=(1,1,0,0,0,0,-2)$ & $\mathbf{r}_{22}=(1,0,0,0,0,0,-1)$ \\ \hline
\end{tabular}
\end{proposition}

For example, $I(\mathbf{r}_{1})_{\ge 0}$ is \\
$I(\mathbf{r}_{1})_{\ge 0}=\{ x_{0}^3,x_{0}^2 x_{1},x_{0}^2 x_{2},x_{0}^2 x_{3},x_{0}^2 x_{4},
x_{0}^2 x_{5},x_{0}^2 x_{6},x_{0} x_{1}^2,x_{0} x_{1} x_{2},x_{0} x_{1} x_{3},x_{0} x_{1} x_{4},\\
x_{0} x_{1} x_{5},x_{0} x_{1} x_{6},x_{0} x_{2}^2,x_{0} x_{2} x_{3},x_{0} x_{2} x_{4},x_{0} x_{2} x_{5},x_{0} x_{2} x_{6},x_{0} x_{3}^2,x_{0} x_{3} x_{4},x_{0}
   x_{3} x_{5},x_{0} x_{3} x_{6},\\
x_{0} x_{4}^2, x_{0} x_{4} x_{5},x_{0} x_{4} x_{6},x_{0} x_{5}^2,x_{1}^3,x_{1}^2 x_{2}, x_{1}^2 x_{3},x_{1}^2 x_{4},x_{1}^2 x_{5},x_{1}^2 x_{6},x_{1} x_{2}^2,
x_{1} x_{2} x_{3},x_{1} x_{2} x_{4},\\
x_{1} x_{2} x_{5},x_{1} x_{3}^2,x_{1} x_{3} x_{4},x_{2}^3,x_{2}^2 x_{3},x_{2}^2 x_{4},x_{2}^2 x_{5},x_{2}
   x_{3}^2 \}.$\\
Here we use the notation $x_{0}^{i_{0}}x_{1}^{i_{1}}\cdots x_{6}^{i_{6}}$ for an element $(i_{0},i_{1},\cdots,i_{6}) \in \mathbb{Z}^{7}_{(3)}$ in order to save space.

\begin{remark}
The following vectors can serve as $r_{23}$, i.e., there are several vectors that yield the set $I(\mathbf{r}_{23})_{\ge 0}$. For example, we can take $\mathbf{r}_{23}=(8, 5, 3, 2, -4, -4, -10)$.

\end{remark}

\begin{remark}
The algorithm in this section has been comprehensively generalized by \textnormal{\cite{GMMS23}}.
\end{remark}

\section{21 maximal strictly semistable cubic fivefolds under action of $\SL(7)$}\label{21}

An element $I(\mathbf{r}_{k})_{\ge 0}$ of $\mathcal{M}$ represents a family of cubic fivefolds whose defining polynomial's support is contained in $I(\mathbf{r}_{k})_{\ge 0}$. In this section, we analyze the inclusion relations among $I(\mathbf{r}_{k})_{\ge 0}$ under the action of $\SL(7)$.  Let $f_{k}$ be a generic polynomial whose support is $I(\mathbf{r}_{k})_{\ge 0}$. $(k=1,2,\cdots,22)$. If we express $f_{k}$ directly, it becomes too long, so we introduce notation.

\begin{definition}
The symbols $c,q,l,\alpha$ stand for a cubic form, a quadratic form, a linear form, and a constant term, respectively. Similarly, the symbols $q_{i},l_{i},\alpha_{i}$ denote the $i$-th quadratic form, a linear form, and a constant term, respectively.
\end{definition}

The following theorem is a direct consequence of the list in Proposition~\ref{vectors}.
\begin{thm}\label{22-list}
Using the above notations, the generic polynomials of $f_{1},\cdots,f_{22}$ are the following forms.
\begin{itemize}

\item $f_{1}=c(x_{0},x_{1},x_{2})+q_{1}(x_{0},x_{1},x_{2})x_{3}+l_{1}(x_{0},x_{1},x_{2})x_{3}^2+ \{ q_{2}(x_{0},x_{1},x_{2})+l_{2}(x_{0},x_{1})x_{3} \} x_{4}+ \alpha_{1} x_{0}x_{4}^2+ \{ q_{3}(x_{0},x_{1},x_{2})+x_{0}l_{3}(x_{3},x_{4}) \} x_{5} +\alpha_{2} x_{0}x_{5}^2+ \{ q_{4}(x_{0},x_{1})+x_{0}l_{4}(x_{2},x_{3},x_{4}) \} x_{6}$

\item $f_{2}=c(x_{0},x_{1},x_{2})+q_{1}(x_{0},x_{1},x_{2})x_{3}+l_{1}(x_{0},x_{1})x_{3}^2+ \{ q_{2}(x_{0},x_{1},x_{2})+l_{2}(x_{0},x_{1})x_{3} \} x_{4}+l_{3}(x_{0},x_{1})x_{4}^2+ \{ q_{3}(x_{0},x_{1})+l_{4}(x_{0},x_{1})x_{2}+l_{5}(x_{0},x_{1})x_{3}+ \alpha_{1}x_{0}x_{4} \}x_{5}+ \alpha_{2}x_{0}x_{5}^2+ \{  q_{4}(x_{0},x_{1})+l_{6}(x_{0},x_{1})x_{2}+ \alpha_{3} x_{0}x_{3}  \}x_{6}$

\item $f_{3}=c(x_{0},x_{1},x_{2})+q_{1}(x_{0},x_{1},x_{2})x_{3}+l_{1}(x_{0},x_{1})x_{3}^2+ \{ q_{2}(x_{0},x_{1},x_{2})+l_{2}(x_{0},x_{1})x_{3}  \} x_{4}+l_{3}(x_{0},x_{1})x_{4}^2+ \{ q_{3}(x_{0},x_{1},x_{2})+x_{0}l_{4}(x_{3},x_{4})  \} x_{5}+\alpha_{1}x_{0}x_{5}^2+ \{ q_{4}(x_{0},x_{1})+l_{5}(x_{0},x_{1})x_{2}+x_{0}l_{6}(x_{3},x_{4})  \} x_{6}$

\item $f_{4}=c(x_{0},x_{1},x_{2})+\{ q_{1}(x_{0},x_{1},x_{2})+l_{1}(x_{0},x_{1},x_{2})x_{3}+\alpha_{1}x_{3}^{2} \} x_{3}+ \{ q_{2}(x_{0},x_{1},x_{2})+l_{2}(x_{0},x_{1},x_{2})x_{3} \} x_{4}+l_{3}(x_{0},x_{1})x_{4}^2+ \{ q_{3}(x_{0},x_{1},x_{2})+l_{4}(x_{0},x_{1})x_{3}+\alpha_{2}x_{0}x_{4} \} x_{5}+ \{ q_{4}(x_{0},x_{1})+l_{5}(x_{0},x_{1})x_{2}+\alpha_{3}x_{0}x_{3} \} x_{6}$

\item $f_{5}=c(x_{0},x_{1},x_{2},x_{3})+\{ q_{1}(x_{0},x_{1},x_{2})+l_{1}(x_{0},x_{1},x_{2})x_{3} \} x_{4}+l_{2}(x_{0},x_{1})x_{4}^2+ \{ q_{2}(x_{0},x_{1},x_{2})+l_{3}(x_{0},x_{1})x_{3}+ \alpha_{1}x_{0}x_{4} \} x_{5}+ \alpha_{3} x_{0} x_{5}^2 + \{ q_{3}(x_{0},x_{1})+\alpha_{4} x_{0} x_{2}+ \alpha_{2} x_{0}x_{3} \} x_{6}$

\item $f_{6}=c(x_{0},x_{1},x_{2},x_{3})+q_{1}(x_{0},x_{1},x_{2},x_{3})x_{4}+l_{1}(x_{0},x_{1},x_{2})x_{4}^2+ \{ q_{2}(x_{0},x_{1},x_{2})+l_{2}(x_{0},x_{1})x_{3}+\alpha_{1}x_{0}x_{4} \} x_{5}+ \{ q_{3}(x_{0},x_{1})+x_{0}l_{3}(x_{2},x_{3}) \} x_{6}$

\item $f_{7}=c(x_{0},x_{1},x_{2},x_{3})+q_{1}(x_{0},x_{1},x_{2},x_{3})x_{4}+l_{1}(x_{0},x_{1})x_{4}^2+ \{ q_{2}(x_{0},x_{1},x_{2})+l_{2}(x_{0},x_{1},x_{2})x_{3}+\alpha_{1}x_{0}x_{4} \}x_{5}+\alpha_{2}x_{0}x_{5}^2+ \{ q_{3}(x_{0},x_{1})+\alpha_{3}x_{0}x_{2} \} x_{6}$

\item $f_{8}=c(x_{0},x_{1},x_{2},x_{3})+\{ q_{1}(x_{0},x_{1},x_{2})+\alpha_{1}x_{0}x_{3} \} x_{4}+\alpha_{2}x_{0}x_{4}^2+ \{ q_{2}(x_{0},x_{1},x_{2})+\alpha_{3}x_{0}x_{3}+\alpha_{4}x_{0}x_{4} \} x_{5}+\alpha_{5}x_{0}x_{5}^2+\{ q_{3}(x_{0},x_{1},x_{2})+\alpha_{6}x_{0}x_{3}+\alpha_{7}x_{0}x_{4}+\alpha_{8}x_{0}x_{5} \} x_{6}+\alpha_{9}x_{0}x_{6}^2$

\item $f_{9}=c(x_{0},x_{1},x_{2},x_{3})+\{ q_{1}(x_{0},x_{1})+l_{1}(x_{0},x_{1})x_{2}+l_{2}(x_{0},x_{1})x_{3} \} x_{4}+l_{3}(x_{0},x_{1})x_{4}^2+ \{ q_{2}(x_{0},x_{1})+l_{4}(x_{0},x_{1})x_{2}+l_{5}(x_{0},x_{1})x_{3}+l_{6}(x_{0},x_{1})x_{4} \} x_{5} +l_{7}(x_{0},x_{1})x_{5}^2 +\{ q_{3}(x_{0},x_{1})+l_{8}(x_{0},x_{1})x_{2}+l_{9}(x_{0},x_{1})x_{3} \} x_{6}$

\item $f_{10}=c(x_{0},x_{1},x_{2},x_{3})+\{ q_{1}(x_{0},x_{1})+l_{1}(x_{0},x_{1})x_{2}+l_{2}(x_{0},x_{1})x_{3} \} x_{4}+\alpha_{1}x_{0}x_{4}^2+\{ q_{2}(x_{0},x_{1})+l_{3}(x_{0},x_{1})x_{2}+l_{4}(x_{0},x_{1})x_{3} \} x_{5}+\alpha_{2}x_{0}x_{5}^2+\{ q_{3}(x_{0},x_{1})+l_{5}(x_{0},x_{1})x_{2}+l_{6}(x_{0},x_{1})x_{3} \} x_{6}+\alpha_{3}x_{0}x_{6}^2+\alpha_{4}x_{0}x_{4}x_{5}+\alpha_{5}x_{0}x_{4}x_{6}+\alpha_{6}x_{0}x_{5}x_{6}$

\item $f_{11}=c(x_{0},x_{1},x_{2},x_{3},x_{4})+x_{0}l_{1}(x_{0},x_{1},x_{2},x_{3},x_{4})x_{5}+\alpha_{1}x_{0}x_{5}^2+x_{0}l_{2}(x_{0},x_{1},x_{2},x_{3},x_{4})x_{6}+\alpha_{2}x_{0}x_{6}^2+\alpha_{3}x_{0}x_{5}x_{6}$

\item $f_{12}=c(x_{0},x_{1},x_{2},x_{3})+q_{1}(x_{0},x_{1},x_{2},x_{3})x_{4}+l_{1}(x_{0},x_{1})x_{4}^2+\{ q_{2}(x_{0},x_{1},x_{2},x_{3})+\alpha_{1}x_{0}x_{4}  \} x_{5} +\{ q_{3}(x_{0},x_{1})+x_{0}l_{2}(x_{2},x_{3})  \} x_{6}$

\item $f_{13}=c(x_{0},x_{1},x_{2},x_{3})+\{ q_{1}(x_{0},x_{1},x_{2})+l_{1}(x_{0},x_{1},x_{2})x_{3}  \} x_{4} +\alpha_{1}x_{0}x_{4}^2+ \{ q_{2}(x_{0},x_{1},x_{2})+l_{2}(x_{0},x_{1},x_{2})x_{3}+\alpha_{2}x_{0}x_{4}  \} x_{5}+\alpha_{3}x_{0}x_{5}^2+\{ q_{3}(x_{0},x_{1},x_{2})+\alpha_{4}x_{0}x_{3}  \} x_{6}$

\item $f_{14}=\alpha x_{2}^{3}+x_{0}q_{1}(x_{0},\dots,x_{6})+x_{1}q_{2}(x_{1},\dots,x_{6})$

\item $f_{15}=c(x_{0},x_{1},x_{2},x_{3},x_{4})+\{ x_{0}l_{1}(x_{0},x_{1},x_{2},x_{3},x_{4})+q_{1}(x_{1},x_{2}) \} x_{5}+\{ x_{0}l_{2}(x_{0},x_{1},x_{2},x_{3},x_{4})+q_{2}(x_{1},x_{2}) \} x_{6}$

\item $f_{16}=c(x_{0},x_{1},x_{2},x_{3},x_{4})+\{ q_{1}(x_{0},x_{1})+l_{1}(x_{0},x_{1})x_{2}+l_{2}(x_{0},x_{1})x_{3}+l_{3}(x_{0},x_{1})x_{4} \} x_{5}+\alpha_{1}x_{0}x_{5}^2+\{ q_{2}(x_{0},x_{1})+x_{0}l_{4}(x_{2},x_{3},x_{4}) \} x_{6}$

\item $f_{17}=c(x_{0},x_{1},x_{2},x_{3},x_{4})+\{ q_{1}(x_{0},x_{1},x_{2})+l_{1}(x_{0},x_{1},x_{2})l_{2}(x_{3},x_{4}) \} x_{5}+q_{2}(x_{0},x_{1},x_{2}) x_{6}$

\item $f_{18}=c(x_{0},x_{1},x_{2},x_{3},x_{4})+\{ x_{0}l_{1}(x_{0},x_{1},x_{2},x_{3},x_{4})+x_{1}l_{2}(x_{0},x_{1},x_{2},x_{3},x_{4}) \} x_{5}+\{ x_{0}l_{3}(x_{0},x_{1},x_{2},x_{3},x_{4})+x_{1}l_{4}(x_{0},x_{1},x_{2},x_{3},x_{4}) \} x_{6}$

\item $f_{19}=c(x_{0},x_{1},x_{2},x_{3})+\{ q_{1}(x_{0},x_{1},x_{2})+l_{1}(x_{0},x_{1},x_{2})x_{3} \} x_{4}+l_{2}(x_{0},x_{1},x_{2})x_{4}^2+ \{ q_{2}(x_{0},x_{1},x_{2})+l_{3}(x_{0},x_{1},x_{2})x_{3}+l_{4}(x_{0},x_{1},x_{2})x_{4} \} x_{5} + l_{5}(x_{0},x_{1},x_{2})x_{5}^2 + q_{3}(x_{0},x_{1},x_{2}) x_{6}$

\item $f_{20}=c(x_{0},x_{1},x_{2},x_{3},x_{4})+q_{1}(x_{0},x_{1},x_{2},x_{3})x_{5}+q_{2}(x_{0},x_{1},x_{2},x_{3})x_{6}$

\item $f_{21}=c(x_{0},x_{1},x_{2},x_{3},x_{4},x_{5})+q(x_{0},x_{1})x_{6}$

\item $f_{22}=c(x_{0},x_{1},x_{2},x_{3},x_{4},x_{5})+x_{0}l(x_{0},x_{1},x_{2},x_{3},x_{4},x_{5})x_{6}$

\end{itemize}
\end{thm}

For an element $\sigma$ in $\SL(7)$ and $\mathbb{J} \subseteq I$, we set $\mathbb{J}^{\sigma}=\cup_{f} \mathrm{Supp}(f^{\sigma})$, where $f$ runs through all polynomials with $\mathrm{Supp}(f) \subseteq \mathbb{J}$.

\begin{definition}
We denote 
$$I(\mathbf{r}_{k})_{\ge 0} \subseteq I(\mathbf{r}_{l})_{\ge 0} \ \mathrm{mod} \ \SL(7)$$
 when there exists $\sigma \in \SL(7)$ such that $I(\mathbf{r}_{k})_{\ge 0}^{\sigma} \subseteq I(\mathbf{r}_{l})_{\ge 0}$ and say that $I(\mathbf{r}_{k})_{\ge 0}$ is included in $I(\mathbf{r}_{l})_{\ge 0}$ modulo $\SL(7)$.
\end{definition}

We construct a smaller subset $\mathcal{M}'$ of $\mathcal{M}$ such that $(1)$ any element $I(\mathbf{r}_{k})_{\ge 0}$ in $\mathcal{M}$ is included in some element $I(\mathbf{r}_{l})_{\ge 0}$  in $\mathcal{M}'$ mod $\SL(7)$, $(2)$ any element $I(\mathbf{r}_{k})_{\ge 0}$ in $\mathcal{M}'$ is not included in any other $I(\mathbf{r}_{l})_{\ge 0}$ in $\mathcal{M}'$ mod $\SL(7)$ $(1 \le l \le 22)$

\begin{proposition}
There are two relations
\begin{itemize}
\item $I(\mathbf{r}_{21})_{\ge 0} \subseteq I(\mathbf{r}_{22})_{\ge 0} \ \mathrm{mod} \ \SL(7)$
\item $I(\mathbf{r}_{22})_{\ge 0} \subseteq I(\mathbf{r}_{21})_{\ge 0} \ \mathrm{mod} \ \SL(7)$.
\end{itemize}
\end{proposition}

\begin{proof}
\begin{itemize}
\item[] $f_{22}=c(x_{0},\cdots,x_{5})+x_{0}l(x_{0},\cdots,x_{5})x_{6}$
\item[] $\equiv c(x_{0},\cdots,x_{5})+x_{0}l(x_{0},x_{1})x_{6}$
\item[] $\equiv c(x_{0},\cdots,x_{5})+q(x_{0},x_{1})x_{6}$
\item[] $=f_{21}$
\end{itemize}
Here $\equiv$ means equality after an $\SL(7)$ change of coordinates.
\end{proof}

From this proposition, we can remove $f_{22}$ from the list. Thus, we obtain a list of $21$ types of cubic fivefolds.

\begin{thm}\label{non-iclusion}
For any $1\le k,l\le 21$ with $k\ne l$, there is no element $g\in \SL(7)$ such that
\[
g\cdot I(\mathbf r_k)_{\ge 0}\subseteq I(\mathbf r_l)_{\ge 0}.
\]
\end{thm}

\begin{proof}
We postpone the proof of this theorem to Section \ref{sec:non-inclusion}.
\end{proof}

\begin{thm}\label{21-list}
The strictly semistable locus
$$\PP(W)^{ss}\setminus \PP(W)^s$$ has $21$ irreducible components, represented by
$f_1,\dots,f_{21}$.
\end{thm}

\section{Closed orbits of 21 families}\label{sec:closed-orbit}
In this section, we find the closed orbits in the 21 families of strictly semistable cubic fivefolds. We define $\lambda_{k}(t) \colon \mathbb{G}_{m} \to \SL(7)$ as 1-PSs corresponding to $\mathbf{r}_{k}$.

\noindent\textit{We shall use the following convex--geometric criterion repeatedly in this section, so we record it at the outset.}

For the following theorem, see \cite{PV94}.
\begin{thm}[Convex-hull criterion]\label{thm:convex-hull}
Let $T$ be an algebraic torus acting linearly on a finite-dimensional vector space, $V$, and let $v\in V$. Subsequently, the following conditions are equivalent:
\begin{enumerate}
  \item the $T$-orbit $T\cdot v$ is closed in $V$;
  \item $0$ is an interior point of the convex-hull of $\mathrm{Supp}(v)$ in $X(T)_\mathbb{R}$.
\end{enumerate}
Here, $\mathrm{Supp}(v)$ denotes the set of $T$-weights that occur in $v$, and $X(T)_\mathbb{R}$ is the real vector space spanned by the character lattice $X(T)$.
\end{thm}

The following series of definitions and theorems will be extremely useful for determining polystability in cases where the centralizer is not a torus and the convex-hull criterion cannot be applied. Because they will be used repeatedly from this point on, we state them here beforehand.

\begin{notation}
Let $G$ be a reductive algebraic group, and let $X$ be an affine variety equipped with an algebraic $G$-action.
\begin{itemize}
\item[$(1)$] We denote by $Y(G)$ the set of one-parameter subgroups of $G$.
\item[$(2)$] When $x \in X$, we put $\Lambda_{x}= \{ \lambda (t) \in Y(G) \colon \lim_{t \to0} \lambda (t) \cdot x \text{ exists} \}$.
\item[$(3)$] When $\lambda \in Y(G)$, we define $P(\lambda)=\{ g \in G \colon \lim_{t \to 0} \lambda(t) \cdot g \cdot \lambda(t)^{-1} \text{ exists}  \}$. It is a parabolic subgroup of G.
\end{itemize}
\end{notation}

\begin{definition}\label{def:symmetric}
We say that a subset $\Lambda \subset Y(G)$ is symmetric if given any $\lambda \in \Lambda$, there is another 1-PS $\lambda^{\prime} \in \Lambda$ such that $P(\lambda) \cap P(\lambda^{\prime})$ is a Levi subgroup of both $P(\lambda)$ and $P(\lambda^{\prime})$.

\end{definition}

We need the following theorem (Theorem~1.1 of \cite{CF12}):

\begin{thm}(Casimiro--Florentino criterion)\label{thm:CF}
Let $G$ be a reductive algebraic group and $X$ be an affine $G$-variety. Then, a point $x \in X$ is polystable if and only if $\Lambda_{x}$ is symmetric.
\end{thm}

\begin{lemma}[rank 1 symmetry for Casimiro--Florentino]\label{lem:rank1-symmetry}
Let $G$ be a reductive algebraic group acting on an affine $G$-variety $X$, and let $x\in X$.
Assume that there exists a one-parameter subgroup $\mu\in Y(G)$ such that
\[
\Lambda_x \;=\; \{\mu_k \mid k\in\mathbb{Z}\} \qquad
\mu_k(t):=\mu(t^k).
\]
Then $\Lambda_x$ is symmetric in the sense of Definition~\ref{def:symmetric}.
Consequently, by the Casimiro--Florentino criterion (Theorem~\ref{thm:CF}), the point $x$ is polystable.
\end{lemma}

\begin{proof}
Fix $k\neq 0$.
It is standard that $P(\mu_k)$ and $P(\mu_{-k})$ are opposite parabolic subgroups and that
\[
P(\mu_k)\cap P(\mu_{-k})
\;=\; C_G\!\bigl(\mu(\mathbb{G}_m)\bigr)
\]
which is a Levi subgroup of both $P(\mu_k)$ and $P(\mu_{-k})$.
The inclusion \(C_{G}\bigl(\mu(\mathbb{G}_m)\bigr)\subset P(\mu_{k})\cap P(\mu_{-k})\) is immediate: if \(g\) centralizes \(\mu(\mathbb{G}_m)\), then \(\mu_{k}(t)\,g\,\mu_{k}(t)^{-1}=g\) for all \(t\in\mathbb{G}_m\), so both limits as \(t\to 0\) and \(t\to\infty\) exist. For the reverse inclusion, take \(g\in P(\mu_{k})\cap P(\mu_{-k})\) and consider
\[
\phi:\mathbb{G}_m\longrightarrow G,\qquad \phi(t):=\mu_{k}(t)\,g\,\mu_{k}(t)^{-1}.
\]
By assumption, the limits \(\lim_{t\to 0}\phi(t)\) and \(\lim_{t\to\infty}\phi(t)\) both exist (the latter because \(g\in P(\mu_{-k})\) and \(\phi(1/s)=\mu_{-k}(s)\,g\,\mu_{-k}(s)^{-1}\) for \(s\to 0\)). Hence, \(\phi\) extends to a morphism \(\tilde\phi:\mathbb{P}^1\to G\). As \(G\) is affine, \(\tilde\phi\) is constant; in particular,
\[
\mu_{k}(t)\,g\,\mu_{k}(t)^{-1}=\phi(t)=\phi(1)=g \quad\text{for all }t\in\mathbb{G}_m.
\]
Thus, \(g\) centralizes \(\mu_{k}(\mathbb{G}_m)=\mu(\mathbb{G}_m)\), i.e. \(g\in C_{G}\bigl(\mu(\mathbb{G}_m)\bigr)\). This proves \(P(\mu_{k})\cap P(\mu_{-k})=C_{G}\bigl(\mu(\mathbb{G}_m)\bigr)\).

Based on the hypothesis $\mu_k,\mu_{-k}\in\Lambda_x$, so the requirement in Definition~\ref{def:symmetric} is satisfied for every $\lambda=\mu_k$.
Hence, $\Lambda_x$ is symmetric.
The last assertion follows from Theorem~\ref{thm:CF} (Casimiro--Florentino).
\end{proof}

\begin{secfourconvention}\label{convention}
Throughout this section, we fix the following setup and notation.

(1) \textbf{Coordinates and the maximal torus.}
We work on $W=\Sym^3\C^7$ with homogeneous coordinates $(x_0,\dots,x_6)$.
We fix the diagonal maximal torus $T\subset SL(7)$ acting by
$\diag(\mu_0,\dots,\mu_6)$ with $\prod_i \mu_i=1$.
A one-parameter subgroup (1-PS) is written as
\[
\lambda(t)=\diag\!\bigl(t^{a_0},\dots,t^{a_6}\bigr),\qquad \sum_{i=0}^6 a_i=0 .
\]

(2) \textbf{1-PS limits and centralizer reduction.}
For each $k\in\{1,\dots,21\}$ let $\lambda_k$ correspond to $r_k$ in Proposition~2.3, 
and let $f_k$ be the generic member from Theorem~3.2.
We set the 1-PS limit
\[
\phi_k:=\lim_{t\to0}\lambda_k(t)\cdot f_k .
\]
Put $H:=\lambda_k(\G_m)$ and write $C_G(H)$ for the centralizer in $G=SL(7)$.
Let $W^H\subset W$ be the $H$-fixed subspace. 
By Luna's centralizer reduction \cite{Lun75}, closedness of $SL(7)\cdot \phi_k$ in $ \mathbb{P}(W)^{ss}$ 
is equivalent to closedness of $C_G(H)\cdot \phi_k$ in $W^H$.

(3) \textbf{Two criteria for polystability.}
If $C_G(H)$ is a torus, we certify closedness by the \emph{convex-hull criterion}
(Theorem~4.1).
If $C_G(H)$ is non-toric, we use the Casimiro--Florentino criterion
(Theorem~4.4) as follows:
for $\lambda\in Y\!\left(C_G(H)\right)$ we write 
$\wt_\lambda(x_i)$ for the $\lambda$-weight on $x_i$, and
$w(m)$ for the induced weight of a monomial $m$.
We denote by $S$ the linear constraint coming from $\det=1$ on $C_G(H)$ 
(the ``trace'' of block weights), so $S=0$ for every $\lambda$.
We then exhibit a positive linear identity
\[
\sum_j c_j\, w(m_j) \;=\; C\cdot S \qquad(c_j>0,\; C>0).
\]
If $\lambda\in \Lambda_{\phi_k}$, then every $w(m_j)\ge0$ and $S=0$; hence, all 
$w(m_j)=0$; solving yields a symmetric $1$-PS $\mu_k$, so 
$\Lambda_{\phi_k}$ is symmetric and $\phi_k$ is polystable.

(4) \textbf{Normal forms and coefficient normalizations.}
Passing to ``normal form,'' we are allowed to:
(i) multiply by a nonzero scalar;
(ii) act by $C_G(H)/H$ (e.g.,\ block $GL(2)$, $GL(3)$ actions) to diagonalize
blocks; and
(iii) use diagonal elements of $T$ (with $\prod \mu_i=1$)
to normalize nonzero coefficients to $1$.
Parameters $(\alpha,\rho,\sigma,\dots)$ record the residual moduli.

(5) \textbf{Dimension count.}
Component dimensions are computed as
\[
\dim(W^H)-\dim_{\mathrm{eff}}\!\bigl(C_G(H)\bigr)-1,
\]
where $\dim_{\mathrm{eff}}$ is the dimension of the effective $C_G(H)$-action
on $W^H$ (central tori acting trivially are subtracted). 
We state explicitly when a central factor acts trivially.

(6) \textbf{Weights and symbols.}
We freely reuse symbols $a_i$ for nonzero coefficients of $\phi_k$ prior to normalization.
For 1-PS families obtained in the CF-check, we write $\mu_k(t)$.
All such conventions are in force throughout §4.
\end{secfourconvention}

Let us now determine, for each $k=1,2,\dots,21$, a polynomial whose $\SL(7)$-orbit is closed. The procedure is uniform across all cases. First, we take a $1$-PS limit to produce a candidate $\phi_k$ for a polystable point. Next, to apply Luna's criterion, we take the stabilizer $H\subset G=\SL(7)$ of $\phi_k$; we choose $H$ as large as possible so that its centralizer $C_G(H)$ is as small as possible, which simplifies the closedness check. Write $W^H$ for the $H$-fixed locus in the ambient representation $W=\mathrm{Sym}^3\mathbb{C}^7$. If $C_G(H)$ is a torus, we apply the convex-hull criterion (Theorem\ref{thm:convex-hull}) to show that $C_G(H)\cdot\phi_k$ is closed in $W^H$; if $C_G(H)$ is not a torus, we instead apply the Casimiro--Florentino criterion (Theorem \ref{thm:CF}) to obtain the same conclusion. 
 In either case, $C_G(H)\cdot\phi_k$ is closed in $W^H$; hence, by Luna's criterion, the orbit $\SL(7)\cdot\phi_k$ is closed in $\mathbb{P}(\mathrm{Sym}^3\mathbb{C}^7)^{ss}$. 
 Lastly, by determining a normal form of $\phi_k$ under the action of $C_G(H)/H$, we fix the dimension of the corresponding component of the moduli space for that $k$.

\subsection{Case $k=1$}

\paragraph{1-PS limit.}
Set
\[
\lambda_{1}(t)=\mathrm{diag}\!\bigl(t^{8},\,t^{3},\,t^{2},\,t^{-1},\,t^{-2},\,t^{-4},\,t^{-6}\bigr),
\quad t\in \mathbb{G}_m .
\]
For a generic $f_{1}$ as in Section~3, the $1$-PS limit is
\[
\phi_{1}:=\lim_{t\to 0}\lambda_{1}(t)\cdot f_{1}
= a_{1}\,x_{2}x_{3}^{2}
+ a_{2}\,x_{1}x_{3}x_{4}
+ a_{3}\,x_{2}^{2}x_{5}
+ a_{4}\,x_{0}x_{5}^{2}
+ a_{5}\,x_{1}^{2}x_{6}
+ a_{6}\,x_{0}x_{4}x_{6}.
\]

\paragraph{$H$ and $C_{G}(H)$.}
Let $H=\lambda_{1}(\mathbb{G}_m)$.
The diagonal weights on $\langle x_{0},\dots,x_{6}\rangle$ are pairwise distinct; hence,
$C_{G}(H)=T$ (the maximal diagonal torus).
Each monomial of $\phi_{1}$ has $H$-weight $0$, so $\phi_{1}\in W^{H}$.

\paragraph{Polystability (Luna + convex-hull).}
By Luna's criterion (see \cite{Lun75}), closedness of the $SL(7)$-orbit of $\phi_{1}$ is equivalent
to closedness of the $T$-orbit in $W^{H}$.
By the convex-hull criterion (Theorem~4.1), it suffices to check that $0$ is an interior
point of $\mathrm{Conv}(\mathrm{Supp}(\phi_{1}))\subset X(T)_{\mathbb{R}}\simeq
\mathbb{R}^{7}/\mathbb{R}(1,\dots,1)$.
This holds because the exponent vectors of the six monomials satisfy
\begin{equation}\label{eq:case1-hull}
\begin{aligned}
&2(0,0,1,2,0,0,0)+2(0,1,0,1,1,0,0)+2(0,0,2,0,0,1,0)\\
&\qquad +2(1,0,0,0,0,2,0)+2(0,2,0,0,0,0,1)+4(1,0,0,0,1,0,1)
=6(1,1,1,1,1,1,1),
\end{aligned}
\end{equation}
Hence, $\phi_{1}$ is polystable and $SL(7)\cdot \phi_{1}$ is closed.

\paragraph{Normal form and component dimension.}
Let $\mathrm{diag}(\mu_{0},\dots,\mu_{6})\in T$ with $\prod_{i=0}^{6}\mu_{i}=1$.
Along with the overall projective scaling, this acts on the six coefficients of
$\phi_{1}$ via the characters determined by the exponent vectors in
\eqref{eq:case1-hull}; we may normalize all six coefficients to $1$ simultaneously.
Thus, a normal form is
\[
\nf_{1}
= x_{2}x_{3}^{2}
+ x_{1}x_{3}x_{4}
+ x_{2}^{2}x_{5}
+ x_{0}x_{5}^{2}
+ x_{1}^{2}x_{6}
+ x_{0}x_{4}x_{6}.
\]

\paragraph{Residual $T$-stabilizer.}
Recall that $H=\lambda_1(\mathbb G_m)$ and $C_G(H)=T$, and that every monomial of
$\varphi_1$ has $H$-weight $0$, so $H$ acts trivially on $W^H$ and the effective group
on $\mathbb P(W^H)$ is $T/H$.
Let $\tau=\diag(\mu_0,\ldots,\mu_6)\in T$ with $\prod_{i=0}^6\mu_i=1$.
For the normal form
\[
\nf_1
= x_2x_3^2+x_1x_3x_4+x_2^2x_5+x_0x_5^2+x_1^2x_6+x_0x_4x_6,
\]
the induced $T$-characters on the six monomials are
\begin{align*}
\chi_1&=\mu_2\mu_3^2 &&(x_2x_3^2), &
\chi_2&=\mu_1\mu_3\mu_4 &&(x_1x_3x_4),\\
\chi_3&=\mu_2^2\mu_5 &&(x_2^2x_5), &
\chi_4&=\mu_0\mu_5^2 &&(x_0x_5^2),\\
\chi_5&=\mu_1^2\mu_6 &&(x_1^2x_6), &
\chi_6&=\mu_0\mu_4\mu_6 &&(x_0x_4x_6).
\end{align*}
Since we work projectively, $\tau$ stabilizes $[\nf_1]\in\mathbb P(W^H)$
if and only if $\tau\cdot\nf_1=c\,\nf_1$ for some
$c\in\mathbb C^\times$, equivalently
\[
\chi_1=\chi_2=\chi_3=\chi_4=\chi_5=\chi_6.
\]
Solving these equations (together with $\prod\mu_i=1$) yields $\mu_4^7=\mu_3^{14}$,
hence $\mu_4=\mu_3^2\xi$ with $\xi^7=1$. Writing $t:=\mu_3^{-1}$, we can factor
\[
\tau=\lambda_1(t)\cdot \delta_\xi,
\qquad
\delta_\xi:=\diag(\xi^{-9},\xi^{-4},\xi^{-3},1,\xi,\xi^3,\xi^5).
\]
Therefore
\[
\mathrm{Stab}_T\!\bigl([\nf_1]\bigr)
=H\cdot\{\delta_\xi:\ \xi^7=1\}\ \cong\ H\times\mu_7,
\]
so the induced stabilizer in the effective group $T/H$ is the finite group $\mu_7$.
In particular, after normalizing all coefficients to $1$, no continuous torus symmetry
remains; hence the corresponding boundary component is zero-dimensional.

\subsection{Case $k=2$}

\paragraph{1-PS limit.}
Set
\[
\lambda_{2}(t)
=
\operatorname{diag}\!\bigl(t^{6},\,t^{4},\,t,\,t^{-1},\,t^{-2},\,t^{-3},\,t^{-5}\bigr),
\qquad
t\in\mathbb{G}_{m}.
\]
For a generic $f_{2}$ as in Section~3, the $1$-PS limit is
\[
\phi_{2}
:=
\lim_{t\to 0}\lambda_{2}(t)\!\cdot\! f_{2}
=
a_{1}x_{2}^{2}x_{4}
+a_{2}x_{1}x_{4}^{2}
+a_{3}x_{1}x_{3}x_{5}
+a_{4}x_{0}x_{5}^{2}
+a_{5}x_{1}x_{2}x_{6}
+a_{6}x_{0}x_{3}x_{6}.
\]

\paragraph{$H$ and $C_{G}(H)$.}
Let $H=\lambda_{2}(\mathbb{G}_{m})$.
The diagonal weights on $\langle x_{0},\dots,x_{6}\rangle$ are pairwise distinct; hence,
$C_{G}(H)=T$ (the maximal diagonal torus).
Each monomial of $\phi_{2}$ has $H$-weight $0$, so $\phi_{2}\in W^{H}$.

\paragraph{Polystability (Luna + convex-hull).}
By Luna's criterion, closedness of the $SL(7)$-orbit of $\phi_{2}$ is equivalent to closedness of the $T$-orbit in the $H$-fixed subspace.
By the convex-hull criterion (Theorem~4.1), it suffices to check that $0$ lies in the interior of
$\operatorname{Conv}(\operatorname{Supp}(\phi_{2}))
\subset
X(T)_{\mathbb{R}}
\simeq
\mathbb{R}^{7}/\mathbb{R}(1,\ldots,1)$.
This holds because
\begin{equation}\label{case2}
\begin{aligned}
&2(0,0,2,0,1,0,0)
+2(0,1,0,0,2,0,0)
+2(0,1,0,1,0,1,0)
\\
&\quad
+2(1,0,0,0,0,2,0)
+2(0,1,1,0,0,0,1)
+4(1,0,0,1,0,0,1)
\\
&=6(1,1,1,1,1,1,1),
\end{aligned}
\end{equation}
written in terms of the exponent vectors of the six monomials of $\phi_{2}$.
Hence, $\phi_{2}$ is polystable and the orbit $SL(7)\cdot\phi_{2}$ is closed.

\paragraph{Normal form and component dimension.}
A diagonal scaling $\operatorname{diag}(\mu_{0},\dots,\mu_{6})\in T$ with $\prod\mu_{i}=1$, together with an overall scalar, acts on the six coefficients via the characters determined by the exponent vectors in \eqref{case2}.
Thus we may normalize all six coefficients to $1$ simultaneously.
Therefore, a normal form is
\[
\nf_{2}
=
x_{2}^{2}x_{4}
+ x_{1}x_{4}^{2}
+ x_{1}x_{3}x_{5}
+ x_{0}x_{5}^{2}
+ x_{1}x_{2}x_{6}
+ x_{0}x_{3}x_{6}.
\]
The residual $T$-stabilizer is finite; hence, the corresponding component $\Phi_{2}$ of the moduli is zero-dimensional.

\subsection{Case $k=3$}

\paragraph{1-PS limit.}
Set
\[
  \lambda_3(t)
  = \operatorname{diag}\!\bigl(t^{4},\, t^{2},\, t,\, t^{-1},\, t^{-1},\, t^{-2},\, t^{-3}\bigr),
  \qquad t\in \mathbb{G}_m .
\]
For a generic $f_3$ as in Section~3, the $1$-PS limit is
\begin{align*}
  \phi_3
  := \lim_{t\to 0}\lambda_3(t)\cdot f_3
   ={}& a_1\, x_1 x_3^{2} + a_2\, x_1 x_3 x_4 + a_3\, x_1 x_4^{2}
\\ &\quad
      + a_4\, x_2^{2} x_5 + a_5\, x_0 x_5^{2}
      + a_6\, x_1 x_2 x_6 + a_7\, x_0 x_3 x_6 + a_8\, x_0 x_4 x_6 .
\end{align*}

\paragraph{$H$ and $C_G(H)$.}
Let $H=\lambda_3(\mathbb{G}_m)$. The multiplicities of the diagonal weights on
$\langle x_0,\dots,x_6\rangle$ are $1$ on $x_0,x_1,x_2,x_5,x_6$ and $2$ on
$\langle x_3,x_4\rangle$; hence,
\[
  C_G(H)
  = \left\{
    \begin{aligned}
      &\operatorname{diag}(\alpha_0,\alpha_1,\alpha_2) \oplus A \oplus \operatorname{diag}(\alpha_5,\alpha_6) :\\[-1pt]
      &\quad \alpha_i\in \mathbb{G}_m,\ A\in \mathrm{GL}(2),\
            \alpha_0\alpha_1\alpha_2\,\det(A)\,\alpha_5\alpha_6=1
    \end{aligned}
  \right\}
  \cong \mathrm{SL}(2)\times \mathbb{G}_m^{5}.
\]
Each monomial of $\phi_3$ has $H$-weight $0$, so $\phi_3$ is $H$-fixed.

\paragraph{Polystability (Luna + Casimiro--Florentino).}
By Luna's slice/centralizer reduction, the closedness of the $\mathrm{SL}(7)$-orbit
of $\phi_3$ is equivalent to polystability for the $C_G(H)$-action on the
$H$-fixed subspace. After conjugating inside the $\mathrm{SL}(2)$-block, any
$\lambda \in Y(C_G(H))$ may be chosen with weights
\[
  \mathrm{wt}(x_0,\dots,x_6)=(a_0,\,a_1,\,a_2,\,c+n,\,c-n,\,a_5,\,a_6),
  \qquad
  S:=a_0+a_1+a_2+2c+a_5+a_6=0
\]
(as in Convention~\ref{convention}.
Let $w_i$ be the $\lambda$-weight of the $i$-th monomial of $\phi_3$. Then
\begin{align*}
  w_1&=a_1+2(c+n), & w_2&=a_1+2c, & w_3&=a_1+2(c-n), \notag\\
  w_4&=2a_2+a_5,   & w_5&=a_0+2a_5, & w_6&=a_1+a_2+a_6, \label{eq:k3-weights}\\
  w_7&=a_0+(c+n)+a_6, & w_8&=a_0+(c-n)+a_6. \notag
\end{align*}
A direct computation yields the positive linear identity
\begin{equation}\label{eq:k3-identity}
  w_1+w_2+2 w_3+2 w_4+2 w_5+2 w_6+3 w_7+w_8
  \;=\; 6 S .
\end{equation}
If $\lambda\in \Lambda_{\phi_3}$, then all $w_i\ge 0$ and $S=0$; by
\eqref{eq:k3-identity} they must all vanish. Solving gives
\[
  n=0,\quad a_1=-2c,\quad a_5=-2a_2,\quad a_0=4a_2,\quad a_6=2c-a_2,\quad a_2=-c .
\]
Thus, with $k\in\mathbb{Z}$,
\[
  \mu_k(t):=\operatorname{diag}\!\bigl(t^{4k},\,t^{2k},\,t^{k},\,t^{-k},\,t^{-k},\,t^{-2k},\,t^{-3k}\bigr),
  \qquad
  \Lambda_{\phi_3}=\{\mu_k\mid k\in\mathbb{Z}\}.
\]
Therefore $\Lambda_{\phi_3}$ is symmetric, and by the Casimiro--Florentino
criterion $\phi_3$ is polystable; in particular, $\mathrm{SL}(7)\cdot \phi_3$ is
closed.

\paragraph{Normal form and component dimension.}
On the $H$-fixed slice, the eight weight-zero monomials are
\[
x_1x_3^2,\ x_1x_3x_4,\ x_1x_4^2,\ x_2^2x_5,\ x_0x_5^2,\ x_1x_2x_6,\ x_0x_3x_6,\ x_0x_4x_6,
\]
Therefore, $W^H$ denotes their span. 
i.e. 
\[
W^{H} = \underbrace{\langle x_{1} \rangle \otimes \mathrm{Sym}^{2}\langle x_{3},x_{4} \rangle}_{\mathrm{(I)}}
\oplus \underbrace{\mathrm{Sym}^{2}\langle x_{2} \rangle \otimes x_{5}}_{\mathrm{(II)}}
 \oplus \underbrace{x_{0} \otimes \mathrm{Sym}^{2}\langle x_{5} \rangle}_{\mathrm{(III)}}
  \oplus \underbrace{x_{0} \otimes \langle x_{3},x_{4} \rangle \otimes x_{6}}_{\mathrm{(IV)}}
  \oplus \underbrace{x_{1} \otimes \langle x_{2} \rangle \otimes x_{6}}_{\mathrm{(V)}}.
\]
The centralizer is
\[
C_G(H)\cong \mathrm{SL}(2)\times \mathbb G_m^5,
\]
acting by $\mathrm{SL}(2)$ on $\langle x_3,x_4\rangle$ and by a diagonal torus on
$\langle x_0,x_1,x_2,x_5,x_6\rangle$ (subject to the product-one condition).  

\paragraph{(I) binary quadratic on $\langle x_3,x_4\rangle$.}
Write the $x_1$-part as a binary quadratic
$Q=a_1x_3^2+a_2x_3x_4+a_3x_4^2\in\Sym^2\langle x_3,x_4\rangle$.
As $\Delta(Q)$ is $\mathrm{SL}(2)$-invariant while rescaling $x_1$ scales $Q$ (and hence $\Delta$) homogeneously,
for a generic (nondegenerate) $Q$, there exists $A\in\mathrm{SL}(2)$ and a rescaling of $x_1$
such that $Q\sim x_3^2+x_3x_4+x_4^2$.
Thus, the block $\mathrm{(I)}$ is fixed to $x_1x_3^2+x_1x_3x_4+x_1x_4^2$.

\paragraph{(II)(III)(IV)(V) torus normalizations.} 
Let $T'=\{\diag(\mu_0,\mu_1,\mu_2)\oplus I_2\oplus\diag(\mu_5,\mu_6):\ \mu_0\mu_1\mu_2\mu_5\mu_6=1\}$. 
Under $T'$, the coefficients transform by the characters determined by exponent vectors: 
$x_2^2x_5$ by $\mu_2^2\mu_5$, $x_0x_5^2$ by $\mu_0\mu_5^2$, $x_1x_2x_6$ by $\mu_1\mu_2\mu_6$, while $x_0x_3x_6$ and $x_0x_4x_6$ both by $\mu_0\mu_6$. 
Using $\mu_0,\mu_1,\mu_2,\mu_5,\mu_6$ (together with an overall scalar), we set $x_2^2x_5$, $x_0x_5^2$, and $x_1x_2x_6$ to have coefficient $1$. 
More precisely, we are imposing the conditions
\[ 
\mu_{2}^{2}\mu_{5}=1,\qquad \mu_{0}\mu_{5}^{2}=1,\qquad \mu_{1}\mu_{2}\mu_{6}=1,
\] 
together with the product-one constraint \[ \mu_{0}\mu_{1}\mu_{2}\mu_{5}\mu_{6}=1. \] 
We can solve these equations explicitly: choose $\mu_2,\mu_6\in\mathbb C^\times$ freely, set
\[
\mu_5=\mu_2^{-2},\qquad \mu_0=\mu_5^{-2}=\mu_2^{4},\qquad \mu_1=(\mu_2\mu_6)^{-1}.
\]
Then the product-one constraint $\mu_0\mu_1\mu_2\mu_5\mu_6=1$ is automatically satisfied.
Consequently, we can normalize the coefficients of $x_2^2x_5$, $x_0x_5^2$, and $x_1x_2x_6$ to be all equal to $1$.

(IV)The remaining pair and the residual $\SL(2)$-symmetry.
The pair $(x_0x_3x_6,\;x_0x_4x_6)$ is scaled by the same torus character $\mu_0\mu_6$,
so the diagonal torus $T'$ alone does not change their ratio.
However, after fixing the nondegenerate binary quadratic
\[
Q \sim x_3^2+x_3x_4+x_4^2,
\]
the stabilizer $\mathrm{Stab}_{\SL(2)}(Q)$ is \emph{not} finite: it is a one-dimensional torus
(isomorphic to $\mathrm{SO}(2,\C)\cong \mathbb{G}_{m}$).
Its induced action on $\PP\langle x_3,x_4\rangle$ has two fixed points (the two roots of $Q$)
and a single open orbit.
Hence, for a generic element (i.e.\ when the linear form in block (IV) is not proportional to a root of $Q$),
we may act by $\mathrm{Stab}_{\mathrm{SL}(2)}(Q)$ to normalize that linear form to $x_3+x_4$.
Using the remaining diagonal torus (and projective rescaling), we normalize its coefficient to $1$.

Consequently, no continuous parameter remains, and a convenient normal form is
\[
\nf_3
= x_1x_3^2+x_1x_3x_4+x_1x_4^2+x_2^2x_5+x_0x_5^2+x_1x_2x_6+x_0x_3x_6+x_0x_4x_6.
\]

Finally, $\dim W^H=8$ and $\dim C_G(H)=8$, while $H\simeq\mathbb{G}_{m}$ acts trivially on $W^H$.
Thus the effective group dimension is $8-1=7$, and after projectivizing we obtain
\[
\dim(\Phi_3)=8-7-1=0.
\]

\subsection{Case $k=4$}

\paragraph{1-PS limit.}
\[
\lambda_{4}(t)
=
\mathrm{diag}\!\left(t^{3},\,t^{2},\,t,\,1,\,t^{-1},\,t^{-2},\,t^{-3}\right),
\quad
t\in\mathbb{G}_{m}.
\]
For a generic $f_{4}$ as in Section~3, the $1$-PS limit is
\[
\begin{aligned}
\phi_{4}
&:=\lim_{t\to 0}\lambda_{4}(t)\cdot f_{4}\\
&=
a_{1}\,x_{3}^{3}
+a_{2}\,x_{2}x_{3}x_{4}
+a_{3}\,x_{1}x_{4}^{2}
+a_{4}\,x_{2}^{2}x_{5}
+a_{5}\,x_{1}x_{3}x_{5}
+a_{6}\,x_{0}x_{4}x_{5}
+a_{7}\,x_{1}x_{2}x_{6}
+a_{8}\,x_{0}x_{3}x_{6}.
\end{aligned}
\]

\paragraph{$H$ and $C_{G}(H)$.}
Let $H=\lambda_{4}(\mathbb{G}_{m})$.
The diagonal weights on $\langle x_{0},\dots,x_{6}\rangle$ are pairwise distinct; hence
$C_{G}(H)=T$ (the maximal diagonal torus).
Each monomial of $\phi_{4}$ has $H$-weight $0$, so $\phi_{4}\in W^{H}$.

\paragraph{Polystability (Luna + convex-hull).}
By Luna's criterion, closedness of the $SL(7)$-orbit of $\phi_{4}$ is equivalent to
closedness of the $T$-orbit in $W^{H}$.
By the convex-hull criterion (Theorem~4.1), it suffices to check that $0$ is an interior
point of
\[
\mathrm{Conv}(\mathrm{Supp}(\phi_{4}))
\subset
X(T)_{\mathbb{R}}
\cong
\mathbb{R}^{7}/\mathbb{R}(1,\dots,1),
\]
which holds because the exponent vectors satisfy
\begin{equation}\label{case4}
\begin{aligned}
&(0,0,0,3,0,0,0)
+ (0,0,1,1,1,0,0)
+ (0,1,0,0,2,0,0)
+ (0,0,2,0,0,1,0)
\\
&\quad
+ 2(0,1,0,1,0,1,0)
+ 6(1,0,0,0,1,1,0)
+ 6(0,1,1,0,0,0,1)
\\
&\quad
+ 3(1,0,0,1,0,0,1)
= 9(1,1,1,1,1,1,1).
\end{aligned}
\end{equation}
Hence, $\phi_{4}$ is polystable and $SL(7)\cdot \phi_{4}$ is closed.

\paragraph{Normal form and component dimension.}
Since every monomial of $\phi_{4}$ has $H$-weight $0$, the subgroup $H\subset T$ acts \emph{trivially}
on $W^{H}$. Hence the $T$-action on $\mathbb{P}(W^{H})\cong\mathbb{P}^{7}$ factors through the
effective torus $T/H$, which has dimension $5$.

For a point of $\mathbb{P}(W^{H})$ with all eight coefficients nonzero, the stabilizer in $T/H$
is finite (equivalently, the eight exponent vectors have affine rank $5$), so a general
$(T/H)$-orbit has dimension $5$. Therefore the corresponding boundary component has dimension
\[
\dim \Phi_{4}=\dim \mathbb{P}(W^{H})-\dim(T/H)=7-5=2.
\]

Using the $5$ parameters of $T/H$ together with the overall projective scaling, we can normalize
\emph{six} of the eight coefficients to $1$, leaving two essential parameters. A convenient normal form is
\[
\nf_{4}(\alpha,\beta)
=
x_{3}^{3}
+x_{2}x_{3}x_{4}
+x_{1}x_{4}^{2}
+x_{2}^{2}x_{5}
+x_{1}x_{3}x_{5}
+x_{0}x_{4}x_{5}
+\alpha\,x_{1}x_{2}x_{6}
+\beta\,x_{0}x_{3}x_{6},
\qquad
(\alpha,\beta)\in(\mathbb{C}^{\times})^{2}.
\]
In particular, the corresponding component $\Phi_{4}$ of the moduli is two-dimensional.

\subsection{Case $k=5$}

\paragraph{1-PS limit.}
Set
\[
\lambda_{5}(t)=\operatorname{diag}\!\bigl(t^{4},\,t^{2},\,t,\,1,\,t^{-1},\,t^{-2},\,t^{-4}\bigr),\qquad t\in\mathbb{G}_{m}.
\]
For a generic $f_{5}$ as in Section~3, the $1$-PS limit is
\[
\phi_{5}:=\lim_{t\to 0}\lambda_{5}(t)\cdot f_{5}
= a_{1}x_{3}^{3}+a_{2}x_{2}x_{3}x_{4}+a_{3}x_{1}x_{4}^{2}+a_{4}x_{2}^{2}x_{5}
+a_{5}x_{1}x_{3}x_{5}+a_{6}x_{0}x_{5}^{2}+a_{7}x_{1}^{2}x_{6}+a_{8}x_{0}x_{3}x_{6}.
\]

\paragraph{$H$ and $C_{G}(H)$.}
Let $H=\lambda_{5}(\mathbb{G}_{m})$. The diagonal weights on $\langle x_{0},\ldots,x_{6}\rangle$ are pairwise distinct;
hence, $C_{G}(H)=T$ (the maximal diagonal torus). Each monomial of $\phi_{5}$ has $H$-weight $0$, so $\phi_{5}\in W^{H}$.

\paragraph{Polystability (Luna + convex-hull).}
By Luna's criterion, closedness of the $SL(7)$-orbit of $\phi_{5}$ is equivalent to closedness of the $T$-orbit
in the $H$-fixed subspace. By the convex-hull criterion (Theorem~4.1), it suffices to check that $0$ lies in the
interior of $\operatorname{Conv}(\operatorname{Supp}(\phi_{5}))\subset X(T)_{\mathbb{R}}\cong\mathbb{R}^{7}/\mathbb{R}(1,\ldots,1)$.
This holds because the exponent vectors of the eight monomials satisfy
\begin{equation}\label{eq:convex-phi5}
\begin{aligned}
& (0,0,0,3,0,0,0) + (0,0,1,1,1,0,0) \\
&\quad + 10(0,1,0,0,2,0,0) + 10(0,0,2,0,0,1,0) \\
&\quad + (0,1,0,1,0,1,0) + 5(1,0,0,0,0,2,0) \\
&\quad + 5(0,2,0,0,0,0,1) + 16(1,0,0,1,0,0,1) \\
&= 21(1,1,1,1,1,1,1),
\end{aligned}
\end{equation}
Hence, $\phi_{5}$ is polystable and $SL(7)\cdot\phi_{5}$ is closed.

\paragraph{Normal form and component dimension.}
Since every monomial of $\phi_{5}$ has $H$-weight $0$, the subgroup $H\subset T$ acts \emph{trivially}
on $W^{H}$. Hence the $T$-action on $\mathbb{P}(W^{H})\cong\mathbb{P}^{7}$ factors through the
effective torus $T/H$, which has dimension $5$.

For a point of $\mathbb{P}(W^{H})$ with all eight coefficients nonzero, the stabilizer in $T/H$
is finite (equivalently, the eight exponent vectors have affine rank $5$), so a general
$(T/H)$-orbit has dimension $5$. Therefore the corresponding boundary component has dimension
\[
\dim \Phi_{5}=\dim \mathbb{P}(W^{H})-\dim(T/H)=7-5=2.
\]

Using the $5$ parameters of $T/H$ together with the overall projective scaling, we can normalize
\emph{six} of the eight coefficients to $1$, leaving two essential parameters. A convenient normal form is
\[
\nf_{5}(\alpha,\beta)
=
x_{3}^{3}
+x_{2}x_{3}x_{4}
+x_{1}x_{4}^{2}
+x_{2}^{2}x_{5}
+x_{1}x_{3}x_{5}
+x_{0}x_{5}^{2}
+\alpha\,x_{1}^{2}x_{6}
+\beta\,x_{0}x_{3}x_{6},
\qquad
(\alpha,\beta)\in(\mathbb{C}^{\times})^{2}.
\]
In particular, the corresponding component $\Phi_{5}$ of the moduli is two-dimensional.

\subsection{Case $k=6$}

\paragraph{1-PS limit.}
Set
\[
\lambda_{6}(t)
=
\mathrm{diag}\!\left(t^{5},\,t^{3},\,t^{2},\,t,\,t^{-1},\,t^{-4},\,t^{-6}\right),
\quad
t\in\mathbb{G}_m .
\]
For a generic $f_{6}$ as in Section~3, the $1$-PS limit is
\[
\phi_{6}
:=
\lim_{t\to 0}\lambda_{6}(t)\cdot f_{6}
=
a_{1}\,x_{2}x_{4}^{2}
+a_{2}\,x_{2}^{2}x_{5}
+a_{3}\,x_{1}x_{3}x_{5}
+a_{4}\,x_{0}x_{4}x_{5}
+a_{5}\,x_{1}^{2}x_{6}
+a_{6}\,x_{0}x_{3}x_{6}.
\]

\paragraph{$H$ and $C_G(H)$.}
Let $H=\lambda_{6}(\mathbb{G}_m)$.
The diagonal weights on $\langle x_{0},\dots,x_{6}\rangle$ are pairwise distinct; hence
$C_G(H)=T$ (the maximal diagonal torus).
Each monomial of $\phi_{6}$ has $H$-weight $0$, so $\phi_{6}\in W^{H}$.

\paragraph{Polystability (Luna + convex-hull).}
Based on Luna's criterion, closedness of the $SL(7)$-orbit of $\phi_{6}$ is equivalent to
closedness of the $T$-orbit in the $H$-fixed subspace.
By Theorem~4.1, it suffices to check that $0$ is an interior point of
\[
\mathrm{Conv}(\mathrm{Supp}(\phi_{6}))
\subset
X(T)_{\mathbb{R}}
\cong
\mathbb{R}^{7}/\mathbb{R}(1,\dots,1),
\]
which holds because the exponent vectors of the six monomials of $\phi_{6}$ satisfy the
positive relation
\begin{equation}\label{eq:k6-hull}
\begin{aligned}
&(0,0,1,0,2,0,0)
+(0,0,2,0,0,1,0)
+(0,1,0,1,0,1,0)
\\
&\qquad
+(1,0,0,0,1,1,0)
+(0,2,0,0,0,0,1)
+2(1,0,0,1,0,0,1)
\\
&=3(1,1,1,1,1,1,1).
\end{aligned}
\end{equation}
Hence, $\phi_{6}$ is polystable and $SL(7)\cdot\phi_{6}$ is closed.

\paragraph{Normal form and component dimension.}
A diagonal scaling $\mathrm{diag}(\mu_{0},\cdots,\mu_{6}) \in T$ with $\Pi \mu_{i} = 1$, together with an overall scalar, acts on the six coefficients via the characters determined by the exponent vectors in \eqref{eq:k6-hull}.
Hence we may normalize all six coefficients to $1$ simultaneously. A normal form is therefore
\[
\nf_{6} = x_{2}x_{4}^{2}+x_{2}^{2}x_{5}+x_{1}x_{3}x_{5}+x_{0}x_{4}x_{5}+x_{1}^{2}x_{6}+x_{0}x_{3}x_{6}.
\]
The residual $T$-stabilizer is finite; hence, the corresponding component $\Phi_{6}$ of the moduli is zero-dimensional.

\subsection{Case $k=7$}

\paragraph{1-PS limit.}
Set
\[
\lambda_{7}(t)
=
\mathrm{diag}\!\left(t^{6},\,t^{4},\,t^{2},\,t,\,t^{-2},\,t^{-3},\,t^{-8}\right),
\qquad
t\in\mathbb{G}_m.
\]
For a generic $f_{7}$ as in Section~3, the $1$-PS limit is
\[
\phi_{7}
:=
\lim_{t\to0}\lambda_{7}(t)\cdot f_{7}
=
a_{1}\,x_{3}^{2}x_{4}
+a_{2}\,x_{1}x_{4}^{2}
+a_{3}\,x_{2}x_{3}x_{5}
+a_{4}\,x_{0}x_{5}^{2}
+a_{5}\,x_{1}^{2}x_{6}
+a_{6}\,x_{0}x_{2}x_{6}.
\]

\paragraph{$H$ and $C_{G}(H)$.}
Let $H=\lambda_{7}(\mathbb{G}_m)$.
The diagonal weights on $\langle x_{0},\ldots,x_{6}\rangle$ are pairwise distinct; hence
$C_{G}(H)=T$ (the maximal diagonal torus).
Each monomial of $\phi_{7}$ has $H$-weight $0$, so $\phi_{7}$ lies in the $H$-fixed subspace.

\paragraph{Polystability (Luna + convex-hull).}
By Luna's criterion, closedness of the $\mathrm{SL}(7)$-orbit of $\phi_{7}$ is equivalent to closedness
of the $T$-orbit in the $H$-fixed subspace.
By the convex-hull criterion (Theorem~4.1), it suffices to check that $0$ is an interior point of
\[
\mathrm{Conv}(\mathrm{Supp}(\phi_{7}))
\subset
X(T)_{\mathbb{R}}
\cong
\mathbb{R}^{7}/\mathbb{R}(1,\ldots,1),
\]
which holds because
\begin{equation}\label{case7}
\begin{aligned}
&(0,0,0,2,1,0,0)
+(0,1,0,0,2,0,0)
+(0,0,1,1,0,1,0)
\\
&\quad
+(1,0,0,0,0,2,0)
+(0,2,0,0,0,0,1)
+2(1,0,1,0,0,0,1)
\\
&=3(1,1,1,1,1,1,1),
\end{aligned}
\end{equation}
written in terms of the exponent vectors of the six monomials of $\phi_{7}$.
Hence, $\phi_{7}$ is polystable and $\mathrm{SL}(7)\cdot\phi_{7}$ is closed.

\paragraph{Normal form and component dimension.}
A diagonal scaling $\mathrm{diag}(\mu_{0},\ldots,\mu_{6})\in T$ with $\prod\mu_{i}=1$, together with an overall scalar,
acts on the six coefficients via the characters determined by the exponent vectors in \eqref{case7}.
Thus we may normalize all six coefficients to $1$ simultaneously.
Therefore, a normal form is
\[
\nf_{7}
=
x_{3}^{2}x_{4}
+ x_{1}x_{4}^{2}
+ x_{2}x_{3}x_{5}
+ x_{0}x_{5}^{2}
+ x_{1}^{2}x_{6}
+ x_{0}x_{2}x_{6}.
\]
The residual $T$-stabilizer is finite; hence, the corresponding component $\Phi_{7}$ of the moduli is zero-dimensional.

\subsection{Case $k=8$}

\paragraph{1-PS limit.}
Set
\[
\lambda_{8}(t)=\mathrm{diag}\!\left(t^{4},\,t,\,t,\,1,\,t^{-2},\,t^{-2},\,t^{-2}\right),\qquad t\in\mathbb{G}_m.
\]
For a generic $f_{8}$ as in Section~3, the $1$-PS limit is
\begin{align*}
\phi_8:=\lim_{t\to 0}\lambda_8(t)\cdot f_8
&=a_1 x_3^3
+a_2 x_1^2x_4+a_3 x_1x_2x_4+a_4 x_2^2x_4
+a_6 x_1^2x_5+a_7 x_1x_2x_5+a_8 x_2^2x_5\\
&\quad +a_{11}x_1^2x_6+a_{12}x_1x_2x_6+a_{13}x_2^2x_6
+a_5 x_0x_4^2+a_9 x_0x_4x_5+a_{14}x_0x_4x_6\\
&\quad +a_{10}x_0x_5^2+a_{15}x_0x_5x_6+a_{16}x_0x_6^2.
\end{align*}

\paragraph{$H$ and $C_G(H)$.}
Let $H=\lambda_8(\mathbb{G}_m)$.
The weights on $\langle x_0,\dots,x_6\rangle$ are $(4,1,1,0,-2,-2,-2)$ with multiplicities $(1,2,1,3)$; hence,
\[
C_G(H)
=\Bigl\{\mathrm{diag}(\alpha)\oplus A\oplus \mathrm{diag}(\beta)\oplus B:
\ \alpha,\beta\in\mathbb{G}_m,\ A\in \mathrm{GL}(2),\ B\in\mathrm{GL}(3),\
\alpha\beta\det(A)\det(B)=1\Bigr\}.
\]
Thus $C_G(H)\cong\bigl(\mathbb{G}_m\times \mathrm{GL}(2)\times \mathbb{G}_m\times \mathrm{GL}(3)\bigr)\cap \mathrm{SL}(7)$ and $\dim C_G(H)=14$.
Every monomial of $\phi_8$ has $H$-weight $0$; hence, $\phi_8\in W^H$.

\paragraph{Polystability (Luna + Casimiro--Florentino).}
Based on Luna's reduction, the closedness of the $\mathrm{SL}(7)$-orbit of $\phi_8$ is equivalent to polystability for the $C_G(H)$-action on $W^H$.
After conjugating inside the $\mathrm{GL}(2)$- and $\mathrm{GL}(3)$-blocks, any $\lambda\in Y(C_G(H))$ may be taken with
\[
\mathrm{wt}(x_0,\ldots,x_6)=(\alpha,\,s+u,\,s-u,\,\beta,\,\gamma+v_1,\,\gamma+v_2,\,\gamma+v_3),
\]
where $\alpha,\beta,s,u,\gamma,v_i\in\mathbb{Z}$, $v_1+v_2+v_3=0$, and
\[
S:=\alpha+\beta+2s+3\gamma=0
\]
is the $\mathrm{SL}$-constraint fixed in Convention~\ref{convention}.
Let $w(\cdot)$ be the weight of a monomial.
A direct computation yields the positive linear identity
\begin{align}
\sum_{j=4}^{6}\Bigl(w(x_1^2x_j)+w(x_2^2x_j)+2\,w(x_1x_2x_j)\Bigr)
+3 w(x_0x_4^2)+3 w(x_0x_5^2)+3 w(x_0x_6^2)\notag  \\
 +w(x_0x_4x_5)+w(x_0x_4x_6)+w(x_0x_5x_6)+4\,w(x_3^3)=12S.
\label{eq:k8-positive-identity}
\end{align}
If $\lambda\in\Lambda_{\phi_8}$, then all the above weights are $\ge 0$ and $S=0$; therefore, by
\eqref{eq:k8-positive-identity} they must all vanish.
Solving gives
\[
\beta=0,\quad v_1=v_2=v_3=0,\quad u=0,\quad 2s+\gamma=0,\quad \alpha+2\gamma=0.
\]
Putting $s=k\in\mathbb{Z}$, we obtain
\[
\mu_k(t)=\mathrm{diag}\!\left(t^{4k},\,t^{k},\,t^{k},\,1,\,t^{-2k},\,t^{-2k},\,t^{-2k}\right),\qquad
\Lambda_{\phi_8}=\{\mu_k\mid k\in\mathbb{Z}\},
\]
which is symmetric. Hence, by the Casimiro--Florentino criterion, $\phi_8$ is polystable.

\paragraph{Normal form and component dimension.}
On the $H$-fixed subspace
\[
W^{H}
=
\underbrace{\mathrm{Sym}^2\langle x_1,x_2\rangle \otimes\langle x_4,x_5,x_6\rangle}_{\mathrm{(I)}}
\ \oplus\ 
\underbrace{x_0\otimes \mathrm{Sym}^2\langle x_4,x_5,x_6\rangle}_{\mathrm{(II)}}
\ \oplus 
\underbrace{\langle x_3^3\rangle}_{\mathrm{(III)}},
\]
use $\mathrm{GL}(3)$ on $\langle x_4,x_5,x_6\rangle$ to diagonalize the quadratic, then the left/right
actions on $\mathrm{Sym}^2\langle x_1,x_2\rangle\otimes\langle x_4,x_5,x_6\rangle$ to diagonalize the $3\times 3$ block (SVD-type reduction under left $\mathrm{Sym}^2\mathrm{GL}(2)$ and right $\mathrm{O}(U,q)$, here $U= \langle x_{4},x_{5},x_{6} \rangle$ and $q=x_{4}^2+x_{5}^2+x_{6}^2$) and
normalize one diagonal entry to $1$; the remaining two appear as parameters $\rho,\sigma$.
A normal form is
\[
\nf_{8}(\rho,\sigma)
=
x_3^3+x_0x_4^2+x_0x_5^2+x_0x_6^2+x_1^2x_4+\rho\,x_1x_2x_5+\sigma\,x_2^2x_6,
\qquad (\rho,\sigma)\in(\mathbb{C}^{\times})^2.
\]
As $\dim W^{H}=16$ and the effective action has dimension $13$ (after projectivizing),
the closed component has dimension $16-13-1=2$.

\begin{remark}
Here $\mathrm{O}(U,q)$ denotes the orthogonal group of the quadratic space $(U,q)$, i.e.
$\mathrm{O}(U,q)=\{\,g\in\mathrm{GL}(U)\mid q(g u)=q(u)\ \forall u\in U\}$.
In the chosen basis $U=\langle x_4,x_5,x_6\rangle$ with $q=x_4^2+x_5^2+x_6^2$, this is
$\{\,B\in\mathrm{GL}_3\mid B^{\mathsf T}B=I_3\,\}$.
\end{remark}

\medskip
The residual $T$-stabilizer is finite; hence, the corresponding component $\Phi_{8}$ of the moduli is two-dimensional.

\subsection{Case $k=9$}

\paragraph{1-PS limit.}
Set
\[
\lambda_{9}(t)
=
\operatorname{diag}\!\bigl(t^{2},\,t^{2},\,1,\,1,\,t^{-1},\,t^{-1},\,t^{-2}\bigr),
\qquad
t\in\mathbb{G}_m.
\]
For a generic $f_{9}$ as in Section~3, the $1$-PS limit is
\[
\phi_{9}
:=
\lim_{t\to 0}\lambda_{9}(t)\cdot f_{9}
=
\begin{aligned}[t]
& a_{1}x_{2}^{3}+a_{2}x_{2}^{2}x_{3}+a_{3}x_{2}x_{3}^{2}+a_{4}x_{3}^{3} \\
& {}+ a_{5}x_{0}x_{4}^{2}+a_{6}x_{1}x_{4}^{2}+a_{7}x_{0}x_{4}x_{5}+a_{8}x_{1}x_{4}x_{5} \\
& {}+ a_{9}x_{0}x_{5}^{2}+a_{10}x_{1}x_{5}^{2} \\
& {}+ a_{11}x_{0}x_{2}x_{6}+a_{12}x_{1}x_{2}x_{6}+a_{13}x_{0}x_{3}x_{6}+a_{14}x_{1}x_{3}x_{6}.
\end{aligned}
\]
(All monomials have $H$-weight $0$.) 

\medskip
\paragraph{$H$ and $C_G(H)$.}
Let $H=\lambda_{9}(\mathbb{G}_m)$.
The $H$-weights on $\langle x_{0},\dots,x_{6}\rangle$ are $(2,2,0,0,-1,-1,-2)$ with block
decomposition $\langle x_{0},x_{1}\rangle$, $\langle x_{2},x_{3}\rangle$, $\langle x_{4},x_{5}\rangle$, $\langle x_{6}\rangle$.
Hence,
\[
C_G(H)
=
\Bigl\{
A\oplus B\oplus C\oplus \gamma
:\ 
A,B,C\in\mathrm{GL}(2),\ \gamma\in\mathbb{G}_m,\ 
\det(A)\det(B)\det(C)\gamma=1
\Bigr\}
\]
\[
\cong\ \bigl(\mathrm{GL}(2)^{3}\times\mathbb{G}_m\bigr)\cap \mathrm{SL}(7).
\]

\medskip
\paragraph{Polystability (Luna + Casimiro--Florentino).}
By Luna's reduction, the closedness of the $\mathrm{SL}(7)$-orbit of $\phi_{9}$
is equivalent to polystability for the $C_G(H)$-action on the $H$-fixed locus.
After conjugating inside each $\mathrm{GL}(2)$-block, any $\lambda\in Y\!\bigl(C_G(H)\bigr)$ may be taken with
\[
\mathrm{wt}(x_{0},\ldots,x_{6})
=
(a+u,\,a-u,\,b+v,\,b-v,\,c+w,\,c-w,\,d),
\qquad
S:=2a+2b+2c+d=0.
\]
A direct computation yields the positive linear identity
\begin{align}
&\bigl[w(x_{2}^{3})+w(x_{2}^{2}x_{3})+w(x_{2}x_{3}^{2})+w(x_{3}^{3})\bigr] \notag\\
&\quad + 2\bigl[w(x_{0}x_{4}^{2})+w(x_{0}x_{4}x_{5})+w(x_{0}x_{5}^{2})
+w(x_{1}x_{4}^{2})+w(x_{1}x_{4}x_{5})+w(x_{1}x_{5}^{2})\bigr] \notag\\
&\quad + 3\bigl[w(x_{0}x_{2}x_{6})+w(x_{1}x_{2}x_{6})+w(x_{0}x_{3}x_{6})+w(x_{1}x_{3}x_{6})\bigr]
= 12\,S. 
\end{align}
If $\lambda\in\Lambda_{\phi_{9}}$, then all the above weights are $\ge 0$ and $S=0$; hence, by $(9)$ they all vanish.
Solving gives
\[
b=v=u=w=0,
\qquad
a+2c=0,
\qquad
a+d=0.
\]
Writing $a=2k$ with $k\in\mathbb{Z}$ we obtain
\[
\mu_{k}(t)
:=
\operatorname{diag}\!\bigl(t^{2k},\,t^{2k},\,1,\,1,\,t^{-k},\,t^{-k},\,t^{-2k}\bigr),
\qquad
\Lambda_{\phi_{9}}
=
\{\mu_{k}\mid k\in\mathbb{Z}\}.
\]
Thus $\Lambda_{\phi_{9}}$ is symmetric, and by the Casimiro--Florentino criterion $\phi_{9}$ is polystable; in particular
$\mathrm{SL}(7)\cdot\phi_{9}$ is closed. 

\medskip
\paragraph{Normal form and component dimension.}

From now on, we normalize coefficients under the $C_G(H)$-action. Decompose
\[
W^{H}
=
\underbrace{\mathrm{Sym}^{3}\langle x_{2},x_{3}\rangle}_{(\mathrm{I})}
\ \oplus\ 
\underbrace{\langle x_{0},x_{1}\rangle\otimes \mathrm{Sym}^{2}\langle x_{4},x_{5}\rangle}_{(\mathrm{II})}
\ \oplus\ 
\underbrace{\langle x_{0},x_{1}\rangle\otimes\langle x_{2},x_{3}\rangle\otimes\langle x_{6}\rangle}_{(\mathrm{III})}.
\]

\paragraph{(I) Binary cubic $\mathrm{Sym}^{3}\langle x_{2},x_{3}\rangle$.}
A general element is $\GL(\langle x_{2},x_{3}\rangle)$-equivalent (up to overall scaling) to
\[
x_{2}^{2}x_{3}+\tau\,x_{2}x_{3}^{2},
\qquad
\tau\in\mathbb{C}^{\times},
\]
leaving one parameter $\tau$.

\smallskip
Indeed, over $\mathbb{C}$ any binary cubic $F\in \Sym^{3}\langle x_{2},x_{3}\rangle$ factors as
\[
F=L_{1}L_{2}L_{3},\qquad L_i\in \langle x_{2},x_{3}\rangle,
\]
and for a general $F$ the three linear factors (equivalently, the three roots on
$\PP^{1}=\PP(\langle x_{2},x_{3}\rangle)$) are distinct.  Since $\mathrm{PGL}(\langle x_{2},x_{3}\rangle)$
acts $3$--transitively on $\PP^{1}$, after a change of basis we may send two of the
roots to $[1\!:\!0]$ and $[0\!:\!1]$, i.e.\ arrange $L_{1}=x_{2}$ and $L_{2}=x_{3}$.  Thus
\[
F=x_{2}x_{3}(\alpha x_{2}+\beta x_{3}),\qquad \alpha,\beta\in\mathbb{C}^{\times}.
\]
Up to an overall scalar we may take $\alpha=1$ and set $\tau:=\beta/\alpha$, giving
\[
F\sim x_{2}^{2}x_{3}+\tau\,x_{2}x_{3}^{2}.
\]

\paragraph{(II) $\langle x_{0},x_{1}\rangle\otimes\mathrm{Sym}^{2}\langle x_{4},x_{5}\rangle$ ($2\times3$ block).}
Via the right $\GL(\langle x_{4},x_{5}\rangle)$ (through $\mathrm{Sym}^{2}$) diagonalize a reference quadratic
to $q=x_{4}^{2}+x_{5}^{2}$; the residual right group is $O(U,q)$ on $U=\langle x_{4},x_{5}\rangle$.
Using the left $\GL(\langle x_{0},x_{1}\rangle)$ together with this right orthogonal action
(an SVD-type reduction), eliminate the $x_{4}x_{5}$ cross term and equalize the $x_{4}^{2}$ entries.
After central torus/projective scalings,
\[
x_{0}(x_{4}^{2}+\rho\,x_{5}^{2}) + x_{1}(x_{4}^{2}+x_{5}^{2}),
\qquad
\rho\in\mathbb{C}^{\times}.
\]

\paragraph{(III) $\langle x_{0},x_{1}\rangle\otimes\langle x_{2},x_{3}\rangle\otimes \langle x_{6}\rangle$ ($2\times2$ block).}
With $\GL(\langle x_{0},x_{1}\rangle)$, $\GL(\langle x_{2},x_{3}\rangle)$ (respecting the choice in (I)), and scaling $x_{6}$,
we diagonalize to
\[
x_{0}x_{2}x_{6}+x_{1}x_{3}x_{6},
\]
and normalize the coefficients to $1$.

\medskip

Combining the three steps yields the normal form
\[
\nf_{9}(\tau,\rho)
=
x_{2}^{2}x_{3}
+\tau\,x_{2}x_{3}^{2}
+ x_{0}x_{4}^{2}
+\rho\,x_{0}x_{5}^{2}
+ x_{1}x_{4}^{2}
+ x_{1}x_{5}^{2}
+ x_{0}x_{2}x_{6}
+ x_{1}x_{3}x_{6},
\qquad
(\tau,\rho)\in(\mathbb{C}^{\times})^{2}.
\]
As $\dim W^{H}=14$ and $\dim C_G(H)=12$, the effective action has dimension $11$;
after projectivizing, we obtain $14-11-1=2$.
The residual $T$-stabilizer is finite; hence, the corresponding component $\Phi_{9}$ of the moduli is two-dimensional.

\subsection{Case $k=10$}

\paragraph{1-PS limit.}
Set
\[
  \lambda_{10}(t)
  = \operatorname{diag}\!\bigl(t^{2},\,t,\,1,\,1,\,t^{-1},\,t^{-1},\,t^{-1}\bigr),
  \qquad t\in\mathbb{G}_m .
\]
For a generic $f_{10}$ as in Section~3, the $1$-PS limit is
\[
\begin{aligned}
  \phi_{10}
  := \lim_{t\to0}\lambda_{10}(t)\cdot f_{10}
   &= a_{1}x_{2}^{3}+a_{2}x_{2}^{2}x_{3}+a_{3}x_{2}x_{3}^{2}+a_{4}x_{3}^{3} \\
   &\quad + a_{5}x_{1}x_{2}x_{4}+a_{6}x_{1}x_{3}x_{4}
          + a_{8}x_{1}x_{2}x_{5}+a_{9}x_{1}x_{3}x_{5} \\
   &\quad + a_{12}x_{1}x_{2}x_{6}+a_{13}x_{1}x_{3}x_{6} \\
   &\quad + a_{7}x_{0}x_{4}^{2}+a_{10}x_{0}x_{4}x_{5}+a_{11}x_{0}x_{5}^{2}
          + a_{14}x_{0}x_{4}x_{6}+a_{15}x_{0}x_{5}x_{6}+a_{16}x_{0}x_{6}^{2}.
\end{aligned}
\]

\paragraph{$H$ and $C_G(H)$.}
Let $H=\lambda_{10}(\mathbb{G}_m)$.
The $H$-weights on $\langle x_{0},\dots,x_{6}\rangle$ are $(2,1,0,0,-1,-1,-1)$ with block multiplicities $(1,1,2,3)$; hence,
\begin{align}
  C_G(H)
  &=\Bigl\{\ \operatorname{diag}(\alpha)\oplus \operatorname{diag}(\beta)\oplus A\oplus B
      :\ \alpha,\beta\in\mathbb{G}_m,\ A\in\mathrm{GL}(2),\ B\in\mathrm{GL}(3),\ \alpha\beta\det(A)\det(B)=1\ \Bigr\}
      \notag\\
  &\cong \bigl(\mathbb{G}_m\times\mathbb{G}_m\times \mathrm{GL}(2)\times \mathrm{GL}(3)\bigr)\cap \mathrm{SL}(7).
  \label{eq:k10-centralizer}
\end{align}
Every monomial of $\phi_{10}$ has $H$-weight $0$, so $\phi_{10}\in W^{H}$.

\paragraph{Polystability (Luna + Casimiro--Florentino).}
By Luna's reduction, the closedness of the $\mathrm{SL}(7)$-orbit of $\phi_{10}$
is equivalent to polystability for the $C_G(H)$-action on $W^{H}$.
After conjugating within the $\mathrm{GL}(2)$- and $\mathrm{GL}(3)$-blocks, any
$\lambda\in Y\!\bigl(C_G(H)\bigr)$ may be taken as
\[
  \mathrm{wt}(x_{0},\ldots,x_{6})
  =
  (\alpha,\,\beta,\,s+u,\,s-u,\,\gamma+v_{1},\,\gamma+v_{2},\,\gamma+v_{3}),
\]
where $\alpha,\beta,s,u,\gamma,v_i\in\mathbb{Z}$ with $v_{1}+v_{2}+v_{3}=0$, and by Convention~\ref{convention} the
$\mathrm{SL}$-constraint is
\[
  S:=\alpha+\beta+2s+3\gamma=0.
\]
A direct computation gives the positive linear identity
\begin{equation}\label{eq:k10-identity}
\begin{aligned}
 &\bigl[w(x_{2}^{3})+w(x_{2}^{2}x_{3})+w(x_{2}x_{3}^{2})+w(x_{3}^{3})\bigr]\\
 &\quad + 2\bigl[w(x_{0}x_{4}^{2})+w(x_{0}x_{4}x_{5})+w(x_{0}x_{5}^{2})+w(x_{0}x_{4}x_{6})+w(x_{0}x_{5}x_{6})+w(x_{0}x_{6}^{2})\bigr]\\
 &\quad + 2\bigl[w(x_{1}x_{2}x_{4})+w(x_{1}x_{3}x_{4})+w(x_{1}x_{2}x_{5})+w(x_{1}x_{3}x_{5})+w(x_{1}x_{2}x_{6})+w(x_{1}x_{3}x_{6})\bigr]
 \;=\; 12\,S.
\end{aligned}
\end{equation}
If $\lambda\in\Lambda_{\phi_{10}}$, then all the weights above are $\ge 0$ and $S=0$; hence, by
\eqref{eq:k10-identity}, they all vanish. Solving yields
\[
  \beta=-\gamma,\qquad s=0,\qquad u=0,\qquad v_{1}=v_{2}=v_{3}=0,\qquad \alpha=-2\gamma.
\]
Writing $\gamma=-k$ with $k\in\mathbb{Z}$ we obtain
\[
  \mu_{k}(t)
  :=
  \operatorname{diag}\!\bigl(t^{2k},\,t^{k},\,1,\,1,\,t^{-k},\,t^{-k},\,t^{-k}\bigr),
  \qquad
  \Lambda_{\phi_{10}}=\{\mu_{k}\mid k\in\mathbb{Z}\}.
\]
Thus $\Lambda_{\phi_{10}}$ is symmetric, and by the Casimiro--Florentino criterion, $\phi_{10}$ is polystable; in particular $\mathrm{SL}(7)\cdot\phi_{10}$ is closed.

\paragraph{Normal form and component dimension.}
Set
\[
U:=\langle x_{2},x_{3}\rangle,\qquad V:=\langle x_{4},x_{5},x_{6}\rangle.
\]
Then the $H$-fixed subspace decomposes as
\[
W^{H}
=
\underbrace{\Sym^{3}U}_{(\mathrm{I})}
\ \oplus\
\underbrace{x_{1}\otimes U\otimes V}_{(\mathrm{II})}
\ \oplus\
\underbrace{x_{0}\otimes \Sym^{2}V}_{(\mathrm{III})}.
\]
We normalize a \emph{general} element of $W^{H}$ under the $C_G(H)$-action, taking care that the
$\GL(U)$-freedom is shared by blocks (I) and (II).

\paragraph{(III) The quadratic block $x_{0}\otimes\Sym^{2}V$.}
For a general element, the quadratic form on $V$ is nondegenerate. Acting by $\GL(V)\subset C_G(H)$
and absorbing an overall scalar into the $x_{0}$-factor (using the central torus/projective scaling),
we may assume the quadratic block is
\[
x_{0}\,q,\qquad q=x_{4}x_{5}+x_{6}^{2}.
\]
After this normalization, the residual right group on $V$ is the orthogonal group
\[
\mathrm{O}(V,q)=\{\,B\in \GL(V)\mid q(Bv)=q(v)\,\}.
\]

\paragraph{(II) The mixed block $x_{1}\otimes U\otimes V$.}
Write the mixed term as $x_{1}m$ with $m\in U\otimes V$, and represent $m$ by a $2\times 3$ matrix
after choosing bases of $U$ and $V$.
For a general element, $\rank(m)=2$ and the restriction $q|_{\mathrm{Im}(m)}$ is nondegenerate.
Equivalently, pulling back $q$ along $m$ yields a nondegenerate binary quadratic form on $U$:
\[
q_{m}(u):=q(m(u))\in \Sym^{2}U^{\vee}.
\]
Using $\GL(U)$ we may put $q_{m}$ into split form. Then, by Witt's extension theorem, there exists
$B\in \mathrm{O}(V,q)$ sending the resulting hyperbolic pair in $\mathrm{Im}(m)$ to
$\langle x_{4},x_{5}\rangle$. Hence, after the $\GL(U)\times \mathrm{O}(V,q)$-action and scaling $x_{1}$,
we may assume
\[
x_{1}m = x_{1}(x_{2}x_{5}+x_{3}x_{4}).
\]
(If $\rank(m)=1$ or $q|_{\mathrm{Im}(m)}$ is degenerate, one lands in a proper closed subset; for the
generic closed orbit in this component we work on the open set above.)

\paragraph{Residual symmetry after (II) and (III).}
After fixing (II) and (III) as above, the stabilizer contains the one-dimensional torus
\[
T=\Bigl\{\ 
x_{2}\mapsto t\,x_{2},\ x_{3}\mapsto t^{-1}x_{3},\ 
x_{4}\mapsto t\,x_{4},\ x_{5}\mapsto t^{-1}x_{5},\ 
x_{6}\mapsto x_{6}
\ \Bigr\}\subset C_G(H),
\qquad t\in\mathbb{G}_m,
\]
which preserves both $x_{2}x_{5}+x_{3}x_{4}$ and $x_{4}x_{5}+x_{6}^{2}$.

\paragraph{(I) The cubic block $\Sym^{3}U$.}
A general element of $\Sym^{3}U$ is
\[
c_{0}x_{2}^{3}+c_{1}x_{2}^{2}x_{3}+c_{2}x_{2}x_{3}^{2}+c_{3}x_{3}^{3}.
\]
Under $T$, these monomials have weights $t^{3},t^{1},t^{-1},t^{-3}$, respectively.
On the dense open set where $c_{0}c_{3}\neq 0$, using the $T$-parameter and projective scaling we
normalize $c_{0}=c_{3}=1$. Writing $\tau:=c_{1}$ and $\rho:=c_{2}$, we obtain
\[
x_{2}^{3}+\tau\,x_{2}^{2}x_{3}+\rho\,x_{2}x_{3}^{2}+x_{3}^{3}.
\]

\medskip\noindent
Combining (I)--(III), we may take the closed-orbit representative in Case $k=10$ to be
\[
\nf_{10}(\tau,\rho)
=
x_{2}^{3}+\tau\,x_{2}^{2}x_{3}+\rho\,x_{2}x_{3}^{2}+x_{3}^{3}
+ x_{1}x_{2}x_{5}+x_{1}x_{3}x_{4}
+ x_{0}x_{4}x_{5}+x_{0}x_{6}^{2},
\qquad
(\tau,\rho)\in (\mathbb{C}^{\times})^{2}
\]
As $\dim W^{H}=16$ and $\dim C_G(H)=14$, the effective action has dimension $13$;
after projectivizing, we obtain $16-13-1=2$.
\emph{The residual $T$-stabilizer of a general representative is finite; hence, the corresponding
component $\Phi_{10}$ of the moduli is two-dimensional.}

\subsection{Case $k=11$}

\paragraph{1-PS limit.}
Set
\[
  \lambda_{11}(t)=\operatorname{diag}\!\bigl(t^{2},\,1,\,1,\,1,\,1,\,t^{-1},\,t^{-1}\bigr),
  \qquad t\in\mathbb{G}_m .
\]
For a generic $f_{11}$ as in Section~3, the $1$-PS limit is
\[
  \phi_{11}:=\lim_{t\to0}\lambda_{11}(t)\cdot f_{11}
  = F(x_{1},x_{2},x_{3},x_{4}) + x_{0}\,Q(x_{5},x_{6}),
\]
where $F\in \Sym^{3}\langle x_{1},x_{2},x_{3},x_{4}\rangle$ is a generic cubic form and $Q\in \Sym^{2}\langle x_{5},x_{6}\rangle$ is a generic quadratic form.

\paragraph{$H$ and $C_G(H)$.}
Let $H=\lambda_{11}(\mathbb{G}_m)$.
The $H$-weights on $\langle x_{0},\dots,x_{6}\rangle$ are $(2,0,0,0,0,-1,-1)$ with block decomposition
$\langle x_{0}\rangle\oplus\langle x_{1},x_{2},x_{3},x_{4}\rangle\oplus\langle x_{5},x_{6}\rangle$.
Hence,
\begin{align}
  C_G(H)
  &= \Bigl\{\ \operatorname{diag}(\alpha)\oplus A\oplus B \ :\ \alpha\in\mathbb{G}_m,\ A\in\mathrm{GL}(4),\ B\in\mathrm{GL}(2),\ \alpha\,\det(A)\det(B)=1\ \Bigr\}\notag\\
  &\cong \bigl(\mathbb{G}_m\times \mathrm{GL}(4)\times \mathrm{GL}(2)\bigr)\cap \mathrm{SL}(7).
  \label{eq:k11-centralizer}
\end{align}
Every monomial of $\phi_{11}$ has $H$-weight $0$, so $\phi_{11}\in W^{H}$.

\paragraph{Polystability (Luna + Casimiro--Florentino).}
By Luna's reduction, the closedness of the $\mathrm{SL}(7)$-orbit of $\phi_{11}$ is equivalent to polystability for the $C_G(H)$-action on $W^{H}$.
After conjugating inside the $\mathrm{GL}(4)$- and $\mathrm{GL}(2)$-blocks, any $\lambda\in Y\!\bigl(C_G(H)\bigr)$ may be taken with
\[
  \mathrm{wt}(x_{0},\ldots,x_{6})=(\alpha,\,s_{1},\,s_{2},\,s_{3},\,s_{4},\,t+u,\,t-u),
  \qquad
  S:=\alpha+(s_{1}+s_{2}+s_{3}+s_{4})+2t=0
\]
(Convention~\ref{convention}).
Summing the $\lambda$-weights of the $20$ cubic monomials in $F$ gives $15(s_{1}+s_{2}+s_{3}+s_{4})$, while the weights of the quadratic terms $x_0 Q$ sum to $3\alpha+6t$.
Using the relation $S=0$, we obtain the positive linear identity
\begin{equation}\label{eq:k11-identity}
  w(F) + 5\,w(x_0 Q) \;=\; 15(s_{1}+s_{2}+s_{3}+s_{4}) + 15(\alpha+2t) \;=\; 15\,S.
\end{equation}
If $\lambda\in\Lambda_{\phi_{11}}$, then all weights are $\ge 0$ and $S=0$; by \eqref{eq:k11-identity} they all vanish.
Solving this implies that $\Lambda_{\phi_{11}}$ is symmetric, and by the Casimiro--Florentino criterion, $\phi_{11}$ is polystable.

\paragraph{Normal form and component dimension.}
We normalize generic elements under the action of the centralizer $C_G(H)$.
The $H$-fixed subspace decomposes as
\[
W^{H}
=
\Sym^{3}\langle x_{1},x_{2},x_{3},x_{4}\rangle
\ \oplus\ 
x_{0}\otimes \Sym^{2}\langle x_{5},x_{6}\rangle.
\]

\noindent
(1) \textbf{Quadratic part.} Under the $\mathrm{GL}(2)$-action on $\langle x_5, x_6\rangle$, a generic binary quadratic is normalized to $x_5^2 + x_6^2$. The stabilizer of this form contains the orthogonal group $\mathrm{O}(2)$, contributing $1$ dimension to the stabilizer of the generic point.

\noindent
(2) \textbf{Cubic part.} Under the $\mathrm{GL}(4)$-action on $\langle x_1, x_2, x_3, x_4\rangle$, a generic cubic surface can be written as the sum of five cubes of linear forms (\textbf{Sylvester's Pentahedron Theorem}). Let $L_0, \dots, L_4$ be fixed general linear forms in $x_1, \dots, x_4$. Then the cubic part is $\sum_{i=0}^4 c_i L_i^3$.

Thus, a closed-orbit representative is:
\[
\nf_{11}(c) = \sum_{i=0}^4 c_i L_i(x_1, x_2, x_3, x_4)^3 + x_0(x_5^2 + x_6^2),
\]
where the parameters are $[c_0 : \dots : c_4] \in \mathbb{P}^4$.

The dimension of the component is computed as follows:
$\dim W^H = 20 (\text{cubic}) + 3 (\text{quadratic}) = 23$.
The effective group dimension is $\dim C_G(H) - 1 = 19$ (since $H$ acts trivially).
Since the generic stabilizer has dimension $1$ (due to $\mathrm{O}(2)$), the orbit dimension is $19 - 1 = 18$.
\[
\dim(\Phi_{11}) = 23 - 18 - 1 = 4.
\]
(Note: This case is structurally isomorphic to Case $k=21$.)

\subsection{Case $k=12$}

\paragraph{1-PS limit.}
Set
\[
  \lambda_{12}(t)=\operatorname{diag}\!\bigl(t^{3},\,t^{2},\,t,\,t,\,t^{-1},\,t^{-2},\,t^{-4}\bigr),
  \qquad t\in\mathbb{G}_m .
\]
For a generic $f_{12}$ as in Section~3, the $1$-PS limit is
\[
  \phi_{12}:=\lim_{t\to0}\lambda_{12}(t)\cdot f_{12}
  = a_{1}x_{1}x_{4}^{2}+a_{2}x_{2}^{2}x_{5}+a_{3}x_{2}x_{3}x_{5}+a_{4}x_{3}^{2}x_{5}
   +a_{5}x_{0}x_{4}x_{5}+a_{6}x_{1}^{2}x_{6}+a_{7}x_{0}x_{2}x_{6}+a_{8}x_{0}x_{3}x_{6}.
\]

\paragraph{$H$ and $C_G(H)$.}
Let $H=\lambda_{12}(\mathbb{G}_m)$.
The $H$-weights on $\langle x_{0},\dots,x_{6}\rangle$ are $(3,2,1,1,-1,-2,-4)$ with multiplicities $(1,1,2,1,1,1)$; hence,
\begin{align}
  C_G(H)
  &= \Bigl\{\ \operatorname{diag}(\alpha)\oplus \operatorname{diag}(\beta)\oplus A\oplus
      \operatorname{diag}(\gamma,\delta,\varepsilon)\ :\ \alpha,\beta,\gamma,\delta,\varepsilon\in\mathbb{G}_m,\ A\in\mathrm{GL}(2),\notag\\
  &\qquad \alpha\beta\det(A)\gamma\delta\varepsilon=1\ \Bigr\}
   \ \cong\ \bigl(\mathbb{G}_m^{5}\times\mathrm{GL}(2)\bigr)\cap \mathrm{SL}(7),
   \qquad \dim C_G(H)=8.
   \label{eq:k12-CG}
\end{align}
Every monomial of $\phi_{12}$ has $H$-weight $0$.

\paragraph{Polystability (Luna + Casimiro--Florentino).}
By Luna's reduction, the closedness of the $\mathrm{SL}(7)$-orbit of $\phi_{12}$ is equivalent to
polystability for the $C_G(H)$-action on $W^{H}$.
After conjugating inside the $\mathrm{GL}(2)$-block on $\langle x_{2},x_{3}\rangle$, any
$\lambda\in Y(C_G(H))$ may be taken with
\[
  \mathrm{wt}(x_{0},\ldots,x_{6})=(\alpha,\,\beta,\,s+u,\,s-u,\,\gamma,\,\delta,\,\varepsilon),
  \qquad S:=\alpha+\beta+2s+\gamma+\delta+\varepsilon=0.
\]
A direct computation yields the positive identity
\begin{align}
  &2 w(x_{1}x_{4}^{2})+w(x_{2}^{2}x_{5})+w(x_{2}x_{3}x_{5})+2 w(x_{3}^{2}x_{5})
   +2 w(x_{0}x_{4}x_{5})+2 w(x_{1}^{2}x_{6})\notag\\
  &\qquad +3 w(x_{0}x_{2}x_{6})+w(x_{0}x_{3}x_{6})=6S.
\label{eq:k12-identity}
\end{align}
If $\lambda\in \Lambda_{\phi_{12}}$, then all the weights on the left are $\ge0$ and $S=0$; by
\eqref{eq:k12-identity} they all vanish. Solving gives
\[
  u=0,\qquad \beta=-2\gamma,\qquad \delta=-2s,\qquad \varepsilon=-2\beta=4\gamma,\qquad
  \alpha=-3\gamma,\qquad s=-\gamma.
\]
Writing $k= - \gamma \in\mathbb{Z}$, we obtain
\[
  \mu_{k}(t)=\operatorname{diag}\!\bigl(t^{3k},\,t^{2k},\,t^{k},\,t^{k},\,t^{-k},\,t^{-2k},\,t^{-4k}\bigr),
  \qquad \Lambda_{\phi_{12}}=\{\mu_{k}\mid k\in\mathbb{Z}\}
\]
Thus $\Lambda_{\phi_{12}}$ is symmetric, and by the Casimiro--Florentino criterion $\phi_{12}$ is polystable; in particular, $\mathrm{SL}(7)\cdot\phi_{12}$ is closed.

\paragraph{Normal form and component dimension.}
Work on the $H$-fixed slice
\[
W^{H}
=
\underbrace{x_{1}\otimes \Sym^{2}\langle x_{4}\rangle}_{(\mathrm{I})}
\ \oplus\ 
\underbrace{\Sym^{2}\langle x_{2},x_{3}\rangle\otimes x_{5}}_{(\mathrm{II})}
\ \oplus\ 
\underbrace{x_{0}\otimes \langle x_{4}\rangle\otimes x_{5}}_{(\mathrm{III})}
\ \oplus\ 
\underbrace{\Sym^{2}\langle x_{1}\rangle\otimes x_{6}}_{(\mathrm{IV})}
\]
\[
\ \oplus\ 
\underbrace{x_{0}\otimes \langle x_{2},x_{3}\rangle\otimes x_{6}}_{(\mathrm{V})}.
\]

\paragraph{Diagonalize the binary quadratic in the $(\mathrm{II})$-block.}
Using $\GL(\langle x_{2},x_{3}\rangle)$, bring
$\Sym^{2}\langle x_{2},x_{3}\rangle\otimes x_{5}$ to $x_{2}^{2}x_{5} + x_{3}^{2}x_{5}$; the cross term
$x_{2}x_{3}x_{5}$ is eliminated.

\paragraph{Align the $(\mathrm{V})$-block.}
Within $x_{0}\otimes\langle x_{2},x_{3}\rangle\otimes x_{6}$, use the same $\GL(2)$ to align
this block to $x_{0}x_{2}x_{6}$ (so the $x_{0}x_{3}x_{6}$ entry vanishes).

\paragraph{Normalize coefficients by torus scalings and projective scaling in $(\mathrm{I})$, $(\mathrm{III})$, $(\mathrm{V})$.}
Use the $1$-dimensional tori on the $1$-dimensional weight spaces and the overall
projective scaling to set the remaining nonzero coefficients to $1$.

This yields the closed orbit normal form
\[
\nf_{12}
=
x_{1}x_{4}^{2}
+ x_{2}^{2}x_{5}
+ x_{3}^{2}x_{5}
+ x_{0}x_{4}x_{5}
+ x_{1}^{2}x_{6}
+ x_{0}x_{2}x_{6}.
\]
Finally, $\dim W^{H}=8$ and $\dim C_G(H)=8$; the effective action has dimension $7$
(with $H$ acting trivially). After projectivizing, we obtain $8-7-1=0$; hence, the
corresponding boundary component is zero-dimensional.

\subsection{Case $k=13$}

\paragraph{1-PS limit.}
Set
\[
  \lambda_{13}(t)=\operatorname{diag}\!\bigl(t^{2},\,t,\,t,\,1,\,t^{-1},\,t^{-1},\,t^{-2}\bigr),
  \qquad t\in\mathbb{G}_m .
\]
For a generic $f_{13}$ as in Section~3, the $1$-PS limit is
\[
\begin{aligned}
  \phi_{13}
  :=\lim_{t\to0}\lambda_{13}(t)\cdot f_{13}
   &= a_{1}x_{3}^{3}+a_{2}x_{1}x_{3}x_{4}+a_{3}x_{2}x_{3}x_{4}+a_{4}x_{0}x_{4}^{2}
    +a_{5}x_{1}x_{3}x_{5}+a_{6}x_{2}x_{3}x_{5}\\
   &\quad +a_{7}x_{0}x_{4}x_{5}+a_{8}x_{0}x_{5}^{2}
    +a_{9}x_{1}^{2}x_{6}+a_{10}x_{1}x_{2}x_{6}+a_{11}x_{2}^{2}x_{6}+a_{12}x_{0}x_{3}x_{6}.
\end{aligned}
\]
This is $H$-fixed for $H=\lambda_{13}(\mathbb{G}_m)$. 

\paragraph{$H$ and $C_G(H)$.}
The $H$-weights on $\langle x_{0},\dots,x_{6}\rangle$ are $(2,1,1,0,-1,-1,-2)$ with blocks
$\langle x_{0}\rangle\oplus\langle x_{1},x_{2}\rangle\oplus\langle x_{3}\rangle\oplus\langle x_{4},x_{5}\rangle\oplus\langle x_{6}\rangle$.
Hence,
\begin{align}
  C_G(H)
  &=\Bigl\{\ \operatorname{diag}(\alpha)\ \oplus\ A\ \oplus\ \operatorname{diag}(\beta)\ \oplus\ B\ \oplus\ \operatorname{diag}(\gamma)
     :\ \alpha,\beta,\gamma\in\mathbb{G}_m,\ A,B\in\mathrm{GL}(2),\notag\\
  &\qquad\qquad\qquad\qquad\qquad\alpha\,\det(A)\,\beta\,\det(B)\,\gamma=1\,\Bigr\} \notag\\
  &\cong \bigl(\mathbb{G}_m\times \mathrm{GL}(2)\times \mathbb{G}_m\times \mathrm{GL}(2)\times \mathbb{G}_m\bigr)\cap \mathrm{SL}(7),
  \qquad \dim C_G(H)=10. 
  \label{eq:k13-CG}
\end{align}
Each monomial of $\phi_{13}$ has $H$-weight $0$. 

\paragraph{Polystability (Luna + Casimiro--Florentino).}
By Luna's reduction, the closedness of the $\mathrm{SL}(7)$-orbit of $\phi_{13}$ is equivalent to polystability for the $C_G(H)$-action on $W^{H}$.
After conjugating within the two $\mathrm{GL}(2)$-blocks, any $\lambda\in Y(C_G(H))$ may be taken as
\[
  \mathrm{wt}(x_{0},\ldots,x_{6})=(\alpha,\,s+u,\,s-u,\,\beta,\,t+v,\,t-v,\,\gamma),
  \qquad
  S:=\alpha+2s+\beta+2t+\gamma=0\ \ (\text{Convention \ref{convention}}).
\] 
A direct computation yields the positive identity
\begin{align}
  &\bigl[w(x_{1}^{2}x_{6})+w(x_{2}^{2}x_{6})+2\,w(x_{1}x_{2}x_{6})\bigr]
   +\bigl[w(x_{0}x_{4}^{2})+2\,w(x_{0}x_{4}x_{5})+w(x_{0}x_{5}^{2})\bigr]\notag\\
  &\qquad +\bigl[w(x_{1}x_{3}x_{4})+w(x_{2}x_{3}x_{4})+w(x_{1}x_{3}x_{5})+w(x_{2}x_{3}x_{5})\bigr]
   + 2\,w(x_{0}x_{3}x_{6})=6S.
\label{eq:k13-identity}
\end{align}
If $\lambda\in\Lambda_{\phi_{13}}$, then the twelve weights on the left are $\ge 0$ and $S=0$; hence, they all vanish. Solving gives
\[
  u=0,\quad v=0,\quad \alpha+2t=0,\quad 2s+\gamma=0,\quad s+\beta+t=0,\quad \alpha+\beta+\gamma=0.
\]
Writing $s=k\in\mathbb{Z}$ yields
\[
  (\alpha,s,\beta,t,\gamma)=(2k,\,k,\,0,\,-k,\,-2k),
  \qquad
  \mu_{k}(t):=\operatorname{diag}\!\bigl(t^{2k},\,t^{k},\,t^{k},\,1,\,t^{-k},\,t^{-k},\,t^{-2k}\bigr),
\]
so $\Lambda_{\phi_{13}}=\{\mu_{k}\mid k\in\mathbb{Z}\}$ is symmetric; by the Casimiro--Florentino criterion $\phi_{13}$ is polystable; hence, $\mathrm{SL}(7)\cdot\phi_{13}$ is closed. 

\paragraph{Normal form and component dimension.}
We display
\[
\begin{aligned}
W^{H}
&=
\underbrace{\Sym^{3}\langle x_{3}\rangle}_{(\mathrm{I})}
\ \oplus\ 
\underbrace{x_{0}\otimes \Sym^{2}\langle x_{4},x_{5}\rangle}_{(\mathrm{II})}
\ \oplus\ 
\underbrace{\langle x_{1},x_{2}\rangle\otimes \langle x_{4},x_{5}\rangle\otimes \langle x_{3}\rangle}_{(\mathrm{III})}\\
&\ \ \ \oplus\ 
\underbrace{\Sym^{2}\langle x_{1},x_{2}\rangle\otimes \langle x_{6}\rangle}_{(\mathrm{IV})}
\ \oplus\ 
\underbrace{x_{0}\otimes \langle x_{3}\rangle\otimes x_{6}}_{(\mathrm{V})}.
\end{aligned}
\]
Normalize block by block under the $C_G(H)$-action:

\paragraph{(II) $x_{0}\otimes \Sym^{2}\langle x_{4},x_{5}\rangle$.}
Diagonalize the binary quadratic to
\[
x_{0}x_{4}^{2}+x_{0}x_{5}^{2}.
\]

\paragraph{(III) $\langle x_{1},x_{2}\rangle\otimes \langle x_{4},x_{5}\rangle\otimes \langle x_{3}\rangle$.}
View the four $x_{3}$-bilinear terms as a $2\times 2$ matrix on
$\langle x_{1},x_{2}\rangle\otimes\langle x_{4},x_{5}\rangle$ and bring it to diagonal form,
\[
x_{1}x_{3}x_{4}+\rho\,x_{2}x_{3}x_{5},
\qquad
\rho\in\mathbb{C}^\times.
\]

\paragraph{(IV) $\Sym^{2}\langle x_{1},x_{2}\rangle\otimes \langle x_{6}\rangle$.}
Diagonalize the symmetric $2\times 2$ form to
\[
x_{1}^{2}x_{6}+\sigma\,x_{2}^{2}x_{6},
\qquad
\sigma\in\mathbb{C}^\times.
\]

\paragraph{(I) and (V).}
Use the three torus factors on $x_{0},x_{3},x_{6}$ together with
projective scaling to normalize the remaining nonzero coefficients to $1$.

\medskip\noindent
This yields the convenient normal form
\[
\nf_{13}(\rho,\sigma)
=
x_{3}^{3}
+ x_{0}x_{4}^{2}
+ x_{0}x_{5}^{2}
+ x_{1}x_{3}x_{4}
+ \rho\,x_{2}x_{3}x_{5}
+ x_{1}^{2}x_{6}
+ \sigma\,x_{2}^{2}x_{6}
+ x_{0}x_{3}x_{6},
\qquad
(\rho,\sigma)\in(\mathbb{C}^\times)^{2}.
\]
Here, $\dim(W^{H})=12$ and $\dim C_G(H)=10$; as $H$ acts trivially, the effective group dimension is $9$. After projectivizing, we obtain $12-9-1=2$.
\emph{The residual $T$-stabilizer is finite; hence, the corresponding component $\Phi_{13}$ of the moduli is two-dimensional.}

\subsection{Case $k=14$}

\paragraph{1-PS limit.}
Set
\[
  \lambda_{14}(t)=\operatorname{diag}\!\bigl(t^{2},\,t^{2},\,1,\,t^{-1},\,t^{-1},\,t^{-1},\,t^{-1}\bigr),
  \qquad t\in\mathbb{G}_m.
\]
For a generic $f_{14}$ as in Section~3, the $1$-PS limit is
\[
\begin{aligned}
  \phi_{14}:=\lim_{t\to0}\lambda_{14}(t)\cdot f_{14}
   &= a_{1}x_{2}^{3}
    + a_{2}x_{0}x_{3}^{2}+a_{3}x_{1}x_{3}^{2}+a_{4}x_{0}x_{3}x_{4}+a_{5}x_{1}x_{3}x_{4}
    + a_{6}x_{0}x_{4}^{2}+a_{7}x_{1}x_{4}^{2}\\
   &\quad + a_{8}x_{0}x_{3}x_{5}+a_{9}x_{1}x_{3}x_{5}+a_{10}x_{0}x_{4}x_{5}+a_{11}x_{1}x_{4}x_{5}
    + a_{12}x_{0}x_{5}^{2}+a_{13}x_{1}x_{5}^{2}\\
   &\quad + a_{14}x_{0}x_{3}x_{6}+a_{15}x_{1}x_{3}x_{6}+a_{16}x_{0}x_{4}x_{6}+a_{17}x_{1}x_{4}x_{6}\\
   &\quad + a_{18}x_{0}x_{5}x_{6}+a_{19}x_{1}x_{5}x_{6}+a_{20}x_{0}x_{6}^{2}+a_{21}x_{1}x_{6}^{2}.
\end{aligned}
\]

\paragraph{$H$ and $C_G(H)$.}
Let $H=\lambda_{14}(\mathbb{G}_m)$.
The $H$-weights on $\langle x_{0},\dots,x_{6}\rangle$ are $(2,2,0,-1,-1,-1,-1)$ with block multiplicities $(2,1,4)$.
Hence,
\[
\begin{aligned}
  C_G(H)
  &=\Bigl\{\,A\oplus \beta\oplus B:\ A\in\mathrm{GL}(2),\ \beta\in\mathbb{G}_m,\ B\in\mathrm{GL}(4),\ \det(A)\,\beta\,\det(B)=1\,\Bigr\}\\
  &\cong \bigl(\mathrm{GL}(2)\times\mathbb{G}_m\times\mathrm{GL}(4)\bigr)\cap \mathrm{SL}(7),\qquad \dim C_G(H)=20.
\end{aligned}
\]
Every monomial of $\phi_{14}$ has $H$-weight $0$; hence, $\phi_{14}\in W^{H}$.

\paragraph{Polystability (Luna + Casimiro--Florentino).}
By Luna's reduction, the closedness of the $\mathrm{SL}(7)$-orbit of $\phi_{14}$ is equivalent to polystability for the $C_G(H)$-action on $W^{H}$.
After conjugating inside the $\mathrm{GL}(2)$- and $\mathrm{GL}(4)$-blocks, any $\lambda\in Y\!\bigl(C_G(H)\bigr)$ may be taken with
\[
  \mathrm{wt}(x_{0},\ldots,x_{6})=(a+u,\,a-u,\,b,\,c+v_{1},\,c+v_{2},\,c+v_{3},\,c+v_{4}),
\]
where $v_{1}+v_{2}+v_{3}+v_{4}=0$, and with the $\mathrm{SL}$-constraint (Convention~\ref{convention})
\[
  S:=2a+b+4c=0.
\]
A direct computation yields the positive identity
\begin{equation}\label{eq:k14-identity}
\begin{aligned}
 &10 w(x_2^3) + 3 \left[ \sum_{j=3}^6 \left( w(x_0 x_j^2) + w(x_1 x_j^2) \right) + \sum_{3 \le k < l \le 6} w(x_0 x_k x_l) + w(x_1 x_k x_l) \right] = 30 S
\end{aligned}
\end{equation}
If $\lambda\in\Lambda_{\phi_{14}}$, then all $21$ weights are $\ge 0$ and $S=0$; hence, by \eqref{eq:k14-identity} they all vanish.
Solving gives
\[
  b=0,\qquad u=0,\qquad v_{1}=v_{2}=v_{3}=v_{4}=0,\qquad a+2c=0.
\]
Writing $c=-k$ with $k\in\mathbb{Z}$, we obtain
\[
  \mu_{k}(t)=\operatorname{diag}\!\bigl(t^{2k},\,t^{2k},\,1,\,t^{-k},\,t^{-k},\,t^{-k},\,t^{-k}\bigr),\qquad
  \Lambda_{\phi_{14}}=\{\mu_{k}\mid k\in\mathbb{Z}\}.
\]
Thus $\Lambda_{\phi_{14}}$ is symmetric, and by the Casimiro--Florentino criterion $\phi_{14}$ is polystable; in particular $\mathrm{SL}(7)\cdot\phi_{14}$ is closed.

\paragraph{Normal form and component dimension.}
We display $W^{H}$ in block form:
\[
\begin{aligned}
W^{H}
&=
\underbrace{\Sym^{3}\langle x_{2}\rangle}_{(\mathrm{I})}
\ \oplus\
\underbrace{\langle x_{0}\rangle\otimes \Sym^{2}\langle x_{3},x_{4},x_{5},x_{6}\rangle}_{(\mathrm{II})}
\ \oplus\
\underbrace{\langle x_{1}\rangle\otimes \Sym^{2}\langle x_{3},x_{4},x_{5},x_{6}\rangle}_{(\mathrm{III})}.
\end{aligned}
\]
Equivalently, writing $U:=\langle x_{3},x_{4},x_{5},x_{6}\rangle$ and $V:=\langle x_{0},x_{1}\rangle$, any element of $W^{H}$ can be written as
\[
\phi \;=\; a\,x_{2}^{3} \;+\; x_{0}Q_{0}(x_{3},x_{4},x_{5},x_{6}) \;+\; x_{1}Q_{1}(x_{3},x_{4},x_{5},x_{6}),
\qquad Q_{0},Q_{1}\in \Sym^{2}U^{\vee}.
\]
Using the $\mathbb{G}_{m}$-factor on $x_{2}$ together with projective scaling, we normalize $a=1$.

\medskip\noindent
\emph{Diagonalization of the quadratic pencil.}
Consider the pencil $\langle Q_{0},Q_{1}\rangle$, i.e.\ the binary quartic
\[
\Delta(s,t):=\det(sQ_{0}+tQ_{1}),\qquad [s:t]\in\PP^{1}.
\]
On the dense open set where $\Delta$ has four distinct roots (equivalently, the pencil is \emph{regular}, i.e.\ it lies off the discriminant locus), we may assume $Q_{0}$ is nondegenerate.
Acting by $\GL(U)\cong\GL(4)$, we take
\[
Q_{0}\sim q:=x_{3}^{2}+x_{4}^{2}+x_{5}^{2}+x_{6}^{2}.
\]
With respect to $q$, the residual group is $O(U,q)$, and for a regular pencil the $q$-selfadjoint endomorphism
$A$ defined by $Q_{1}(u,v)=q(Au,v)$ has four distinct eigenvalues. Hence, after an $O(U,q)$-change of coordinates,
\[
Q_{1}\sim \lambda_{3}x_{3}^{2}+\lambda_{4}x_{4}^{2}+\lambda_{5}x_{5}^{2}+\lambda_{6}x_{6}^{2},
\qquad \lambda_{3},\lambda_{4},\lambda_{5},\lambda_{6}\ \text{pairwise distinct}.
\]

\medskip\noindent
\emph{Reduction to a one-parameter normal form (cross-ratio).}
The $\GL(V)\cong\GL(2)$-action on $(x_{0},x_{1})$ replaces $(Q_{0},Q_{1})$ by
\[
(Q_{0}',Q_{1}')=(\alpha Q_{0}+\beta Q_{1},\ \gamma Q_{0}+\delta Q_{1}),
\qquad
\begin{pmatrix}\alpha&\beta\\ \gamma&\delta\end{pmatrix}\in\GL(2),
\]
and (after re-normalizing $Q_{0}'$ back to $q$ by a diagonal rescaling in $U$) this transforms the eigenvalues by the
fractional-linear map
\[
\lambda\ \longmapsto\ \frac{\gamma+\delta\lambda}{\alpha+\beta\lambda}.
\]
Thus the regular pencil is determined, up to permutation of the diagonal coordinates in $U$,
by the unordered set $\{\lambda_{3},\lambda_{4},\lambda_{5},\lambda_{6}\}\subset \PP^{1}$ modulo $\mathrm{PGL}_{2}$,
i.e.\ by a single cross-ratio parameter.

To match the normal form displayed below, we normalize three eigenvalues to $1,-1,2$ and denote the fourth by $\tau$.
Using the $3$-transitivity of $\mathrm{PGL}_{2}$ on $\PP^{1}$ and then permuting $(x_{3},x_{4},x_{5},x_{6})$, we may assume
\[
(\lambda_{3},\lambda_{4},\lambda_{5},\lambda_{6})=(1,\tau,-1,2).
\]
Consequently (on the dense open set $\tau\notin\{1,-1,2\}$; and restricting further to $\tau\in\C^{\times}$ if we also require $Q_{1}$ nondegenerate),
a convenient normal form is
\[
\nf_{14}(\tau)
= x_{2}^{3}
+ x_{0}\bigl(x_{3}^{2}+x_{4}^{2}+x_{5}^{2}+x_{6}^{2}\bigr)
+ x_{1}\bigl(x_{3}^{2}+\tau x_{4}^{2}-x_{5}^{2}+2x_{6}^{2}\bigr),
\qquad
\tau\in\C^{\times}\setminus\{1,-1,2\}.
\]
(For the special values $\tau\in\{1,-1,2\}$ the pencil acquires repeated eigenvalues and the stabilizer jumps; these form a proper closed subset of the same one-dimensional family.)

As $\dim W^{H}=21$ and $\dim C_G(H)=20$, and since $H\simeq\mathbb{G}_m$ acts trivially on $W^{H}$, the effective action has
dimension $19$; after projectivizing we obtain $21-19-1=1$.
\emph{For a general parameter $\tau$ the residual stabilizer is finite; hence, the corresponding component $\Phi_{14}$ of the moduli is one-dimensional.}

\subsection{Case $k=15$}

\paragraph{1-PS limit.}
Set
\[
  \lambda_{15}(t)=\operatorname{diag}\!\bigl(t^{2},\,t,\,t,\,1,\,1,\,t^{-2},\,t^{-2}\bigr),
  \qquad t\in\mathbb{G}_m.
\]
For a generic $f_{15}$ as in Section~3, the $1$-PS limit is
\[
\begin{aligned}
  \phi_{15}:=\lim_{t\to0}\lambda_{15}(t)\cdot f_{15}
   &= a_{1}x_{3}^{3}+a_{2}x_{3}^{2}x_{4}+a_{3}x_{3}x_{4}^{2}+a_{4}x_{4}^{3}
    +a_{5}x_{1}^{2}x_{5}+a_{6}x_{1}x_{2}x_{5}+a_{7}x_{2}^{2}x_{5}\\
   &\quad +a_{8}x_{0}x_{3}x_{5}+a_{9}x_{0}x_{4}x_{5}
    +a_{10}x_{1}^{2}x_{6}+a_{11}x_{1}x_{2}x_{6}+a_{12}x_{2}^{2}x_{6}\\
   &\quad +a_{13}x_{0}x_{3}x_{6}+a_{14}x_{0}x_{4}x_{6}.
\end{aligned}
\]
All monomials have $H$-weight $0$, so $\phi_{15}\in W^{H}$ for $H=\lambda_{15}(\mathbb{G}_m)$.  

\paragraph{$H$ and $C_G(H)$.}
The $H$-weights on $\langle x_{0},\dots,x_{6}\rangle$ are $(2,1,1,0,0,-2,-2)$ with block decomposition
$\langle x_{0}\rangle\oplus\langle x_{1},x_{2}\rangle\oplus\langle x_{3},x_{4}\rangle\oplus\langle x_{5},x_{6}\rangle$.
Hence,
\[
\begin{aligned}
  C_G(H)
  &=\Bigl\{\ \operatorname{diag}(\alpha)\oplus A\oplus B\oplus C:\ \alpha\in\mathbb{G}_m,\ A,B,C\in\mathrm{GL}(2),\
      \alpha\,\det(A)\det(B)\det(C)=1\ \Bigr\}\\
  &\cong \bigl(\mathbb{G}_m\times \mathrm{GL}(2)\times \mathrm{GL}(2)\times \mathrm{GL}(2)\bigr)\cap \mathrm{SL}(7),
  \qquad \dim C_G(H)=12.
\end{aligned}
\]

\paragraph{Polystability (Luna + Casimiro--Florentino).}
By Luna's reduction, the closedness of the $\mathrm{SL}(7)$-orbit of $\phi_{15}$ is equivalent to polystability for the
$C_G(H)$-action on $W^{H}$.
After conjugating inside the three $\mathrm{GL}(2)$-blocks, any $\lambda\in Y\!\bigl(C_G(H)\bigr)$ may be taken with
\[
  \mathrm{wt}(x_{0},\ldots,x_{6})
  =
  (\alpha,\,s+u,\,s-u,\,t+v,\,t-v,\,\gamma+w,\,\gamma-w),
  \qquad
  S:=\alpha+2s+2t+2\gamma=0.
\]
A direct computation yields the positive identity
\begin{align}
  &2\bigl[w(x_{3}^{3})+w(x_{3}^{2}x_{4})+w(x_{3}x_{4}^{2})+w(x_{4}^{3})\bigr]
   +3\Bigl[\bigl(w(x_{1}^{2}x_{5})+2w(x_{1}x_{2}x_{5})+w(x_{2}^{2}x_{5})\bigr)\notag\\
  &\qquad\qquad\qquad\qquad\qquad\qquad
            +\bigl(w(x_{1}^{2}x_{6})+2w(x_{1}x_{2}x_{6})+w(x_{2}^{2}x_{6})\bigr)\Bigr]\notag\\
  &\qquad\qquad\qquad\qquad\qquad\qquad
   +6\bigl[w(x_{0}x_{3}x_{5})+w(x_{0}x_{4}x_{5})+w(x_{0}x_{3}x_{6})+w(x_{0}x_{4}x_{6})\bigr]
   \;=\;24\,S. \label{eq:k15-identity}
\end{align}
If $\lambda\in\Lambda_{\phi_{15}}$, then all $14$ weights above are $\ge 0$ and $S=0$; hence, by \eqref{eq:k15-identity} they all vanish.
Solving gives
\[
  u=0,\quad v=0,\quad w=0,\quad t=0,\quad \alpha+\gamma=0,\quad 2s+\gamma=0.
\]
Writing $s=k\in\mathbb{Z}$ we obtain
\[
  \mu_{k}(t):=\operatorname{diag}\!\bigl(t^{2k},\,t^{k},\,t^{k},\,1,\,1,\,t^{-2k},\,t^{-2k}\bigr),
  \qquad \Lambda_{\phi_{15}}=\{\mu_{k}\mid k\in\mathbb{Z}\}.
\]
Thus $\Lambda_{\phi_{15}}$ is symmetric; by the Casimiro--Florentino criterion $\phi_{15}$ is polystable, so $\mathrm{SL}(7)\cdot\phi_{15}$ is closed.

\paragraph{Normal form and component dimension.}
We have
\[
W^{H}
=
\underbrace{\Sym^{3}\langle x_{3},x_{4}\rangle}_{(\mathrm{I})}
\ \oplus\ 
\underbrace{\bigl(\Sym^{2}\langle x_{1},x_{2}\rangle\bigr)\otimes \langle x_{5},x_{6}\rangle}_{(\mathrm{II})}
\ \oplus\
\underbrace{x_{0}\otimes \langle x_{3},x_{4}\rangle\otimes \langle x_{5},x_{6}\rangle}_{(\mathrm{III})},
\]
of respective dimensions $4$, $6$, and $4$ (total $\dim W^{H}=14$). 

\paragraph{Reduction to normal form.}
We now normalize $\varphi_{15}$ under the action of $C_G(H)$ on $W^{H}$, using only:
(i) the left $\GL(2)$ on $\langle x_{3},x_{4}\rangle$,
(ii) the left $\Sym^{2}\GL(2)$ on $\Sym^{2}\langle x_{1},x_{2}\rangle$,
(iii) the right $\GL(2)$ on $\langle x_{5},x_{6}\rangle$,
(iv) diagonal tori (subject to $\det=1$) and projective rescaling.
We proceed block by block.

\medskip
\paragraph{(I) The binary cubic block $\Sym^{3}\langle x_{3},x_{4}\rangle$.}
A general binary cubic is $\GL(2)$-equivalent (after one overall scalar) to
\[
x_{3}^{2}x_{4}+\tau\,x_{3}x_{4}^{2},
\qquad
\tau\in\mathbb{C}^\times,
\]
which fixes the $\Sym^{3}$-part up to the single modulus $\tau$.

\medskip
\paragraph{(III) The bilinear $2\times 2$ block $x_{0}\otimes\langle x_{3},x_{4}\rangle\otimes \langle x_{5},x_{6}\rangle$.}
Write this part as $x_{0}\,(x_{3},x_{4})\,M\,(x_{5},x_{6})^{\!\top}$ with $M\in M_{2\times2}$. 
Using the left $\GL(2)$ action on $\langle x_{3},x_{4}\rangle$ and the right $\GL(2)$ action on $\langle x_{5},x_{6}\rangle$ (an SVD-type reduction), we bring $M$
to the identity; a diagonal torus and projective rescaling normalize the two coefficients to $1$:
\[
x_{0}x_{3}x_{5} + x_{0}x_{4}x_{6}.
\]

\medskip
\paragraph{(II) The $3\times 2$ block $\bigl(\Sym^{2}\langle x_{1},x_{2}\rangle\bigr)\otimes\langle x_{5},x_{6}\rangle$.}
Choose bases $\{x_{1}^{2},\ x_{1}x_{2},\ x_{2}^{2}\}$ and $\{x_{5},x_{6}\}$.
The left action of $\Sym^{2}\GL(2)$ on $\Sym^{2}\langle x_{1},x_{2}\rangle$ together with the right action of $\GL(2)$ on $\langle x_{5},x_{6}\rangle$
allows a simultaneous reduction that eliminates the mixed $x_{1}x_{2}$ row and diagonalizes the remaining two rows.
After using diagonal tori and an overall scale, we obtain
\[
x_{1}^{2}x_{5} + \rho\,x_{2}^{2}x_{6},
\qquad
\rho\in\mathbb{C}^\times.
\]

\medskip
\paragraph{Normal form.}
Hence, a closed-orbit representative is
\begin{equation}\label{eq:nf15}
\nf_{15}(\tau,\rho)
\;=\;
x_{3}^{2}x_{4} \;+\; \tau\,x_{3}x_{4}^{2}
\;+\; x_{0}x_{3}x_{5} \;+\; x_{0}x_{4}x_{6}
\;+\; x_{1}^{2}x_{5} \;+\; \rho\,x_{2}^{2}x_{6},
\qquad
(\tau,\rho)\in(\mathbb{C}^\times)^{2}.
\end{equation}
For general $(\tau,\rho)$ the residual stabilizer in $C_G(H)$ is finite, so the parameters
$(\tau,\rho)$ record genuine moduli on the closed stratum.

\subsection{Case $k=16$}

\paragraph{1-PS limit.}
Set
\[
  \lambda_{16}(t)=\operatorname{diag}\!\bigl(t^{2},\,t,\,1,\,1,\,1,\,t^{-1},\,t^{-2}\bigr),\qquad t\in\mathbb{G}_m .
\]
For a generic $f_{16}$ as in Section~3, the $1$-PS limit is
\[
\begin{aligned}
  \phi_{16}
  := \lim_{t\to0}\lambda_{16}(t)\cdot f_{16}
  &= a_{1}x_{2}^{3}+a_{2}x_{2}^{2}x_{3}+a_{3}x_{2}x_{3}^{2}+a_{4}x_{3}^{3}
   +a_{5}x_{2}^{2}x_{4}+a_{6}x_{2}x_{3}x_{4}+a_{7}x_{3}^{2}x_{4}\\
  &\quad +a_{8}x_{2}x_{4}^{2}+a_{9}x_{3}x_{4}^{2}+a_{10}x_{4}^{3}
   +a_{11}x_{1}x_{2}x_{5}+a_{12}x_{1}x_{3}x_{5}+a_{13}x_{1}x_{4}x_{5}\\
  &\quad +a_{14}x_{0}x_{5}^{2}+a_{15}x_{1}^{2}x_{6}+a_{16}x_{0}x_{2}x_{6}
   +a_{17}x_{0}x_{3}x_{6}+a_{18}x_{0}x_{4}x_{6}.
\end{aligned}
\]
(All monomials have $H$-weight $0$, so $\phi_{16}\in W^{H}$ for $H=\lambda_{16}(\mathbb{G}_m)$.)  

\paragraph{$H$ and $C_G(H)$.}
The $H$-weights on $\langle x_{0},\dots,x_{6}\rangle$ are $(2,1,0,0,0,-1,-2)$ with block decomposition
$\langle x_{0}\rangle\oplus\langle x_{1}\rangle\oplus\langle x_{2},x_{3},x_{4}\rangle\oplus\langle x_{5}\rangle\oplus\langle x_{6}\rangle$.
Hence,
\[
\begin{aligned}
  C_G(H)
  &=\Bigl\{\ \operatorname{diag}(\alpha)\oplus \operatorname{diag}(\beta)\oplus B\oplus \operatorname{diag}(\delta)\oplus \operatorname{diag}(\varepsilon)\ :\
     \alpha,\beta,\delta,\varepsilon\in\mathbb{G}_m,\ B\in\mathrm{GL}(3),\\
  &\qquad\qquad\alpha\,\beta\,\det(B)\,\delta\,\varepsilon=1\ \Bigr\} \\[2pt]
  &\cong\ \bigl(\mathbb{G}_m\times\mathbb{G}_m\times \mathrm{GL}(3)\times \mathbb{G}_m\times \mathbb{G}_m\bigr)\cap \mathrm{SL}(7),
  \qquad \dim C_G(H)=12.
\end{aligned}
\]  

\paragraph{Polystability (Luna + Casimiro--Florentino).}
By Luna's reduction, the closedness of the $\mathrm{SL}(7)$-orbit of $\phi_{16}$ is equivalent to polystability for the $C_G(H)$-action on $W^{H}$.
After conjugating within the $\mathrm{GL}(3)$-block on $\langle x_{2},x_{3},x_{4}\rangle$, any $\lambda\in Y\!\bigl(C_G(H)\bigr)$ may be taken with
\[
  \mathrm{wt}(x_{0},\ldots,x_{6})=(\alpha,\,\beta,\,c+v_{1},\,c+v_{2},\,c+v_{3},\,\delta,\,\varepsilon),
  \qquad v_{1}+v_{2}+v_{3}=0,
\]
and, by Convention~\ref{convention}, with the $\mathrm{SL}$-constraint
\[
  S:=\alpha+\beta+3c+\delta+\varepsilon=0.
\]  
Writing $w(\cdot)$ for the $\lambda$-weight of a monomial, a direct calculation yields the positive identity
\begin{equation}\label{eq:k16-identity}
\sum_{\substack{i+j+k=3\\ i,j,k\ge 0}} w(x_2^i x_3^j x_4^k)
\;+\;6\,w(x_0x_5^2)\;+\;6\,w(x_1^2x_6)
\;+\;2\sum_{i=2}^{4} w(x_0x_ix_6)
\;=\;12S.
\end{equation}
If $\lambda\in\Lambda_{\phi_{16}}$, then all $18$ weights on the left are $\ge0$ and $S=0$; hence, by \eqref{eq:k16-identity} they all vanish. From the ten cubics, we obtain
\[
  c=0,\qquad v_{1}=v_{2}=v_{3}=0,
\]
and from the remaining terms
\[
  \alpha+2\delta=0,\qquad 2\beta+\varepsilon=0,\qquad \alpha+\varepsilon=0.
\]
Thus $\alpha=2\beta,\ \delta=-\beta,\ \varepsilon=-2\beta$. Writing $\beta=k\in\mathbb{Z}$ gives
\[
  \mu_{k}(t)=\operatorname{diag}\!\bigl(t^{2k},\,t^{k},\,1,\,1,\,1,\,t^{-k},\,t^{-2k}\bigr),
  \qquad \Lambda_{\phi_{16}}=\{\mu_{k}\mid k\in\mathbb{Z}\}.
\]
As $\Lambda_{\phi_{16}}$ is symmetric, $\phi_{16}$ is polystable by the Casimiro--Florentino criterion; in particular, $\mathrm{SL}(7)\cdot\phi_{16}$ is closed.  

\paragraph{Normal form and component dimension.}

On the $H$--fixed slice we have the block decomposition
\[
W^{H}
=
\underbrace{\Sym^{3}\langle x_{2},x_{3},x_{4}\rangle}_{(\mathrm{I})}
\;\oplus\;
\underbrace{x_{1}\otimes\langle x_{2},x_{3},x_{4}\rangle\otimes x_{5}}_{(\mathrm{II})}
\;\oplus\;
\underbrace{\langle x_{0}x_{5}^{2}\rangle}_{(\mathrm{III})}
\;\oplus\;
\underbrace{\langle x_{1}^{2}x_{6}\rangle}_{(\mathrm{IV})}
\;\oplus\;
\underbrace{x_{0}\otimes\langle x_{2},x_{3},x_{4}\rangle\otimes x_{6}}_{(\mathrm{V})}.
\]

The centralizer is
\[
CG(H)\simeq (\mathbb{G}_{m})^{4}\times \GL\!\left(\langle x_{2},x_{3},x_{4}\rangle\right),
\]
acting blockwise (subject to the determinant-one condition).  
We normalize the blocks one by one.

\medskip
\textbf{Step 1: Normalization of the rank-one blocks $(\mathrm{II})$ and $(\mathrm{V})$.}

The blocks
\[
(\mathrm{II})=\{x_{1}x_{2}x_{5},\, x_{1}x_{3}x_{5},\, x_{1}x_{4}x_{5}\},\qquad
(\mathrm{V})=\{x_{0}x_{2}x_{6},\, x_{0}x_{3}x_{6},\, x_{0}x_{4}x_{6}\}
\]
correspond to two linear functionals on
\[
U^{\vee}=\langle x_{2},x_{3},x_{4}\rangle^{\vee}.
\]
For a general member, these two $3$-vectors are linearly independent; hence there exists a single 
$g\in\GL(U)$ sending them to the coordinate covectors:
\[
(\mathrm{II})\;\sim\; x_{1}x_{2}x_{5},
\qquad
(\mathrm{V})\;\sim\; x_{0}x_{3}x_{6}.
\]
Using the diagonal tori on $(x_{1},x_{5})$ and $(x_{0},x_{6})$ (and projective scaling), we normalize the coefficients to $1$.

\medskip
\textbf{Step 2: Normalization of the one-dimensional blocks $(\mathrm{III})$ and $(\mathrm{IV})$.}

The $\mathbb{G}_{m}$-factors on $x_{0}$ and $x_{1}$ allow us to set
\[
(\mathrm{III})\sim x_{0}x_{5}^{2},
\qquad
(\mathrm{IV})\sim x_{1}^{2}x_{6}.
\]

\medskip
\textbf{Step 3: Normalization of the ternary cubic in $(\mathrm{I})$.}

Let
\[
F(x_{2},x_{3},x_{4})\in \Sym^{3}\langle x_{2},x_{3},x_{4}\rangle.
\]
After Step~1, the residual subgroup of $\GL(\langle x_{2},x_{3},x_{4}\rangle)$ preserving $(\mathrm{II})$ and $(\mathrm{V})$ consists of
\[
x_{2}\mapsto s\,x_{2},\qquad
x_{3}\mapsto t\,x_{3},\qquad
x_{4}\mapsto u\,x_{4}+\alpha x_{2}+\beta x_{3},
\qquad
(s,t,u)\in(\mathbb{G}_{m})^{3},\;(\alpha,\beta)\in\mathbb{A}^{2}.
\]

\textit{(a) Elimination of $x_{2}x_{4}^{2}$ and $x_{3}x_{4}^{2}$.}  
The substitution $x_{4}\mapsto x_{4}+\alpha x_{2}+\beta x_{3}$ yields
\[
(x_{4}+\alpha x_{2}+\beta x_{3})^{3}
=
x_{4}^{3}+3\alpha\,x_{2}x_{4}^{2}+3\beta\,x_{3}x_{4}^{2}+\cdots,
\]
so choosing $(\alpha,\beta)$ suitably (with the $x_{4}^{3}$ coefficient nonzero), we kill both coefficients.

\textit{(b) Normalization of pure cubes}.  
Scaling by $(s,t,u)$ (and compensating using the diagonal tori on $(x_{1},x_{5})$ and $(x_{0},x_{6})$ so that the coefficients fixed in Steps~1--2 remain equal to $1$), we may assume
\[
F=x_{2}^{3}+x_{3}^{3}+\sigma x_{4}^{3}+\cdots,
\qquad
\sigma\in\mathbb{C}^{\times}.
\]

\textit{(c) Remaining mixed terms}.  
After (a) and (b), the possible mixed terms are
\[
x_{2}^{2}x_{3},\; x_{2}x_{3}^{2},\; x_{2}^{2}x_{4},\; x_{3}^{2}x_{4},\; x_{2}x_{3}x_{4}.
\]
Note that the residual subgroup preserving the normalizations in Step~1 does \emph{not} contain the shear
$x_{3}\mapsto x_{3}+\gamma x_{2}$ (it would reintroduce an $x_{0}x_{2}x_{6}$ term in the $(\mathrm{V})$-block). Hence the coefficient of $x_{2}x_{3}^{2}$ cannot be eliminated in general. Therefore we write
\[
F(x_{2},x_{3},x_{4})
=
x_{2}^{3}+x_{3}^{3}+\sigma x_{4}^{3}
+\rho\,x_{2}x_{3}x_{4}
+\kappa\,x_{2}^{2}x_{3}
+\nu\,x_{2}x_{3}^{2}
+\mu\,x_{3}^{2}x_{4}
+\lambda\,x_{2}^{2}x_{4},
\]
with $\rho,\kappa,\nu,\mu,\lambda\in\mathbb{C}^{\times}$.

\medskip
\textbf{Collecting all blocks.}

Hence a closed-orbit representative is
\[
\begin{aligned}
\nf_{16}
&=
x_{2}^{3}+x_{3}^{3}+\sigma x_{4}^{3}
+\rho\,x_{2}x_{3}x_{4}
+\kappa\,x_{2}^{2}x_{3}
+\nu\,x_{2}x_{3}^{2}
+\mu\,x_{3}^{2}x_{4}
+\lambda\,x_{2}^{2}x_{4}
\\
&\qquad
+x_{1}x_{2}x_{5}
+x_{0}x_{3}x_{6}
+x_{0}x_{5}^{2}
+x_{1}^{2}x_{6},
\end{aligned}
\]
with six free parameters
\[
(\sigma,\rho,\kappa,\nu,\mu,\lambda)\in(\mathbb{C}^{\times})^{6}.
\]

\paragraph{Component dimension.}

We compute
\[
\dim W^{H}
=
10+3+1+1+3 = 18.
\]
Since $C_{G}(H)$ has dimension $12$ and the $H\simeq\mathbb{G}_{m}$ factor acts trivially on $W^{H}$, the effective dimension is
\[
\dim_{\mathrm{eff}}C_{G}(H)=11.
\]
Projectivization subtracts one more dimension, giving
\[
\dim \Phi_{16}
=
18 - 11 - 1 = 6,
\]
agreeing with the six independent parameters of $\nf_{16}$.

\subsection{Case $k=17$}

\paragraph{1-PS limit.}
Set
\[
  \lambda_{17}(t)=\mathrm{diag}\!\bigl(t,\,t,\,t,\,1,\,1,\,t^{-1},\,t^{-2}\bigr),
  \qquad t\in\mathbb{G}_m.
\]
For a generic $f_{17}$ as in Section~3, the $1$-PS limit is
\[
\begin{aligned}
  \phi_{17}:=\lim_{t\to 0}\lambda_{17}(t)\cdot f_{17}
  &= a_{1}\,x_{3}^{3}+a_{2}\,x_{3}^{2}x_{4}+a_{3}\,x_{3}x_{4}^{2}+a_{4}\,x_{4}^{3}\\
  &\quad +a_{5}\,x_{0}x_{3}x_{5}+a_{6}\,x_{1}x_{3}x_{5}+a_{7}\,x_{2}x_{3}x_{5}\\
  &\quad +a_{8}\,x_{0}x_{4}x_{5}+a_{9}\,x_{1}x_{4}x_{5}+a_{10}\,x_{2}x_{4}x_{5}\\
  &\quad +a_{11}\,x_{0}^{2}x_{6}+a_{12}\,x_{0}x_{1}x_{6}+a_{13}\,x_{1}^{2}x_{6}\\
  &\quad +a_{14}\,x_{0}x_{2}x_{6}+a_{15}\,x_{1}x_{2}x_{6}+a_{16}\,x_{2}^{2}x_{6}.
\end{aligned}
\]

\paragraph{$H$ and $C_{G}(H)$.}
Let $H=\lambda_{17}(\mathbb{G}_m)$. The $H$-weights on
$\langle x_{0},\dots,x_{6}\rangle$ are $(1,1,1,0,0,-1,-2)$ with block
multiplicities $\langle x_{0},x_{1},x_{2}\rangle$, $\langle x_{3},x_{4}\rangle$,
$\langle x_{5}\rangle$, $\langle x_{6}\rangle$. Thus
\[
\begin{aligned}
  C_{G}(H)
  &=\Bigl\{\,A\oplus B\oplus\mathrm{diag}(\gamma)\oplus\mathrm{diag}(\delta):\
    A\in\mathrm{GL}(3),\ B\in\mathrm{GL}(2),\ \gamma,\delta\in\mathbb{G}_m,\\
  &\hspace{7.2em}\det(A)\det(B)\gamma\delta=1\,\Bigr\}\\
  &\cong \bigl(\mathrm{GL}(3)\times\mathrm{GL}(2)\times\mathbb{G}_m\times\mathbb{G}_m\bigr)\cap \mathrm{SL}(7).
\end{aligned}
\]
Each monomial of $\phi_{17}$ has $H$-weight $0$.
Moreover, $\dim C_G(H)=14$.

\paragraph{Polystability (Luna + Casimiro--Florentino).}
By Luna's reduction, the closedness of the $\mathrm{SL}(7)$-orbit of $\phi_{17}$
is equivalent to polystability for the $C_{G}(H)$-action on the $H$-fixed
subspace. After conjugating inside the $\mathrm{GL}(3)$- and $\mathrm{GL}(2)$-blocks,
any $\lambda\in Y\!\bigl(C_{G}(H)\bigr)$ may be taken with
\[
  \mathrm{wt}(x_{0},\dots,x_{6})
  =(a+u_{1},\,a+u_{2},\,a+u_{3},\,b+v,\,b-v,\,c,\,d),
  \quad u_{1}+u_{2}+u_{3}=0,
\]
and the $\mathrm{SL}(7)$-condition (Convention~\ref{convention})
\[
  S:=3a+2b+c+d=0.
\]
A direct calculation yields the positive linear identity
\begin{align}
  &\bigl[w(x_{3}^{3})+w(x_{3}^{2}x_{4})+w(x_{3}x_{4}^{2})+w(x_{4}^{3})\bigr]\notag\\
  &\quad + 2\bigl[w(x_{0}x_{3}x_{5})+w(x_{1}x_{3}x_{5})+w(x_{2}x_{3}x_{5})
                 +w(x_{0}x_{4}x_{5})+w(x_{1}x_{4}x_{5})+w(x_{2}x_{4}x_{5})\bigr]\notag\\
  &\quad + 2\bigl[w(x_{0}^{2}x_{6})+w(x_{0}x_{1}x_{6})+w(x_{1}^{2}x_{6})
                 +w(x_{0}x_{2}x_{6})+w(x_{1}x_{2}x_{6})+w(x_{2}^{2}x_{6})\bigr]\notag\\
  &\qquad = 12\,S.\label{eq:k17-identity}
\end{align}
If $\lambda\in\Lambda_{\phi_{17}}$, then all $16$ monomial weights are $\ge0$ and
$S=0$; by \eqref{eq:k17-identity} they must all vanish. Solving gives
\[
  b=0,\quad v=0,\quad u_{1}=u_{2}=u_{3}=0,\quad c=-a,\quad d=-2a.
\]
Writing $a=k$ ($k\in\mathbb{Z}$) yields
\[
  \Lambda_{\phi_{17}}=\{\mu_{k}\mid k\in\mathbb{Z}\},\qquad
  \mu_{k}(t)=\mathrm{diag}\!\bigl(t^{k},\,t^{k},\,t^{k},\,1,\,1,\,t^{-k},\,t^{-2k}\bigr).
\]
Thus, $\Lambda_{\phi_{17}}$ is symmetric, and by the Casimiro--Florentino criterion,
$\phi_{17}$ is polystable; in particular, $\mathrm{SL}(7)\cdot\phi_{17}$ is closed.

\paragraph{Normal form and component dimension.}
On the $H$-fixed slice
\[
W^{H}
=
\underbrace{\mathrm{Sym}^{3}\langle x_{3},x_{4}\rangle}_{(\mathrm{I})}
\ \oplus\ 
\underbrace{\langle x_{0},x_{1},x_{2}\rangle\otimes \langle x_{3},x_{4}\rangle\otimes x_{5}}_{(\mathrm{II})}
\ \oplus\ 
\underbrace{\mathrm{Sym}^{2}\langle x_{0},x_{1},x_{2}\rangle\otimes x_{6}}_{(\mathrm{III})},
\]
of total dimension $4+6+6=16$.

\smallskip
\paragraph{$(\mathrm{III})$ The quadratic block $\mathrm{Sym}^{2}\langle x_{0},x_{1},x_{2}\rangle\otimes x_{6}$.}
By the $\mathrm{GL}_{3}$-action, one diagonalizes the ternary quadratic, and using diagonal
tori together with projective rescaling, the coefficients of $x_{0}^{2}x_{6}$ and $x_{1}^{2}x_{6}$
are normalized to $1$, leaving
\[
x_{0}^{2}x_{6} \;+\; x_{1}^{2}x_{6} \;+\; \rho\,x_{2}^{2}x_{6},
\qquad
\rho\in\mathbb{C}^{\times}.
\]
Importantly, after the reduction of $(\mathrm{II})$ below, the subgroup of $\mathrm{GL}_{3}$ that
preserves $(\mathrm{II})$ is block-diagonal on $\langle x_{0},x_{1}\rangle\oplus\langle x_{2}\rangle$ (up to an
$\mathrm{O}(2)$ on $\langle x_{0},x_{1}\rangle$ and an independent scaling on $x_{2}$). Hence, the relative
scale of $x_{2}$ cannot be absorbed, and the parameter $\rho$ cannot be removed.

\smallskip
\paragraph{$(\mathrm{II})$ The $3\times2$ block $\langle x_{0},x_{1},x_{2}\rangle\otimes \langle x_{3},x_{4}\rangle\otimes x_{5}$.}
Using the left $\mathrm{GL}_{3}$ and right $\mathrm{GL}_{2}$ actions (an SVD-type reduction), a
generic rank-$2$ element is brought to diagonal form; scaling $x_{5}$ then fixes the two
nonzero entries to $1$:
\[
x_{0}x_{3}x_{5} \;+\; x_{1}x_{4}x_{5}.
\]

\smallskip
\paragraph{$(\mathrm{I})$ The binary cubic $\mathrm{Sym}^{3}\langle x_{3},x_{4}\rangle$.}
Acting by $\mathrm{GL}_{2}$ on $\langle x_{3},x_{4}\rangle$ puts a general binary cubic into a two-term
form (sending three roots to $\{0,\infty,-\tau\}$), and a residual diagonal scaling fixes the
first coefficient to $1$:
\[
x_{3}^{2}x_{4} \;+\; \tau\,x_{3}x_{4}^{2},\qquad \tau\in\mathbb{C}^{\times}.
\]

\smallskip
Collecting $(\mathrm{I})$-$(\mathrm{III})$, we obtain the closed-orbit representative
\[
\nf_{17}(\tau,\rho)
=
x_{3}^{2}x_{4}+\tau\,x_{3}x_{4}^{2}
+x_{0}x_{3}x_{5}
+x_{1}x_{4}x_{5}
+x_{0}^{2}x_{6}
+x_{1}^{2}x_{6}
+\rho\,x_{2}^{2}x_{6},
\qquad
(\tau,\rho)\in(\mathbb{C}^{\times})^{2}.
\]

Lastly, $\dim W^{H}=16$ and $\dim C_G(H)=14$. As the $H\simeq\mathbb{G}_m$-factor
acts trivially on $W^{H}$, the effective group dimension is $14-1=13$. After
projectivizing, we get
\[
\dim(\text{component})=\dim W^{H}-13-1=2,
\]
so $(\tau,\rho)$ are the two genuine moduli of the closed stratum.
Hence, the corresponding component $\Phi_{17}$ of the moduli is two-dimensional.

\subsection{Case $k=18$}

\paragraph{1-PS limit.}
Set
\[
  \lambda_{18}(t)=\mathrm{diag}\!\bigl(t,\,t,\,1,\,1,\,1,\,t^{-1},\,t^{-1}\bigr),
  \qquad t\in\mathbb{G}_m.
\]
For a generic $f_{18}$ as in Section~3, the $1$-PS limit is
\[
\begin{aligned}
  \phi_{18}:=\lim_{t\to 0}\lambda_{18}(t)\cdot f_{18}
  &= a_{1}\,x_{2}^{3}+a_{2}\,x_{2}^{2}x_{3}+a_{3}\,x_{2}x_{3}^{2}+a_{4}\,x_{3}^{3}
   +a_{5}\,x_{2}^{2}x_{4}+a_{6}\,x_{2}x_{3}x_{4}+a_{7}\,x_{3}^{2}x_{4}\\
  &\quad +a_{8}\,x_{2}x_{4}^{2}+a_{9}\,x_{3}x_{4}^{2}+a_{10}\,x_{4}^{3}
   +a_{11}\,x_{0}x_{2}x_{5}+a_{12}\,x_{1}x_{2}x_{5}+a_{13}\,x_{0}x_{3}x_{5}\\
  &\quad +a_{14}\,x_{1}x_{3}x_{5}+a_{15}\,x_{0}x_{4}x_{5}+a_{16}\,x_{1}x_{4}x_{5}
   +a_{17}\,x_{0}x_{2}x_{6}+a_{18}\,x_{1}x_{2}x_{6}\\
  &\quad +a_{19}\,x_{0}x_{3}x_{6}+a_{20}\,x_{1}x_{3}x_{6}+a_{21}\,x_{0}x_{4}x_{6}
   +a_{22}\,x_{1}x_{4}x_{6}.
\end{aligned}
\]

\paragraph{$H$ and $C_{G}(H)$.}
Let $H=\lambda_{18}(\mathbb{G}_m)$. The $H$-weights on $\langle x_{0},\dots,x_{6}\rangle$ are  
$(1,1,0,0,0,-1,-1)$ with block decomposition
\[
\langle x_{0},x_{1}\rangle\oplus\langle x_{2},x_{3},x_{4}\rangle\oplus\langle x_{5},x_{6}\rangle .
\]
Hence
\[
\begin{aligned}
  C_{G}(H)
  &= \Bigl\{\,A\oplus B\oplus C:\ A\in\mathrm{GL}(2),\ B\in\mathrm{GL}(3),\ C\in\mathrm{GL}(2),\\
  &\hspace{7.5em}\det(A)\det(B)\det(C)=1\,\Bigr\}\\
  &\cong \bigl(\mathrm{GL}(2)\times\mathrm{GL}(3)\times\mathrm{GL}(2)\bigr)\cap \mathrm{SL}(7),
\end{aligned}
\]
so $\dim C_{G}(H)=16$.  
Each monomial of $\phi_{18}$ has $H$-weight $0$; hence $\phi_{18}\in W^{H}$.

\paragraph{Polystability (Luna + Casimiro--Florentino).}
By Luna's reduction, the closedness of the $\mathrm{SL}(7)$-orbit of $\phi_{18}$ is equivalent to
polystability for the $C_{G}(H)$-action on the $H$-fixed subspace.
After conjugating within the three blocks, any $\lambda\in Y\!\bigl(C_{G}(H)\bigr)$ may be
taken with
\[
  \mathrm{wt}(x_{0},\dots,x_{6})
  =(a+u,\,a-u,\,b+v_{1},\,b+v_{2},\,b+v_{3},\,c+w,\,c-w),
\]
where $a,b,c,u,w,v_{i}\in\mathbb{Z}$ and $v_{1}+v_{2}+v_{3}=0$, and with the $\mathrm{SL}$-constraint
\[
  S:=2a+3b+2c=0 .
\]

Let $w(\cdot)$ denote the $\lambda$-weight of a monomial. Then
\[
  \sum_{\text{all 10 cubics in }x_{2},x_{3},x_{4}} w \;=\; 30\,b,
\]
and
\[
  \sum_{i=2}^{4}[w(x_{0}x_{i}x_{5})+w(x_{1}x_{i}x_{5})]=6(a+b+c+w),
\qquad
  \sum_{i=2}^{4}[w(x_{0}x_{i}x_{6})+w(x_{1}x_{i}x_{6})]=6(a+b+c-w).
\]

Hence the positive identity
\begin{align}
  &\Bigl[\sum_{\text{all 10 cubics in }x_{2},x_{3},x_{4}} w\Bigr]
   + 5\sum_{i=2}^{4}\bigl[w(x_{0}x_{i}x_{5})+w(x_{1}x_{i}x_{5})\bigr] \notag\\
  &\quad
   + 5\sum_{i=2}^{4}\bigl[w(x_{0}x_{i}x_{6})+w(x_{1}x_{i}x_{6})\bigr]
   \;=\; 30\,S. \label{eq:k18-identity}
\end{align}
If $\lambda\in\Lambda_{\phi_{18}}$, then all $22$ monomial weights are $\ge 0$ and $S=0$; hence, by
\eqref{eq:k18-identity} they must all vanish. Solving yields
\[
  b=0,\quad v_{1}=v_{2}=v_{3}=0,\quad u=0,\quad w=0,\quad a+c=0.
\]
Writing $a=k$ ($k\in\mathbb{Z}$) gives
\[
  \Lambda_{\phi_{18}}=\{\mu_{k}\mid k\in\mathbb{Z}\},\qquad
  \mu_{k}(t)=\mathrm{diag}\!\bigl(t^{k},\,t^{k},\,1,\,1,\,1,\,t^{-k},\,t^{-k}\bigr).
\]
Thus $\Lambda_{\phi_{18}}$ is symmetric, and by the Casimiro--Florentino criterion,
$\phi_{18}$ is polystable.

\paragraph{Normal form and component dimension.}
Put
\[
V:=\langle x_0,x_1\rangle,\qquad
U:=\langle x_2,x_3,x_4\rangle,\qquad
W:=\langle x_5,x_6\rangle.
\]
Then the $H$--fixed slice decomposes as
\[
W^H
=
\Sym^3 U
\;\oplus\;
(V\otimes U\otimes x_5)
\;\oplus\;
(V\otimes U\otimes x_6)
\;\cong\;
\Sym^3 U \;\oplus\; (V\otimes U\otimes W),
\]
and $C_G(H)=\{A\oplus B\oplus C:\ A\in GL(V),\ B\in GL(U),\ C\in GL(W),\ \det(A)\det(B)\det(C)=1\}$
acts blockwise (with the \emph{same} $B\in GL(U)$ acting simultaneously on the cubic block and on the mixed block).

\smallskip
\noindent
To avoid any ambiguity about this shared $GL(U)$--action, we normalize in two stages:
first we fix a convenient basis of $U$ using the cubic block, and \emph{then} we reduce the mixed block using only those $U$--changes that preserve the chosen cubic shape (namely permutations and diagonal scalings of $x_2,x_3,x_4$).

\smallskip
\noindent\textbf{(I) The ternary cubic block (fixing the $U$--basis).}
Write the cubic part as $F\in\Sym^3 U$.
On the dense open subset where $F$ is a smooth plane cubic with nonvanishing Hessian,
the classical Hesse normal form gives a change of basis $B\in GL(U)$ and a scalar
such that (after rescaling and renaming the coordinates on $U$) we may assume
\begin{equation}\label{eq:k18-hesse}
F
\;\sim\;
x_2^3+\sigma_1 x_3^3+\sigma_2 x_4^3+\rho\,x_2x_3x_4,
\qquad
(\sigma_1,\sigma_2,\rho)\in(\C^\times)^3.
\end{equation}
From this point on, the basis $(x_2,x_3,x_4)$ is \emph{fixed}.
The remaining $GL(U)$--freedom that preserves the $4$--term shape in \eqref{eq:k18-hesse}
is the monomial subgroup: permutations of $(x_2,x_3,x_4)$ and diagonal scalings
$x_i\mapsto t_i x_i$ (which only rescale $(\sigma_1,\sigma_2,\rho)$).

\smallskip
\noindent\textbf{(II)+(III) The mixed block as a $2\times 3$ matrix pencil (reduced relative to the fixed $U$--basis).}
Write the mixed part as
\[
x_5\,L_5 \;+\; x_6\,L_6,\qquad
L_5,L_6\in V\otimes U,
\]
and represent $L_5,L_6$ by $2\times 3$ matrices $M_5,M_6$ with respect to the fixed bases
$(x_0,x_1)$ of $V$ and $(x_2,x_3,x_4)$ of $U$:
\[
L_j=(x_0,x_1)\,M_j\,(x_2,x_3,x_4)^{\mathsf T}\qquad (j=5,6).
\]
The $GL(V)$--factor acts by left multiplication on both $M_5,M_6$,
the $GL(W)$--factor changes the basis of the pencil (so it replaces $(M_5,M_6)$ by another spanning pair),
and the residual monomial subgroup in $GL(U)$ permutes/rescales the three columns.

We work on the dense open set where the pencil $\langle M_5,M_6\rangle$ contains a rank-$2$
matrix and where suitable $2\times 2$ minors in the fixed column basis are nonzero (an open condition).
Using $GL(W)$ we choose the basis $(x_5,x_6)$ so that $M_5$ has rank $2$.
After possibly permuting $(x_2,x_3,x_4)$ (allowed by the monomial subgroup), we may assume
the $2\times 2$ minor of $M_5$ in the $(x_2,x_3)$--columns is invertible.
Then, multiplying on the left by its inverse (an element of $GL(V)$), we achieve
\[
M_5
\;\sim\;
\begin{pmatrix}
1&0&\alpha\\
0&1&0
\end{pmatrix},
\qquad
\alpha\in\C^\times,
\]
where the remaining entry in the $(x_4)$--column is recorded as the modulus $\alpha$
(we do \emph{not} use column additions, only the allowed row operations, column permutation/diagonal scaling, and basis change in $W$).

Next, keeping $M_5$ fixed, we use the subgroup of $GL(W)$ stabilizing the first basis vector $x_5$
(i.e.\ upper--triangular matrices) to replace $M_6$ by $M_6+\lambda M_5$, and use $GL(V)$--row operations
that preserve the pivot normalization of $M_5$.
On the same open set one can arrange that $M_6$ takes the form
\[
M_6
\;\sim\;
\begin{pmatrix}
1&0&0\\
0&\beta&\gamma
\end{pmatrix},
\qquad
(\beta,\gamma)\in(\C^\times)^2,
\]
after additionally allowing diagonal rescalings of $(x_2,x_3,x_4)$ (which remain inside the $4$--term family
\eqref{eq:k18-hesse} and merely rescale $(\sigma_1,\sigma_2,\rho)$).

Translating back to monomials, the mixed block becomes
\[
x_0x_2x_5+x_1x_3x_5+\alpha\,x_0x_4x_5
\;+\;
x_0x_2x_6+\beta\,x_1x_3x_6+\gamma\,x_1x_4x_6,
\qquad
(\alpha,\beta,\gamma)\in(\C^\times)^3.
\]

\smallskip
\noindent\textbf{Collecting (I)--(III).}
Combining \eqref{eq:k18-hesse} with the above reduction of the mixed block, we obtain the convenient normal form
\[
\nf_{18}
=
x_2^3+\sigma_1x_3^3+\sigma_2x_4^3+\rho\,x_2x_3x_4
+x_0x_2x_5+x_1x_3x_5+\alpha\,x_0x_4x_5
+x_0x_2x_6+\beta\,x_1x_3x_6+\gamma\,x_1x_4x_6,
\]
with $(\sigma_1,\sigma_2,\rho,\alpha,\beta,\gamma)\in(\C^\times)^6$. The corresponding component
$\Phi_{18}$ of the moduli is six-dimensional.

\subsection{Case $k=19$}

\paragraph{1-PS limit.}
Set
\[
  \lambda_{19}(t)=\mathrm{diag}\!\bigl(t^{2},\,t^{2},\,t^{2},\,1,\,t^{-1},\,t^{-1},\,t^{-4}\bigr),
  \qquad t\in\mathbb{G}_m.
\]
For a generic $f_{19}$ as in Section~3, the $1$-PS limit is
\[
\begin{aligned}
  \phi_{19}:=\lim_{t\to 0}\lambda_{19}(t)\cdot f_{19}
  &= a_{1}\,x_{3}^{3}
   + a_{2}\,x_{0}x_{4}^{2}+a_{3}\,x_{1}x_{4}^{2}+a_{4}\,x_{2}x_{4}^{2}\\
  &\quad + a_{5}\,x_{0}x_{4}x_{5}+a_{6}\,x_{1}x_{4}x_{5}+a_{7}\,x_{2}x_{4}x_{5}\\
  &\quad + a_{8}\,x_{0}x_{5}^{2}+a_{9}\,x_{1}x_{5}^{2}+a_{10}\,x_{2}x_{5}^{2}\\
  &\quad + a_{11}\,x_{0}^{2}x_{6}+a_{12}\,x_{0}x_{1}x_{6}+a_{13}\,x_{1}^{2}x_{6}\\
  &\quad + a_{14}\,x_{0}x_{2}x_{6}+a_{15}\,x_{1}x_{2}x_{6}+a_{16}\,x_{2}^{2}x_{6}.
\end{aligned}
\]

\paragraph{$H$ and $C_{G}(H)$.}
Let $H=\lambda_{19}(\mathbb{G}_m)$. The $H$-weights on
$\langle x_{0},\dots,x_{6}\rangle$ are $(2,2,2,0,-1,-1,-4)$ with block decomposition
\[
  \langle x_{0},x_{1},x_{2}\rangle\ \oplus\ \langle x_{3}\rangle\ \oplus\ \langle x_{4},x_{5}\rangle\ \oplus\ \langle x_{6}\rangle.
\]
Hence,
\[
\begin{aligned}
  C_{G}(H)
  &= \Bigl\{\,A\oplus \mathrm{diag}(\beta)\oplus B\oplus \mathrm{diag}(\delta):\
       A\in\mathrm{GL}(3),\ B\in\mathrm{GL}(2),\\
  &\hspace{9.8em}\beta,\delta\in\mathbb{G}_m,\ \det(A)\,\beta\,\det(B)\,\delta=1\,\Bigr\}\\
  &\cong \bigl(\mathrm{GL}(3)\times\mathrm{GL}(2)\times\mathbb{G}_m\times\mathbb{G}_m\bigr)\cap \mathrm{SL}(7),
\end{aligned}
\]
so $\dim C_{G}(H)=14$. Each monomial of $\phi_{19}$ has $H$-weight $0$.

\paragraph{Polystability (Luna + Casimiro--Florentino).}
By Luna's reduction, the closedness of the $\mathrm{SL}(7)$-orbit of $\phi_{19}$ is
equivalent to polystability for the $C_{G}(H)$-action on the $H$-fixed subspace.
After conjugating inside the blocks, any $\lambda\in Y\!\bigl(C_{G}(H)\bigr)$ may be taken with
\[
  \mathrm{wt}(x_{0},\dots,x_{6})=(s+u_{1},\,s+u_{2},\,s+u_{3},\,b,\,t+v,\,t-v,\,c),
\]
where $s,t,b,c,u_{i},v\in\mathbb{Z}$, $u_{1}+u_{2}+u_{3}=0$, and the $\mathrm{SL}(7)$-condition
(Convention~\ref{convention}) is
\[
  S:=3s+b+2t+c=0.
\]
A direct computation yields the positive identity
\begin{align}
  &\sum_{i=0}^{2}\Bigl[w\bigl(x_{i}x_{4}^{2}\bigr)
      + 2\,w\bigl(x_{i}x_{4}x_{5}\bigr)
      + w\bigl(x_{i}x_{5}^{2}\bigr)\Bigr]\notag\\
  &\quad + 2\Bigl[w(x_{0}^{2}x_{6})+w(x_{0}x_{1}x_{6})+w(x_{1}^{2}x_{6})
                   +w(x_{0}x_{2}x_{6})+w(x_{1}x_{2}x_{6})+w(x_{2}^{2}x_{6})\Bigr]\notag\\
  &\quad + 4\,w(x_{3}^{3}) \;=\; 12\,S. \label{eq:k19-identity}
\end{align}
If $\lambda\in\Lambda_{\phi_{19}}$, then all $16$ monomial weights are $\ge 0$ and $S=0$; by
\eqref{eq:k19-identity} they must all vanish. Solving gives
\[
  b=0,\quad v=0,\quad s+u_{r}+2t=0\ (r=1,2,3),\quad 2s+c=0.
\]
Using $u_{1}+u_{2}+u_{3}=0$ we obtain $s+2t=0$ and $u_{1}=u_{2}=u_{3}=0$.
Writing $s=2k$ with $k\in\mathbb{Z}$ yields
\[
  \Lambda_{\phi_{19}}=\{\mu_{k}\mid k\in\mathbb{Z}\},\qquad
  \mu_{k}(t)=\mathrm{diag}\!\bigl(t^{2k},\,t^{2k},\,t^{2k},\,1,\,t^{-k},\,t^{-k},\,t^{-4k}\bigr).
\]
Thus $\Lambda_{\phi_{19}}$ is symmetric, and by the Casimiro--Florentino criterion
$\phi_{19}$ is polystable; in particular $\mathrm{SL}(7)\cdot\phi_{19}$ is closed.

\paragraph{Normal form and component dimension.}
In the $H$-fixed subspace
\[
  W^{H}=\ \langle x_{3}^{3}\rangle\
     \oplus\ \langle x_{0},x_{1},x_{2}\rangle\otimes\mathrm{Sym}^{2}\langle x_{4},x_{5}\rangle\
     \oplus\ \mathrm{Sym}^{2}\langle x_{0},x_{1},x_{2}\rangle\otimes x_{6},
\]
using the $\mathrm{GL}(3)$-action on $\langle x_{0},x_{1},x_{2}\rangle$ to identify the basis of $\langle x_{0},x_{1},x_{2}\rangle$ with the standard basis $\{x_{4}^{2}, x_{4}x_{5}, x_{5}^{2}\}$ of $\mathrm{Sym}^{2}\langle x_{4},x_{5}\rangle$ (ensuring full rank), and using the residual stabilizer (isomorphic to $\mathrm{GL}(2)$) to normalize the third block, a generic element is brought to
\[
  \nf_{19}(\rho,\sigma)
  = x_{3}^{3} + (x_{0}x_{4}^{2}+x_{1}x_{4}x_{5}+x_{2}x_{5}^{2})
    + x_{6}(x_{0}^{2}+x_{2}^{2}+\rho\,x_{1}^{2}+\sigma\,x_{0}x_{2}),
\]
with $(\rho,\sigma)\in {\mathbb{C}^{\times}}^{2}$ general. As $\dim W^{H}=16$ and
$\dim C_{G}(H)=14$ (with $H$ acting trivially), the effective action has
dimension $13$; after projectivizing we obtain
\[
  \dim \Phi_{19} = 16-13-1=2.
\]
The residual $T$-stabilizer is finite; hence, the corresponding component
$\Phi_{19}$ of the moduli is two-dimensional.

\subsection{Case $k=20$}

\paragraph{1-PS limit.}
Set
\[
  \lambda_{20}(t)=\mathrm{diag}\!\bigl(t,\,t,\,t,\,t,\,1,\,t^{-2},\,t^{-2}\bigr),\quad t\in\mathbb{G}_m.
\]
For a generic $f_{20}$ as in Section~3, the $1$-PS limit is
\[
\begin{aligned}
\phi_{20}:=\lim_{t\to 0}\lambda_{20}(t)\cdot f_{20}
&= a_{1}\,x_{4}^{3}
+a_{2}\,x_{0}^{2}x_{5}+a_{3}\,x_{0}x_{1}x_{5}+a_{4}\,x_{1}^{2}x_{5}+a_{5}\,x_{0}x_{2}x_{5}+a_{6}\,x_{1}x_{2}x_{5}\\
&\quad +a_{7}\,x_{2}^{2}x_{5}+a_{8}\,x_{0}x_{3}x_{5}+a_{9}\,x_{1}x_{3}x_{5}+a_{10}\,x_{2}x_{3}x_{5}+a_{11}\,x_{3}^{2}x_{5}\\
&\quad+a_{12}\,x_{0}^{2}x_{6}+a_{13}\,x_{0}x_{1}x_{6}+a_{14}\,x_{1}^{2}x_{6}+a_{15}\,x_{0}x_{2}x_{6}+a_{16}\,x_{1}x_{2}x_{6}\\
&\quad +a_{17}\,x_{2}^{2}x_{6}+a_{18}\,x_{0}x_{3}x_{6}+a_{19}\,x_{1}x_{3}x_{6}+a_{20}\,x_{2}x_{3}x_{6}+a_{21}\,x_{3}^{2}x_{6}.
\end{aligned}
\]

\paragraph{$H$ and $C_{G}(H)$.}
Let $H=\lambda_{20}(\mathbb{G}_m)$. The $H$-weights on $\langle x_{0},\dots,x_{6}\rangle$ are
$(1,1,1,1,0,-2,-2)$ with block multiplicities
$\langle x_{0},x_{1},x_{2},x_{3}\rangle$, $\langle x_{4}\rangle$, $\langle x_{5},x_{6}\rangle$. Hence,
\[
\begin{aligned}
C_{G}(H)
&=\Bigl\{A\oplus\mathrm{diag}(\beta)\oplus C:\ A\in\mathrm{GL}(4),\ \beta\in\mathbb{G}_m,\
C\in\mathrm{GL}(2),\\
&\hspace{13.4em}\det(A)\,\beta\,\det(C)=1\Bigr\}\\
&\cong \bigl(\mathrm{GL}(4)\times\mathbb{G}_m\times\mathrm{GL}(2)\bigr)\cap \mathrm{SL}(7),
\end{aligned}
\]
so $\dim C_{G}(H)=20$. Every monomial of $\phi_{20}$ has $H$-weight $0$.

\paragraph{Polystability (Luna + Casimiro--Florentino).}
By Luna's reduction, the closedness of the $\mathrm{SL}(7)$-orbit of $\phi_{20}$ is equivalent
to polystability for the $C_{G}(H)$-action on the $H$-fixed subspace. After conjugating
inside the $\mathrm{GL}(4)$- and $\mathrm{GL}(2)$-blocks, any $\lambda\in Y(C_{G}(H))$ may be taken with weights
\[
\mathrm{wt}(x_{0},\dots,x_{6})=\bigl(a+u_{0},\,a+u_{1},\,a+u_{2},\,a+u_{3},\,b,\,c+w,\,c-w\bigr),
\]
where $a,b,c,u_{i},w\in\mathbb{Z}$, $u_{0}+u_{1}+u_{2}+u_{3}=0$, and, by Convention~\ref{convention},
\[
S:=4a+b+2c=0.
\]
Let $w(\cdot)$ denote the $\lambda$-weight of a monomial. Then
\[
\begin{aligned}
& w(x_{4}^{3})=3b,\\
& w(x_{i}^{2}x_{5})=2(a+u_{i})+c+w,\qquad
  w(x_{i}^{2}x_{6})=2(a+u_{i})+c-w\quad (i=0,1,2,3),\\
& w(x_{i}x_{j}x_{5})=2a+(u_{i}+u_{j})+c+w,\quad
  w(x_{i}x_{j}x_{6})=2a+(u_{i}+u_{j})+c-w\ (0\le i<j\le 3).
\end{aligned}
\]
A direct computation yields the positive identity
\begin{align}
&3\sum_{i=0}^{3}\bigl[w(x_{i}^{2}x_{5})+w(x_{i}^{2}x_{6})\bigr]
+3\sum_{0\le i<j\le 3}\bigl[w(x_{i}x_{j}x_{5})+w(x_{i}x_{j}x_{6})\bigr]\notag\\
&\qquad\qquad\qquad\qquad\qquad\qquad
+10\,w(x_{4}^{3})\;=\;30\,S. \label{eq:k20-identity}
\end{align}
If $\lambda\in\Lambda_{\phi_{20}}$, then all $21$ monomial weights are $\ge0$ and $S=0$; by
\eqref{eq:k20-identity} they must all vanish. Solving gives
\[
b=0,\quad w=0,\quad 2(a+u_{i})+c=0\ (i=0,1,2,3),\quad
2a+(u_{i}+u_{j})+c=0\ (i<j),
\]
Hence, $u_{0}=u_{1}=u_{2}=u_{3}=0$ and $2a+c=0$. Writing $a=k\in\mathbb{Z}$ yields
\[
\Lambda_{\phi_{20}}=\{\mu_{k}\mid k\in\mathbb{Z}\},\qquad
\mu_{k}(t)=\mathrm{diag}\!\bigl(t^{k},\,t^{k},\,t^{k},\,t^{k},\,1,\,t^{-2k},\,t^{-2k}\bigr).
\]
As $\Lambda_{\phi_{20}}$ is symmetric, $\phi_{20}$ is polystable by the
Casimiro--Florentino criterion; in particular, $\mathrm{SL}(7)\cdot\phi_{20}$ is closed.

\paragraph{Normal form and component dimension.}
We have
\[
W^{H}
=
\underbrace{\langle x_4^3\rangle}_{(\mathrm{I})}
\ \oplus\
\underbrace{\Sym^{2}\langle x_0,x_1,x_2,x_3\rangle\otimes \langle x_5\rangle}_{(\mathrm{II})}
\ \oplus\
\underbrace{\Sym^{2}\langle x_0,x_1,x_2,x_3\rangle\otimes \langle x_6\rangle}_{(\mathrm{III})},
\]
so a general element of $W^{H}$ is
\[
\varphi
=
a\,x_4^3
+ x_5 Q_5(x_0,\dots,x_3)
+ x_6 Q_6(x_0,\dots,x_3),
\qquad
Q_5,Q_6\in \Sym^{2}\langle x_0,\dots,x_3\rangle^{\vee}.
\]
Using the central torus factor together with projective scaling, we normalize $a=1$.

\medskip\noindent
\emph{Diagonalization of the quadratic pencil.}
Consider the pencil $\langle Q_5,Q_6\rangle$ and its discriminant
\[
\Delta(s,t):=\det(sQ_5+tQ_6),\qquad [s:t]\in \PP^{1}.
\]
On the dense open set where $\Delta$ has four distinct roots (equivalently, the pencil is regular, i.e.\ it lies off the discriminant locus), we may assume $Q_5$ is nondegenerate.
Acting by $\GL\langle x_0,x_1,x_2,x_3\rangle\cong \GL(4)$, we take
\[
Q_5 \sim q:=x_0^2+x_1^2+x_2^2+x_3^2.
\]
With respect to $q$, the residual group on $\langle x_0,\dots,x_3\rangle$ is $O(4)$.
For a regular pencil, the $q$-selfadjoint endomorphism $A$ defined by
$Q_6(u,v)=q(Au,v)$ has four distinct eigenvalues, hence after an $O(4)$-change
of coordinates
\[
Q_6 \sim \lambda_0 x_0^2+\lambda_1 x_1^2+\lambda_2 x_2^2+\lambda_3 x_3^2,
\qquad \lambda_i \ \text{pairwise distinct}.
\]

\medskip\noindent
\emph{Reduction to a one-parameter normal form (cross-ratio).}
The $\GL\langle x_5,x_6\rangle\cong \GL(2)$-action replaces $(Q_5,Q_6)$ by
\[
(Q_5',Q_6')=(\alpha Q_5+\beta Q_6,\ \gamma Q_5+\delta Q_6),
\qquad
\begin{pmatrix}\alpha&\beta\\ \gamma&\delta\end{pmatrix}\in\GL(2),
\]
and, after re-normalizing $Q_5'$ back to $q$, this acts on the eigenvalues by the
fractional-linear transformation
\[
\lambda\ \longmapsto\ \frac{\gamma+\delta\lambda}{\alpha+\beta\lambda}.
\]
Thus a regular pencil is determined (up to permuting the diagonal coordinates) by the
unordered set $\{\lambda_0,\lambda_1,\lambda_2,\lambda_3\}\subset\PP^{1}$ modulo $\mathrm{PGL}_2$,
i.e.\ by one cross-ratio parameter. Using $3$-transitivity of $\mathrm{PGL}_2$ on $\PP^{1}$,
we may normalize three eigenvalues to $1,-1,2$ and denote the fourth by $\tau$.
After permuting $(x_0,x_1,x_2,x_3)$ if necessary, we assume
\[
(\lambda_0,\lambda_1,\lambda_2,\lambda_3)=(1,\tau,-1,2),
\qquad
\tau\notin\{1,-1,2\}.
\]
Consequently, on the open set $\tau\in\C^{\times}\setminus\{1,-1,2\}$, a convenient normal form is
\[
\nf_{20}(\tau)
=
x_4^3
+ x_5(x_0^2+x_1^2+x_2^2+x_3^2)
+ x_6(x_0^2+\tau x_1^2- x_2^2+2x_3^2),
\qquad
\tau\in\C^{\times}\setminus\{1,-1,2\}.
\]
(For the excluded values, the pencil acquires repeated eigenvalues and the stabilizer jumps; these form a proper closed subset.)

\medskip\noindent
\emph{Component dimension.}
Here $\dim W^{H}=21$ and $\dim C_G(H)=20$, and $C_G(H)$ contains a one-dimensional central torus acting trivially on $W^{H}$.
Hence the effective action has dimension $19$, and after projectivizing we obtain
\[
\dim(\Phi_{20})=\dim W^{H}-19-1=1,
\]
in agreement with the single modulus $\tau$ for a regular pencil.

\subsection{Case $k=21$}

\paragraph{1-PS limit.}
Set
\[
  \lambda_{21}(t)=\operatorname{diag}\!\bigl(t,\,t,\,1,\,1,\,1,\,1,\,t^{-2}\bigr),
  \qquad t\in\mathbb{G}_m.
\]
For a generic $f_{21}$ as in Section~3, the $1$-PS limit is
\[
  \phi_{21}
  := \lim_{t\to0}\lambda_{21}(t)\cdot f_{21}
  = F(x_{2},x_{3},x_{4},x_{5}) + x_{6}\,Q(x_{0},x_{1}),
\]
where $F\in \Sym^{3}\langle x_{2},x_{3},x_{4},x_{5}\rangle$ and $Q\in \Sym^{2}\langle x_{0},x_{1}\rangle$.

\paragraph{$H$ and $C_G(H)$.}
Let $H=\lambda_{21}(\mathbb{G}_m)$.
The $H$-weights on $\langle x_{0},\dots,x_{6}\rangle$ are $(1,1,0,0,0,0,-2)$ with block decomposition
$\langle x_{0},x_{1}\rangle\oplus\langle x_{2},x_{3},x_{4},x_{5}\rangle\oplus\langle x_{6}\rangle$.
Hence,
\[
\begin{aligned}
  C_G(H)
  &=\Bigl\{\ A\oplus B\oplus \gamma \ :\ A\in\mathrm{GL}(2),\ B\in\mathrm{GL}(4),\ \gamma\in\mathbb{G}_m,\ 
     \det(A)\det(B)\gamma=1\ \Bigr\}\notag\\
  &\cong \bigl(\mathrm{GL}(2)\times \mathrm{GL}(4)\times \mathbb{G}_m\bigr)\cap \mathrm{SL}(7),
  \qquad \dim C_G(H)=20.
  \label{eq:k21-centralizer}
\end{aligned}
\]
Every monomial of $\phi_{21}$ has $H$-weight $0$, so $\phi_{21}\in W^{H}$.

\paragraph{Polystability (Luna + Casimiro--Florentino).}
By Luna's reduction, the closedness of the $\mathrm{SL}(7)$-orbit of $\phi_{21}$ is equivalent to polystability for the $C_G(H)$-action on $W^{H}$.
After conjugating inside the $\mathrm{GL}(2)$- and $\mathrm{GL}(4)$-blocks, any $\lambda\in Y\!\bigl(C_G(H)\bigr)$ may be taken with
\[
  \mathrm{wt}(x_{0},\ldots,x_{6})=(\alpha+u,\,\alpha-u,\,\beta+v_1,\,\dots,\,\beta+v_4,\,\gamma),
\]
where $\sum v_i=0$, and with the $\mathrm{SL}$-constraint $S:=2\alpha+4\beta+\gamma=0$.

Summing the $\lambda$-weights of the $20$ cubic monomials in $F$ gives
\[
w(F)=15\bigl((\beta+v_1)+(\beta+v_2)+(\beta+v_3)+(\beta+v_4)\bigr)=60\beta,
\]
while the weights of the three quadratic monomials in $x_6Q$ sum to
\[
w(x_6Q)=(\gamma+2(\alpha+u))+(\gamma+2\alpha)+(\gamma+2(\alpha-u))=3\gamma+6\alpha.
\]
Using $S=0$, we obtain the positive linear identity
\begin{equation}\label{eq:k21-identity}
  w(F) + 5\,w(x_6Q)
  \;=\;
  60\beta + 5(3\gamma+6\alpha)
  \;=\;
  15(2\alpha+4\beta+\gamma)
  \;=\;
  15\,S.
\end{equation}
If $\lambda\in\Lambda_{\phi_{21}}$, then every monomial weight occurring in $F$ and $x_6Q$ is $\ge 0$ and $S=0$; hence \eqref{eq:k21-identity} forces
\[
w(F)=0,\qquad w(x_6Q)=0,
\]
so all those monomial weights vanish individually.
From the cubic monomials $x_2^3,\dots,x_5^3$ we get $\beta+v_i=0$ for all $i$, hence $\beta=0$ and $v_i=0$.
From the quadratic monomials $x_6x_0^2$ and $x_6x_1^2$ we get
\[
\gamma+2\alpha+2u=0,\qquad \gamma+2\alpha-2u=0,
\]
hence $u=0$ and $\gamma+2\alpha=0$.
Therefore $\lambda$ is of the form
\[
\lambda(t)=\operatorname{diag}\!\bigl(t^{k},\,t^{k},\,1,\,1,\,1,\,1,\,t^{-2k}\bigr)
=\lambda_{21}(t^{k})
\qquad (k\in\mathbb{Z}),
\]
and consequently
\[
\Lambda_{\phi_{21}}=\{\lambda_{21}(t^{k})\mid k\in\mathbb{Z}\}.
\]
By Lemma~\ref{lem:rank1-symmetry}, $\Lambda_{\phi_{21}}$ is symmetric; hence, by the Casimiro--Florentino criterion, $\phi_{21}$ is polystable.

\paragraph{Normal form and component dimension.}
We normalize generic elements under the action of the centralizer $C_G(H)$.

\medskip\noindent
(1) \textbf{Quadratic part.} Under the $\mathrm{GL}(2)$-action on $\langle x_0, x_1\rangle$, a generic binary quadratic $Q$ is normalized to $x_0^2 + x_1^2$. The stabilizer of this form contains the orthogonal group $\mathrm{O}(2)$, contributing \textbf{1 dimension} to the stabilizer of the generic point.

\medskip\noindent
(2) \textbf{Cubic part.} Under the $\mathrm{GL}(4)$-action on $\langle x_2, x_3, x_4, x_5\rangle$, a generic cubic surface $F$ can be written as the sum of five cubes of linear forms (\textbf{Sylvester's Pentahedron Theorem}). Let $L_0, \dots, L_4$ be fixed linear forms defining a generic pentahedron in $\mathbb{C}^4$. Then $F = \sum_{i=0}^4 c_i L_i^3$.

\medskip\noindent
Thus, the normal form is:
\[
\nf_{21}(c) = \sum_{i=0}^4 c_i L_i(x_2, x_3, x_4, x_5)^3 + x_6(x_0^2 + x_1^2),
\]
where the parameters are $[c_0 : \dots : c_4] \in \mathbb{P}^4$.

\medskip\noindent
The dimension of the component is computed as follows:
$\dim W^H = 20 (\text{cubic}) + 3 (\text{quadratic}) = 23$.
Since $H$ acts trivially on $W^H$, the effective group is $C_G(H)/H$, of dimension $20-1=19$.
Since the generic stabilizer has dimension $1$ (due to $\mathrm{O}(2)$), the orbit dimension is $19 - 1 = 18$.
\[
\dim(\Phi_{21}) = 23 - 18 - 1 = 4.
\]

\begin{table}[htbp]
\centering
\small
\setlength{\tabcolsep}{4pt}
\renewcommand{\arraystretch}{1.2}
\caption{Summary of closed orbit representatives, criteria used, and component dimensions for $k=1,\dots,21$. C-H means the Convex-hull criterion and C-F means the Casimiro--Florentino criterion.}
\label{table:normalforms}
\label{tab:case-summary}
\begin{tabular}{c|c|p{0.58\linewidth}|c|c}
\hline
$k$ & Criterion & Normal form & Parameters & Dim \\
\hline
1 & C-H &
$ x_{2}x_{3}^{2}+x_{1}x_{3}x_{4}+x_{2}^{2}x_{5}+x_{0}x_{5}^{2}+x_{1}^{2}x_{6}+x_{0}x_{4}x_{6}$ &
none & $0$ \\
\hline
2 & C-H &
$ x_{2}^{2}x_{4}+x_{1}x_{4}^{2}+x_{1}x_{3}x_{5}+x_{0}x_{5}^{2}+x_{1}x_{2}x_{6}+x_{0}x_{3}x_{6}$ &
none & $0$ \\
\hline
3 & C-F &
$ x_{1}x_{3}^{2}+x_{1}x_{3}x_{4}+x_{1}x_{4}^{2}+x_{2}^{2}x_{5}+x_{0}x_{5}^{2}+x_{1}x_{2}x_{6}+x_{0}x_{3}x_{6}
+x_{0}x_{4}x_{6}$ &
none & $0$ \\
\hline
4 & C-H &
$ x_{3}^{3}+x_{2}x_{3}x_{4}+x_{1}x_{4}^{2}+x_{2}^{2}x_{5}+x_{1}x_{3}x_{5}+x_{0}x_{4}x_{5}+\alpha x_{1}x_{2}x_{6}
+\beta x_{0}x_{3}x_{6}$ &
$\alpha,\beta \in \mathbb{C}^{\times}$ & $2$ \\
\hline
5 & C-H &
$ x_{3}^{3}+x_{2}x_{3}x_{4}+x_{1}x_{4}^{2}+x_{2}^{2}x_{5}+x_{1}x_{3}x_{5}+x_{0}x_{5}^{2}+\alpha x_{1}^{2}x_{6}
+\beta x_{0}x_{3}x_{6}$ &
$\alpha,\beta \in \mathbb{C}^{\times}$ & $2$ \\
\hline
6 & C-H &
$ x_{2}x_{4}^{2}+x_{2}^{2}x_{5}+x_{1}x_{3}x_{5}+x_{0}x_{4}x_{5}+x_{1}^{2}x_{6}+x_{0}x_{3}x_{6}$ &
none & $0$ \\
\hline
7 & C-H &
$ x_{3}^{2}x_{4}+x_{1}x_{4}^{2}+x_{2}x_{3}x_{5}+x_{0}x_{5}^{2}+x_{1}^{2}x_{6}+x_{0}x_{2}x_{6}$ &
none & $0$ \\
\hline
8 & C-F &
$ x_{3}^{3}+x_{0}x_{4}^{2}+x_{0}x_{5}^{2}+x_{0}x_{6}^{2}+x_{1}^{2}x_{4}
+\rho\,x_{1}x_{2}x_{5}+\sigma\,x_{2}^{2}x_{6}$ &
$(\rho,\sigma)\in(\mathbb{C}^{\times})^{2}$ & $2$ \\
\hline
9 & C-F &
$ x_{2}^{2}x_{3}+\tau\,x_{2}x_{3}^{2}+x_{0}x_{4}^{2}+\rho\,x_{0}x_{5}^{2}
+x_{1}x_{4}^{2}+x_{1}x_{5}^{2}+x_{0}x_{2}x_{6}+x_{1}x_{3}x_{6}$ &
$(\tau,\rho)\in(\mathbb{C}^{\times})^{2}$ & $2$ \\
\hline
10 & C-F &
$ x_{2}^{3}+\tau\,x_{2}^{2}x_{3}+\rho\,x_{2}x_{3}^{2}+x_{3}^{3}
+ x_{1}x_{2}x_{5}+x_{1}x_{3}x_{4}
+ x_{0}x_{4}x_{5}+x_{0}x_{6}^{2}$ &
$(\tau,\rho)\in(\mathbb{C}^{\times})^{2}$ & $2$ \\
\hline
11 & C-F &
$\sum_{i=0}^4 c_i L_i(x_1, x_2, x_3, x_4)^3 + x_0(x_5^2 + x_6^2)$ &
$(c_0:c_1:c_2:c_3:c_4) \in \mathbb{P}^{4}$ & $4$ \\
\hline
12 & C-F &
$ x_{1}x_{4}^{2}+x_{2}^{2}x_{5}+x_{3}^{2}x_{5}+x_{0}x_{4}x_{5}+x_{1}^{2}x_{6}+x_{0}x_{2}x_{6}$ &
none & $0$ \\
\hline
13 & C-F &
$ x_{3}^{3}+x_{0}x_{4}^{2}+x_{0}x_{5}^{2}+x_{1}x_{3}x_{4}
+\rho\,x_{2}x_{3}x_{5}+x_{1}^{2}x_{6}+\sigma\,x_{2}^{2}x_{6}+x_{0}x_{3}x_{6}$ &
$(\rho,\sigma)\in(\mathbb{C}^{\times})^{2}$ & $2$ \\
\hline
14 & C-F &
$x_{2}^{3}+ x_{0}\bigl(x_{3}^{2}+x_{4}^{2}+x_{5}^{2}+x_{6}^{2}\bigr) + x_{1}\bigl(x_{3}^{2}+\tau x_{4}^{2}-x_{5}^{2}+2x_{6}^{2}\bigr)$
&
$\tau\in\mathbb{C}^{\times}\setminus\{-1,1,2\}$ & $1$\\
\hline
15 & C-F &
$ x_{3}^{2}x_{4}+\tau\,x_{3}x_{4}^{2}+x_{0}x_{3}x_{5}+x_{0}x_{4}x_{6}
+x_{1}^{2}x_{5}+\rho\,x_{2}^{2}x_{6}$ &
$(\tau,\rho)\in(\mathbb{C}^{\times})^{2}$ & $2$ \\
\hline
16 & C-F &
$ x_{2}^{3}+x_{3}^{3}+\sigma x_{4}^{3}+\rho\,x_{2}x_{3}x_{4}
+x_{1}x_{2}x_{5}+x_{0}x_{3}x_{6}+x_{0}x_{5}^{2}+x_{1}^{2}x_{6}
+\kappa\,x_{2}^{2}x_{3}+\nu\,x_{2}x_{3}^{2}+\mu\,x_{3}^{2}x_{4}+\lambda\,x_{2}^{2}x_{4}$ &
$(\sigma,\rho,\kappa,\nu,\mu,\lambda)\in(\mathbb{C}^{\times})^{6}$ & $6$ \\
\hline
17 & C-F &
$ x_{3}^{2}x_{4}+\tau\,x_{3}x_{4}^{2}+x_{0}x_{3}x_{5}+x_{1}x_{4}x_{5}
+x_{0}^{2}x_{6}+x_{1}^{2}x_{6}+\rho\,x_{2}^{2}x_{6}$ &
$(\tau,\rho)\in(\mathbb{C}^{\times})^{2}$ & $2$ \\
\hline
18 & C-F &
$ x_{2}^{3}+\sigma_{1}x_{3}^{3}+\sigma_{2}x_{4}^{3}+\rho\,x_{2}x_{3}x_{4}
+x_{0}x_{2}x_{5}+x_{1}x_{3}x_{5}+\alpha\,x_{0}x_{4}x_{5}
+x_{0}x_{2}x_{6}+\beta\,x_{1}x_{3}x_{6}+\gamma\,x_{1}x_{4}x_{6}$ &
$(\sigma_{1},\sigma_{2},\rho,\alpha,\beta,\gamma)\in(\mathbb{C}^{\times})^{6}$ & $6$ \\
\hline
19 & C-F &
$x_{3}^{3} + (x_{0}x_{4}^{2}+x_{1}x_{4}x_{5}+x_{2}x_{5}^{2})
    + x_{6}(x_{0}^{2}+x_{2}^{2}+\rho\,x_{1}^{2}+\sigma\,x_{0}x_{2})$ &
$(\rho, \sigma) \in (\mathbb{C}^{\times})^{2}$ & $2$ \\
\hline
20 & C-F &
$x_4^3
+ x_5(x_0^2+x_1^2+x_2^2+x_3^2)
+ x_6(x_0^2+\tau x_1^2- x_2^2+2x_3^2)$ &
$\tau\in\mathbb{C}^{\times}\setminus\{-1,1,2\}$ & $1$ \\
\hline
21 & C-F &
$\sum_{i=0}^4 c_i L_i(x_2, x_3, x_4, x_5)^3 + x_6(x_0^2 + x_1^2)$  &
$(c_0:c_1:c_2:c_3:c_4) \in \mathbb{P}^{4}$  & $4$ \\
\hline
\end{tabular}
\end{table}

\newpage

\section{Singular loci of 21 polystable cubic fivefolds}\label{sec:5}

In this section, we determine the singular loci of general closed--orbit representatives 
$\nf_k$ constructed in Section~\ref{sec:closed-orbit}. 
For each $k=1,\dots,21$, we set 
\[
X_k := V(\nf_k) \subset \mathbb{P}^6
\]
and compute
\[
\Sing(X_k) \;=\; V\big(J(\nf_k)\big) \subset \mathbb{P}^6,
\qquad
J(\varphi) := \left(\frac{\partial \varphi}{\partial x_0},\dots,\frac{\partial \varphi}{\partial x_6}\right):(x_0,\dots,x_6)^\infty,
\]
i.e.\ the scheme cut out by the saturated Jacobian ideal. 
We give a set--theoretic description of $\Sing(X_k)$ in each case, and when isolated points occur, 
we also record the corank and the local invariants (Milnor and Tjurina number, which agree 
for our quasi--homogeneous normal forms). 
Table~3 presents a compact summary---listing the type and degree of the top--dimensional part and indicating the presence (or absence) of isolated points.  
Detailed case--by--case statements are recorded as Propositions~\ref{start}--\ref{end}.

Note that for the boundary families with moduli (parameters), the descriptions of the singular loci in the following propositions and the minimal exponent computations in Section 6 apply to general parameters. At certain special values of the parameters, the singularities may degenerate further, and the minimal exponent may strictly drop.

\subsection*{Isolated singularities appearing in this paper}
We define the isolated singularities that appear in the investigation of the singularities of $\nf_{k}$. Since we could not locate an explicit normal form in the standard tables (e.g. \cite{AGV12}), we fix notation here.

\begin{definition}\label{def:QH1}
We write $\mathrm{QH}(9,6,5;24)_{19}$ for the isolated hypersurface singularity analytically equivalent to 
$X^{2}Y + Y^{4} + XZ^{3}$. 
It is quasi-homogeneous of total degree $24$ with respect to weights $(w_X,w_Y,w_Z)=(9,6,5)$. Its corank, Milnor number and Tjurina number are $3$, $19$ and $19$ respectively.
\end{definition}

\begin{definition}\label{def:QH2}
We write $\mathrm{QH}(2,3,5;12)_{21}(t)$ for the isolated hypersurface singularity analytically equivalent to 
$ t X^{6}+X^{3}Y^{2}+X^{2}YZ+XZ^{2}+Y^{4}$, where $t \in \mathbb{C}^{\times}$ is a parameter. 
It is quasi-homogeneous of total degree $12$ with respect to weights $(w_X,w_Y,w_Z)=(2,3,5)$. Its corank, Milnor number and Tjurina number are $3$, $21$ and $21$ respectively.
\end{definition}

\begin{definition}
We write $\mathrm{QH}(3,3,4;12)_{18}(t)$ for the isolated hypersurface singularity analytically equivalent to 
$Z^3+X^4+t X^2 Y^2+Y^4$.
It is quasi-homogeneous of total degree $12$ with respect to weights $(w_X,w_Y,w_Z)=(3,3,4)$. Its corank, Milnor number and Tjurina number are $3$, $18$ and $18$ respectively.
\end{definition}

\begin{definition}\label{QH25}
We write $\mathrm{QH}(1,1;6)_{25}$ for the isolated hypersurface singularity
analytically equivalent to a homogeneous sextic
\[
\Psi(X,Y)=\prod_{i=1}^{6} L_i(X,Y),
\]
where $L_1,\dots,L_6$ are pairwise distinct linear forms in $(X,Y)$.
Equivalently, after a linear change of coordinates, one may write
$\Psi(X,Y)=\prod_{i=1}^6 (X-\lambda_i Y)$ with $\lambda_i$ pairwise distinct in
$\mathbb{P}^1$.

It is quasi-homogeneous of total degree $6$ with respect to weights
$(w_X,w_Y)=(1,1)$.
Its corank, Milnor number and Tjurina number are $2$, $25$ and $25$ respectively.
\end{definition}

\begin{definition}
We write $QH(1,1,1,1;3)_{16}([c])$ for the isolated hypersurface singularity
analytically equivalent to a homogeneous cubic
\[
G(X_1,X_2,X_3,X_4)=\Sigma_{i=0}^{4}c_{i}L_{i}^3
\]
where $L_{i}=L_{i}(X_1,X_2,X_3,X_4)$ are generic linear forms and $[c]=(c_0:c_1:c_2:c_3:c_4) \in \mathbb{P}^{4}$.  
Its corank, Milnor number and Tjurina number are $4$, $16$ and $16$. It has four-dimensional moduli.
\end{definition}

\begin{definition}
$\mathrm{QH}(1,2,2;6)_{20}(t)$
\[
F(X,Y,Z) = Y^3 + Z^3 + X^6 
    + t_1 X^2 Y^2 + t_2 X^2 Y Z + t_3 Y^2 Z + t_4 X^4 Y + t_5 X^4 Z,
\]
where $t=(t_{1},t_{2},\cdots,t_{5})$.
Its corank, Milnor number and Tjurina number are $3$, $20$ and $20$.
\end{definition}

\begin{remark}[Remark for Definition~\ref{QH25}]
The germ $\Psi(X,Y)=\prod_{i=1}^{6}L_i(X,Y)$ with pairwise distinct linear forms
defines an \emph{ordinary sextuple point} (also called an ordinary $6$-fold point):
it has $6$ smooth branches meeting transversely at the origin, with $6$ distinct
tangent directions in $\mathbb{P}^1$.
\end{remark}

We will frequently use the following shorthand labels for isolated singularity types:
\[
\begin{array}{c|c|c}
\text{Shorthand Label} & Label & \text{Defining equation} \\ \hline
\mathsf{E}_{19} & \mathrm{QH}(9,6,5;24)_{19} & X^{2}Y + Y^{4} + XZ^{3} = 0\\ 
\mathsf{E}_{21} & \mathrm{QH}(2,3,5;12)_{21}(t)  &   t X^{6}+X^{3}Y^{2}+X^{2}YZ+XZ^{2}+Y^{4} = 0\\
\mathsf{E}_{18} & \mathrm{QH}(3,3,4;12)_{18}(t)  &  Z^3+X^4+t X^2 Y^2+Y^4=0 \\ 
\mathcal{O}_{6} & \mathrm{QH}(1,1;6)_{25} & \prod_{i=1}^{6} (X-\lambda_{i}Y) = 0 \\
\mathcal{C}_{cub} & \mathrm{QH}(1,1,1,1;3)_{16}([c]) & \Sigma_{i=0}^{4}c_{i}L_{i}^3=0 \\
\mathsf{E}_{20} & \mathrm{QH}(1,2,2;6)_{20}(t) &  \scriptstyle{Y^3 + Z^3 + X^6 
    + t_1 X^2 Y^2 + t_2 X^2 Y Z + t_3 Y^2 Z + t_4 X^4 Y + t_5 X^4 Z=0}\\

\hline
\end{array}
\]

\subsection{Case $k=1$}

\begin{proposition}\label{start}
Let $X_1=V(\nf_{1})\subset\mathbb{P}^6$. The set-theoretic singular locus is
\[
  \Sing(X_1)= C \cup \{P\},
\]
where
\[
  C=\{\,x_0=x_1=x_2=x_3=0,\; x_5^{\,2}+x_4x_6=0\,\}
\]
is a smooth conic in the plane $\{x_0=x_1=x_2=x_3=0\}\cong\mathbb{P}^2$ (hence, $\deg C=2$), and
\[
  P=(1:0:0:0:0:0:0)
\]
is an isolated singular point.

The isolated singular point is of type $\mathsf{E}_{19}$.
\end{proposition}

\subsection{Case $k=2$}
\begin{proposition}
Let $X_2 = V(\nf_{2}) \subset \mathbb{P}^{6}$. Then
\[
\Sing(X_2)= L_{01}\cup C,
\]
with $L_{01}=V(x_2,x_3,x_4,x_5,x_6)\simeq \mathbb{P}^1$, and 
\[
C = V(x_0,x_1,x_2,\;x_5^2+x_3x_6,\;x_4^2+x_3x_5)\subset \Pi,\quad
\Pi=\{x_0=x_1=x_2=0\}\simeq\mathbb{P}^3.
\]
Here $C$ is a complete intersection $\mathrm{CI}(2,2)$ of degree $4$, with $L_{01}\cap C=\varnothing$.
\end{proposition}

\subsection{Case $k=3$}
\begin{proposition}
Let $X_3 = V(\nf_{3}) \subset \mathbb{P}^{6}$. Then
\[
\Sing(X_3)=L_{01}\cup C,
\]
with $L_{01}=V(x_2,x_3,x_4,x_5,x_6)\simeq \mathbb{P}^1$, and 
\[
C = V(x_0,x_1,x_2, x_5^2+x_3x_6+x_4x_6, x_3^2+x_3x_4+x_4^2)\subset \Pi,\quad
\Pi=\{x_0=x_1=x_2=0\}\simeq\mathbb{P}^3.
\]
Here $C$ is a complete intersection $\mathrm{CI}(2,2)$ of degree $4$, with $L_{01}\cap C=\varnothing$.

\end{proposition}

\subsection{Case $k=4$}
\begin{proposition}
Let $X_4 = V(\nf_{4}) \subset \mathbb{P}^{6}$. Then
\[
\Sing(X_4)=L_{01}\cup L_{56},
\]
with $L_{01}=V(x_2,x_3,x_4,x_5,x_6)\simeq \mathbb{P}^1$ and 
$L_{56}=V(x_0,x_1,x_2,x_3,x_4)\simeq \mathbb{P}^1$.
In particular, $L_{01}\cap L_{56}=\varnothing$ and no isolated singular points occur.
\end{proposition}

\subsection{Case \(k=5\)}

\begin{proposition}\label{prop:k5}
Let $X_5 = V(\nf_{5}) \subset \mathbb{P}^{6}$. Then
\[
\Sing(X_{5})=\{ P \} \cup \{ Q \},
\]
where $P=(1:0:0:0:0:0:0)$, $Q=(0:0:0:0:0:0:1)$ are isolated singular points.

Then $(X_5,P)$ and $(X_5,Q)$ are isolated hypersurface singularities of corank $3$,
and they are analytically equivalent.
The isolated singular points $P,Q$ are of type $\mathsf{E}_{21}$ and have same parameter $t$.
\end{proposition}

\subsection{Case $k=6$}
\begin{proposition}\label{k6}
Let $X_6 = V(\nf_{6}) \subset \mathbb{P}^{6}$. Then
\[
\Sing(X_6)= C \cup L_{56},\quad L_{56}=V(x_0,x_1,x_2,x_3,x_4),
\]
\[
C=V(x_4,x_5,x_6,\;x_2^2+x_1x_3,\;x_1^2+x_0x_3)\subset \Pi=\{x_4=x_5=x_6=0\}.
\]
$C$ is a $\mathrm{CI}(2,2)$ of degree $4$; $C\cap L_{56}=\varnothing$.
\end{proposition}

\subsection{Case $k=7$}
\begin{proposition}\label{k7}
Let $X_7 = V(\nf_{7}) \subset \mathbb{P}^{6}$. Then
\[
\Sing(X_7)=C \cup \{ P \},
\]
where
\[
C=V(x_3,x_4,x_5,x_6,\;x_1^2+x_0x_2),\quad P=(0:0:0:0:0:0:1).
\]
The isolated singular point $P$ is of type $\mathsf{E}_{19}$.
\end{proposition}

\subsection{Case $k=8$}

\begin{proposition}\label{prop:case-k8}
Let $X_8 := V(\nf_8)\subset \mathbb{P}^6$.
Then the singular locus has two irreducible components:

$$\Sing(X_8)=C\cup\{P\}$$, where
$$
C:=V\!\bigl(x_0,x_1,x_2,x_3,\;x_4^{2}+x_5^{2}+x_6^{2}\bigr)
\subset \Pi:=\{x_0=x_1=x_2=x_3=0\}\cong \mathbb{P}^2
$$
is a (smooth) conic of degree $2$ (hence $\mathrm{codim}_{\mathbb{P}^6}C=5$), and
$$
P:=V\!\bigl(x_1,x_2,x_3,x_4,x_5,x_6\bigr)=(1:0:0:0:0:0:0)
$$
is an isolated point (hence $\mathrm{codim}_{\mathbb{P}^6}P=6$, $\deg P=1$).
In particular, $C\cap\{P\}=\varnothing$.
The isolated singular point $P$ is of type $\mathsf{E}_{18}$.
\end{proposition}

\subsection{Case $k=9$}

\begin{proposition}\label{prop:case-k9}
Let $X_{9} := V(\nf_{9})\subset \mathbb{P}^{6}$.
Then the singular locus has two irreducible components:
$$
\Sing(X_9)=L_{01} \cup \{ P \},
$$
where $L_{01}=V(x_2,x_3,x_4,x_5,x_6)$ and $P=(0:0:0:0:0:0:1)$.
The isolated point $P$ is of type $\mathcal{O}_{6}$.
\end{proposition}

\subsection{Case $k=10$}

\begin{proposition}\label{prop:case-k10}
Let $X_{10} := V(\nf_{10})\subset \mathbb{P}^{6}$.
Then the singular locus has two irreducible components:
$$
\Sing(X_{10})=L_{01} \cup C,
$$
where $L_{01}=V(x_2,x_3,x_4,x_5,x_6)$ and $C=V(x_0,x_1,x_2,x_3,x_{4}x_{5}+x_{6}^{2})$ which is a curve of degree $2$. We have $L_{01} \cap C = \varnothing$.
\end{proposition}

\subsection{Case $k=11$}

\begin{proposition}\label{prop:case-k11}
Let $X_{11} := V(\nf_{11})\subset \mathbb{P}^{6}$.
Then the singular locus has three irreducible components:
$$
\Sing(X_{11})=\{ P_{1} \} \cup \{ P_{2} \} \cup \{ P_{3} \},
$$
where $P_{1},P_{2}$ and $P_{3}$ are isolated singular points of type $\mathcal{C}_{cub}$.
\end{proposition}

\subsection{Case $k=12$}

\begin{proposition}\label{prop:case-k12}
Let $X_{12} := V(\nf_{12})\subset \mathbb{P}^{6}$.
Then the singular locus has two irreducible components:
$$
\Sing(X_{12})=L_{56} \cup C,
$$
where $L_{56}=V(x_0,x_1,x_2,x_3,x_4)$ is a projective line and $C=V(x_2^2+x_3^2,x_1^2+x_0 x2,x_4,x_5,x_6)$ is a $(2,2)$-complete intersection. We have $L_{56} \cap C = \varnothing$.
\end{proposition}

\subsection{Case $k=13$}

\begin{proposition}\label{prop:case-k13} 
Let $X_{13} := V(\nf_{13}(\rho,\sigma))\subset \mathbb{P}^{6}$.
Then the singular locus has two connected components:
$$
\Sing(X_{13})=C_{1} \cup C_{2},
$$
where $C_{1} = V(x_{1}^2+\sigma x_{2}^2,x_{3},x_{4},x_{5},x_{6})$ and $C_{2} = V(x_{0},x_{1},x_{2},x_{3},x_{4}^2+x_{5}^2)$. Each $C_{i}$ consists of two lines which intersect at one point. We have $C_{1} \cap C_{2} = \varnothing$.
\end{proposition}

\subsection{Case $k=14$}

\begin{proposition}\label{prop:case-k14} 
Let $X_{14} := V(\nf_{14}(\tau))\subset \mathbb{P}^{6}$.
Then the singular locus has two irreducible components:
$$
\Sing(X_{14})=L_{01} \cup C,
$$
where $L_{01}= V(x_{2},x_{3},x_{4},x_{5},x_{6})$ and 
\[
C = V(x_{0},x_{1},x_{2},(\tau-1)x_{4}^2-2x_{5}^2+x_{6}^2,(\tau-1)x_{3}^2+(\tau+1)x_{5}^2+(\tau-2)x_{6}^2).
\]
We have $L_{01} \cap C = \varnothing$.
\end{proposition}

\subsection{Case $k=15$}

\begin{proposition}\label{prop:case-k15}
Let $X_{15} := V(\nf_{15}(\tau,\sigma))\subset \mathbb{P}^{6}$.
Then the singular locus consists of three irreducible components:
$$
\Sing(X_{15})=C_{1} \cup C_{2} \cup L,
$$
where
$C_{1} = V(x_{1}, x_{3}, x_{6}, \tau x_{4}^2 + x_{0}x_{5}, \sigma x_{2}^2 + x_{0}x_{4})$,
$C_{2} = V(x_{2}, x_{4}, x_{5}, x_{3}^2 + x_{0}x_{6}, x_{1}^2 + x_{0}x_{3})$,
and $L = V(x_{0},x_{1},x_{2},x_{3},x_{4})$.
Here, $C_{1}$ and $C_{2}$ are curves of degree $4$ (complete intersections of two quadrics in $\mathbb{P}^3$), and $L$ is a line.
\end{proposition}

\subsection{Case $k=16$}
\begin{proposition}\label{prop:case-k16}
Let $X_{16} := V(\nf_{16})\subset \mathbb{P}^{6}$.
Then the singular locus has two irreducible components:
$$
\Sing(X_{16})=\{ P \} \cup \{ Q \},
$$
$P=(1:0:0:0:0:0:0)$, $Q=(0:0:0:0:0:0:1)$,
where $P,Q$ are isolated singular points of type $\mathsf{E}_{20}$ with the same modulus.
\end{proposition}

\subsection{Case $k=17$}

\begin{proposition}\label{prop:case-k17}
Let $X_{17} := V(\nf_{17}(\tau,\sigma))\subset \mathbb{P}^{6}$.
Then the singular locus consists of two disjoint components:
$$
\Sing(X_{17})=C \cup L_{56},
$$
where
$C = V(x_{3}, x_{4}, x_{5}, x_{6}, x_{0}^2 + x_{1}^2 + \sigma x_{2}^2)$
and
$L_{56} = V(x_{0}, x_{1}, x_{2}, x_{3}, x_{4})$.
Here, $C$ is a smooth conic (degree $2$) and $L$ is a line (degree $1$). We have $C \cap L_{56} = \varnothing$.
\end{proposition}

\subsection{Case $k=18$}
\begin{proposition}\label{prop:case-k18}
Let $X_{18} := V(\nf_{18})\subset \mathbb{P}^{6}$.
Then the singular locus consists of two disjoint components:
$$
\Sing(X_{18})=L_{1} \cup L_{2},
$$
where
$L_{1} = V(x_{2}, x_{3}, x_{4}, x_{5}, x_{6})$
and
$L_{2} = V(x_{0}, x_{1}, x_{2}, x_{3}, x_{4})$.
Here, both $L_{1}$ and $L_{2}$ are lines (degree $1$). We have $L_{1} \cap L_{2} = \varnothing$.
\end{proposition}

\subsection{Case $k=19$}
\begin{proposition}\label{prop:case-k19}
Let $X_{19} := V(\nf_{19})\subset \mathbb{P}^{6}$.
Then the singular locus has two irreducible components:
$$
\Sing(X_{19})=\{ P \} \cup C,
$$
$P=(0:0:0:0:0:0:1)$, $C=V(x_{0}^2+\rho x_{1}^2+\sigma x_{0}x_{2}+x_{2}^2,x_{3},x_{4},x_{5},x_{6})$. 
Here $P$ is an isolated singular point of type $\mathsf{E}_{18}$.
\end{proposition}

\subsection{Case $k=20$}

\begin{proposition}\label{prop:case-k20}
Let $X_{20} := V(\nf_{20}(\tau))\subset \mathbb{P}^{6}$.
Then the singular locus consists of two disjoint components:
$$
\Sing(X_{20})=C \cup L,
$$
where
$C = V(x_{4}, x_{5}, x_{6}, (\tau-1) x_{1}^2 - 2x_{2}^2 + x_{3}^2, (\tau-1)x_{0}^2 + (\tau+1)x_{2}^2 + (\tau-2)x_{3}^2)$
and
$L = V(x_{0}, x_{1}, x_{2}, x_{3}, x_{4})$.
Here, $C$ is a curve of degree $4$ (a complete intersection of two quadrics, which is an elliptic curve) and $L$ is a line (degree $1$). We have $C \cap L = \varnothing$.
\end{proposition}

\subsection{Case $k=21$}

\begin{proposition}\label{end}
Let $X_{21} := V(\nf_{21})\subset \mathbb{P}^{6}$.
Then the singular locus has three irreducible components:
$$
\Sing(X_{21})=\{ P_{1} \} \cup \{ P_{2} \} \cup \{ P_{3} \},
$$
where $P_{1},P_{2}$ and $P_{3}$ are isolated singular points of type $\mathcal{C}_{cub}$.
\end{proposition}

\newpage

\begin{table}[H]
  \centering
  \small
  \renewcommand{\arraystretch}{1.2}
  \setlength{\tabcolsep}{3pt}
  \caption{Singular loci of the 21 closed-orbit representatives (Section~\ref{sec:5})}
  \label{tab:section5-sing}
  \begin{tabular}{@{}c p{5.2cm} p{4.5cm} p{3.2cm}@{}}
    \toprule
    $k$ & Singular locus (notation of \S\ref{sec:5}) & Type / degree of top-dimensional part & Isolated point(s) / invariants \\
    \midrule
     1  & $C \cup \{P\}$ ($C$ smooth conic)
        & conic (deg~2)
        & one point $P$ of type $\mathsf{E}_{19}$\\

     2  & $L_{01} \cup C$ ($L_{01}\cong \mathbb{P}^1$, $C$ a CI$(2,2)$)
        & line $+$ quartic curve (deg~1+4)
        & --- \\

     3  & $L_{01} \cup C$ ($L_{01}\cong \mathbb{P}^1$, $C$ a CI$(2,2)$)
        & line $+$ quartic curve (deg~1+4)
        & --- \\

     4  & $L_{01} \cup L_{56}$ (two disjoint lines)
        & two disjoint lines (deg~1+1)
        & --- \\

     5  & $\{P\} \cup \{Q\}$
        & two points
        & $P,Q$ of type $\mathsf{E}_{21}$ (same modulus) \\

     6  & $C \cup L_{56}$ ($C$ a CI$(2,2)$, $L_{56}$ line)
        & line $+$ quartic curve (deg~1+4)
        & --- \\

     7  & $\{P\} \cup C$ ($C$ conic)
        & conic (deg~2)
        & one point $P$ of type $\mathsf{E}_{19}$ \\

     8  & $C \cup \{P\}$ ($C$ smooth conic)
        & conic (deg~2)
        & one point $P$ of type $\mathsf{E}_{18}$\\

     9  & $L_{01} \cup \{P\}$ ($L_{01}$ line)
        & line (deg~1)
        & one point $P$ of type $\mathcal{O}_{6}$ \\

     10 & $L_{01} \cup C$ ($L_{01}$ line, $C$ conic)
        & line $+$ conic (deg~1+2)
        & --- \\

     11 & $\{P_{1}\} \cup \{P_{2}\}\cup \{P_{3}\}$
        & three points
        & three points $P_{i}$ of type $\mathcal{C}_{cub}$\\

     12 & $L_{56} \cup C$ ($L_{56}$ line, $C$ a CI$(2,2)$)
        & line $+$ quartic curve (deg~1+4)
        & --- \\

     13 & $C_1 \cup C_2$ (disjoint)
        & two pairs of intersecting lines (deg~1+1+1+1)
        & --- \\

     14 & $L_{01} \cup C$ ($L_{01}$ line, $C$ a CI$(2,2)$)
        & line $+$ quartic curve (deg~1+4)
        & --- \\

     15 & $C_1 \cup C_2 \cup L$ (two CI$(2,2)$s + line)
        & two quartics $+$ line (deg~4+4+1)
        & --- \\

     16 & $\{P\} \cup \{Q\}$
        & two points
        & $P,Q$ of type $\mathsf{E}_{20}$ (same modulus) \\

     17 & $C \cup L$ ($C$ smooth conic, $L$ line)
        & conic $+$ line (deg~2+1)
        & --- \\

     18 & $L_1 \cup L_2$ (two disjoint lines)
        & two disjoint lines (deg~1+1)
        & --- \\

     19 & $C \cup \{P\}$ ($C$ smooth conic)
        & conic (deg~2)
        & one point $P$ of type $\mathsf{E}_{18}$ \\

     20 & $C \cup L$ ($C$ elliptic curve, $L$ line)
        & elliptic curve $+$ line (deg~4+1)
        & --- \\

     21 & $\{P_{1}\} \cup \{P_{2}\}\cup \{P_{3}\}$
        & three points
        & three points $P_{i}$ of type $\mathcal{C}_{cub}$\\
    \bottomrule
  \end{tabular}
\end{table}

\begin{table}[H]
  \centering
  \small
  \renewcommand{\arraystretch}{1.3}
   \caption{List of isolated singularities}
  \label{tab:sing-recap}
  \begin{tabular}{c c l}
    \toprule
    Symbol & Type & Defining Equation (Normal Form) \\
    \midrule
    $\mathsf{E}_{19}$ & $\mathrm{QH}(9,6,5;24)_{19}$ & $X^{2}Y + Y^{4} + XZ^{3} = 0$ \\
    $\mathsf{E}_{21}$ & $\mathrm{QH}(2,3,5;12)_{21}(t)$ & $t X^{6} + X^{3}Y^{2} + X^{2}YZ + XZ^{2} + Y^{4} = 0$ \\
    $\mathsf{E}_{18}$ & $\mathrm{QH}(3,3,4;12)_{18}(t)$ & $Z^{3} + X^{4} + t X^{2} Y^{2} + Y^{4} = 0$ \\
    $\mathcal{O}_{6}$ & $\mathrm{QH}(1,1;6)_{25}$ & $\prod_{i=1}^{6} (X - \lambda_{i}Y) = 0$ \\
    $\mathcal{C}_{cub}$ & $\mathrm{QH}(1,1,1,1;3)_{16}([c])$ & $\sum_{i=0}^{4} c_{i}L_{i}^{3} = 0$ \\
    $\mathsf{E}_{20}$ & $\mathrm{QH}(1,2,2;6)_{20}(t)$ & $\scriptstyle{Y^3 + Z^3 + X^6  + t_1 X^2 Y^2 + t_2 X^2 Y Z + t_3 Y^2 Z + t_4 X^4 Y + t_5 X^4 Z=0}$\\
    \bottomrule
  \end{tabular}
\end{table}

\newpage

\section{The minimal exponents}\label{sec:minimal-exponents}

\subsection{Minimal exponents of isolated singularities}

In Definitions~5.1--5.6 we introduced six quasi-homogeneous isolated hypersurface singularities that occur as the isolated points on the GIT boundary.  The next theorem records a uniform numerical feature of these singularities, and explains why they sit exactly at the critical value in Park's stability criterion.

\begin{thm}\label{thm:minexp-73}
Let $(X,x)$ be one of the six isolated boundary singularities from Definitions~5.1--5.6.
Then the local minimal exponent is
\[
\widetilde{\alpha}_x(X)=\frac{7}{3}.
\]
\end{thm}

\begin{proof}
By the holomorphic splitting lemma, the local defining equation of $X$ at an isolated singular point $x$
is analytically equivalent to
\[
u_1^2+\cdots+u_{6-c}^2 \;+\; g(z_1,\dots,z_c),
\]
where $c=\mathrm{corank}_x(X)$ and $g$ is one of the weighted-homogeneous germs from
Definitions~5.1--5.6

By the Thom-Sebastiani type theorem for minimal exponents (see, for example, \cite[Theorem 2.3]{Par25}),
\[
\widetilde{\alpha}_x(X)
=\widetilde{\alpha}_0\!\left(u_1^2+\cdots+u_{6-c}^2\right)+\widetilde{\alpha}_0(g)
=\frac{6-c}{2}+\widetilde{\alpha}_0(g).
\]
For a weighted-homogeneous isolated hypersurface singularity of type $\mathrm{QH}(w_1,\dots,w_c;D)$,
we have $\widetilde{\alpha}_0(g)=\frac{w_1+\cdots+w_c}{D}$.  In our list, this yields
\[
\widetilde{\alpha}_0(g)=
\begin{cases}
\frac56, & g\in\{\mathrm{QH}(9,6,5;24)_{19},\,\mathrm{QH}(2,3,5;12)_{21}(t),\,\mathrm{QH}(3,3,4;12)_{18}(t),\,\mathrm{QH}(1,2,2;6)_{20}(t)\},\\[2pt]
\frac13, & g=\mathrm{QH}(1,1;6)_{25},\\[2pt]
\frac43, & g=\mathrm{QH}(1,1,1,1;3)_{16}([c]).
\end{cases}
\]
The corresponding coranks are $c=3,2,4$ respectively (cf.\ the definitions), hence in each case
\[
\widetilde{\alpha}_x(X)=\frac{6-c}{2}+\widetilde{\alpha}_0(g)=\frac{7}{3}.
\]
\end{proof}

\begin{remark}[Significance via Park's threshold]
For cubic fivefolds $(n,d)=(6,3)$, Park's GIT criterion gives the sharp threshold
$\frac{n+1}{d}=\frac{7}{3}$:
a degree-$d$ hypersurface $X\subset\mathbb{P}^{n}$ is GIT stable (resp.\ semistable) if
$\widetilde{\alpha}(X)>\frac{n+1}{d}$ (resp.\ $\ge \frac{n+1}{d}$) \cite[Theorem A]{Par25}.
Hence, if $X$ has a point with local minimal exponent $7/3$, then
$\widetilde{\alpha}(X)\le 7/3$, so $X$ cannot lie in the strictly stable locus.
This explains why the isolated singularities in Theorem~\ref{thm:minexp-73}
occur precisely on the GIT boundary.

Moreover, Park observed that the classical explicit GIT analyses for cubic hypersurfaces
with $n\le 5$ can be reformulated uniformly in terms of the same threshold
$\frac{n+1}{3}$ for the global minimal exponent \cite[Remark 7.6]{Par25};
Theorem~\ref{thm:minexp-73} shows that, for cubic fivefolds ($n=6$), the equality
\[
\widetilde{\alpha}_x(X)=\frac{7}{3}=\frac{n+1}{3}
\]
is attained along the strictly semistable boundary.
\end{remark}

\begin{definition}[Extremal and $7/3$-singularities]\label{def:extremal-sing}
We call the six isolated analytic types appearing in Theorem~\ref{thm:minexp-73}
\emph{extremal cubic fivefold singularities}.
More generally, an isolated cubic fivefold singularity with local minimal exponent $7/3$
is called a \emph{$7/3$-singularity}.
\end{definition}

\subsection{The global minimal exponents of general members of the boundary families}

For families whose singular locus $\mathrm{Sing}(X_k)$ has positive-dimensional components (curves or lines), we must compute the local minimal exponent both at a general point and at the special points along these strata in order to determine the global minimal exponent
\[
\widetilde{\alpha}(X_k) \;=\; \min_{x \in \mathrm{Sing}(X_k)} \widetilde{\alpha}_x(X_k).
\]

Our computations for all $21$ closed-orbit normal forms reveal a striking uniformity that matches Park's threshold:

\begin{thm}\label{thm:global-minimal-exponents}
For each $k \in \{1, \dots, 21\}$, the global minimal exponent of a general strictly semistable closed-orbit representative
$X_k = V(\nf_{k}) \subset \mathbb{P}^6$ in the component $\Phi_{k}$
is
\[
\widetilde{\alpha}(X_k) \;=\; \frac{7}{3}.
\]
\end{thm}

\begin{proof}
We prove the theorem by a uniform case analysis organized by the structure of the one-dimensional components of $\mathrm{Sing}(X_{k})$. More precisely, for each $k$ with $\dim \mathrm{Sing}(X_{k})=1$, we classify every irreducible curve component $\Gamma\subset \mathrm{Sing}(X_{k})$ into one of finitely many pattern types according to the behavior of the local minimal exponent $p\mapsto\widetilde{\alpha}_{p}(X_{k})$ along $\Gamma$ (see Table~\ref{tab:me-pattern-index-noA0}). 

In each pattern type, the minimum of $\widetilde{\alpha}_{p}(X_{k})$ along $\Gamma$ equals $7/3$. For types A1 and A2, this minimum is attained at an explicitly described finite subset of special points, whereas for types B and C, the minimal exponent is identically $7/3$ everywhere along $\Gamma$. Consequently, $\min_{p\in\Gamma}\widetilde{\alpha}_{p}(X_{k})=\frac{7}{3}$ for every curve component $\Gamma\subset \mathrm{Sing}(X_{k})$.

If $\mathrm{Sing}(X_k)$ has only isolated singular points (namely $k \in \{5,11,16,21\}$), then $\widetilde{\alpha}(X_k)=7/3$ follows directly from the computation of the six isolated extremal singularity types recorded in Table~\ref{tab:sing-recap}.
Taking the minimum over all singular points, we obtain $\widetilde{\alpha}(X_k)=7/3$ for every $k=1,\dots,21$.
\end{proof}

\begin{table}[H]
\centering
\scriptsize
\setlength{\tabcolsep}{5pt}
\renewcommand{\arraystretch}{1.25}
\resizebox{\textwidth}{!}{%
\begin{tabular}{l|l|l|l|l}
\toprule
component type
& Type A1
& Type A2
& Type B
& Type C \\
\midrule

line
& \begin{tabular}[t]{@{}l@{}}
$k=2$: $L_{01}$\\
$k=3$: $L_{01}$\\
$k=4$: $L_{01},\,L_{56}$\\
$k=6$: $L_{56}$\\
$k=10$: $L_{01}$\\
$k=12$: $L_{56}$\\
$k=17$: $L_{56}$
\end{tabular}
& --
& $k=18$: $L_{01},\,L_{56}$
& \begin{tabular}[t]{@{}l@{}}
$k=9$: $L_{01}$\\
$k=14$: $L_{01}$\\
$k=15$: $L$\\
$k=20$: $L$
\end{tabular}
\\
\midrule

smooth conic
& \begin{tabular}[t]{@{}l@{}}
$k=1$: $C$\\
$k=7$: $C$
\end{tabular}
& --
& \begin{tabular}[t]{@{}l@{}}
$k=10$: $C$\\
$k=17$: $C$
\end{tabular}
& \begin{tabular}[t]{@{}l@{}}
$k=8$: $C$\\
$k=19$: $C$
\end{tabular}
\\
\midrule

twisted quartic
& \begin{tabular}[t]{@{}l@{}}
$k=6$: $C$\\
$k=12$: $C$
\end{tabular}
& --
& --
& --
\\
\midrule

elliptic quartic
& --
& --
& \begin{tabular}[t]{@{}l@{}}
$k=14$: $C$\\
$k=20$: $C$
\end{tabular}
& --
\\
\midrule

reducible conic
& $k=13$: $C_1,\,C_2$
& --
& --
& --
\\
\midrule

CI(2,2)
& \begin{tabular}[t]{@{}l@{}}
$k=2$: $C$\\
$k=3$: $C$
\end{tabular}
& $k=15$: $C_1,\,C_2$
& --
& --
\\

\bottomrule
\end{tabular}}
\caption{Index of one-dimensional singular-locus components by minimal-exponent pattern type.}
\label{tab:me-pattern-index-noA0}
\end{table}

\subsubsection*{Pattern types for one-dimensional singular components}

Let $X \subset \mathbb{P}^6$ be one of the closed-orbit normal forms $X_k=V(\nf_k)$, and let
$\Gamma \subset \Sing(X)$ be an irreducible curve component.
We classify $\Gamma$ according to the behavior of the function
\[
\Gamma \ni p \longmapsto \widetilde{\alpha}_p(X)
\]
along $\Gamma$.
In all cases appearing in this paper, the only values that occur on $\Gamma$ are $5/2$ and $7/3$.
We say that $\Gamma$ is of one of the following pattern types.

\medskip
\noindent\textbf{Type A1 (one special point).}
$\Gamma$ is of Type~A1 if there exists a unique point $p_0\in \Gamma$ such that
\[
\widetilde{\alpha}_{p_0}(X)=\frac{7}{3},
\qquad
\widetilde{\alpha}_{p}(X)=\frac{5}{2}\ \text{ for all }p\in \Gamma\setminus\{p_0\}.
\]
Equivalently, $\widetilde{\alpha}_p(X)=5/2$ at a general point of $\Gamma$, and the minimum along $\Gamma$ is attained at a single
explicitly described ``special'' point $p_0$.

\medskip
\noindent\textbf{Type A2 (two special points).}
$\Gamma$ is of Type~A2 if there exist exactly two points $p_1,p_2\in \Gamma$ such that
\[
\widetilde{\alpha}_{p_1}(X)=\widetilde{\alpha}_{p_2}(X)=\frac{7}{3},
\qquad
\widetilde{\alpha}_{p}(X)=\frac{5}{2}\ \text{ for all }p\in \Gamma\setminus\{p_1,p_2\}.
\]
Equivalently, $\widetilde{\alpha}_p(X)=5/2$ at a general point of $\Gamma$, and the minimum along $\Gamma$ is attained at two
explicitly described special points.

\medskip
\noindent\textbf{Type B (uniform $7/3$ without rank drop).}
$\Gamma$ is of Type~B if
\[
\widetilde{\alpha}_{p}(X)=\frac{7}{3}\quad\text{for all }p\in \Gamma,
\]
and moreover the transverse quadratic part to $\Gamma$ has constant rank along $\Gamma$
(i.e.\ there are no rank-drop points of the Hessian along $\Gamma$).
In particular, the minimum of $\widetilde{\alpha}_p(X)$ on $\Gamma$ equals $7/3$ and is attained everywhere on $\Gamma$.

\medskip
\noindent\textbf{Type C (uniform $7/3$ with rank drop).}
$\Gamma$ is of Type~C if
\[
\widetilde{\alpha}_{p}(X)=\frac{7}{3}\quad\text{for all }p\in \Gamma,
\]
but the transverse quadratic part degenerates at finitely many points of $\Gamma$
(equivalently, the Hessian rank drops along $\Gamma$ at a finite set).
Thus Type~C shares the same uniform value $7/3$ as Type~B, but differs in that certain special points on $\Gamma$
exhibit additional quadratic degenerations, which nevertheless do not change the minimal exponent.

\medskip
In all cases in this paper, every curve component $\Gamma \subset \Sing(X_k)$ is of one of the types A1, A2, B, or C;
see Table~\ref{tab:me-pattern-index-noA0}.

\section{Adjacency relations among strictly semistable components}\label{sec:adjacency}

In this section we record the \emph{codimension-one wall adjacency} among the closed strata
$\{\Phi_k\}_{k=1}^{21}\subset \mathbb{P}(W)^{\mathrm{ss}}$, where $W=\Sym^3\mathbb{C}^7$.
Throughout Section~\ref{sec:adjacency}, we use the term ``adjacent'' \emph{only} in the sense of
\emph{wall adjacency} (i.e.\ codimension-one incidence in the Hilbert--Mumford weight fan).

\smallskip
\noindent\textbf{Sign convention and why we match $\pm r_j$.}
A wall is determined by a hyperplane, hence it is insensitive to the orientation of a normal vector.
Consequently, when we compare slices of the form
\[
I(r)_{\ge 0}\cap I(r_k)_{=0}\subset I(r_k)_{=0},
\]
the same geometric wall may be represented either by $r$ or by $-r$.
For this reason, in the computational matching against the reduced list
$r_1,\dots,r_{21}$ (Section~2), we always test \emph{both} signs $\pm r_j$.
This eliminates the spurious ``one-way'' phenomena that occur if one only matches the fixed
sign representatives $r_j$.

\smallskip
\noindent
For background on wall-crossing in GIT and Kirwan's stratification, see
\cite{Kir84,DH98,Tha96}.

\begin{definition}[Adjacency via codimension-one walls]\label{def:adjacency}
Let $\Phi_i,\Phi_j$ be closed strata in the strictly semistable locus.
We say that $\Phi_i$ and $\Phi_j$ are \emph{adjacent} if the associated maximal cones
$I(r_i)_{\ge 0}$ and $I(r_j)_{\ge 0}$ meet along a codimension-one face (a \emph{wall})
in the Hilbert--Mumford weight fan.
Equivalently, there exists a codimension-one wall on which the Hilbert--Mumford numerical functions
for the corresponding maximally destabilizing $1$-PS's tie, and the wall separates the chambers
corresponding to $r_i$ and $r_j$.
\end{definition}

\subsection{Maximal slices and the computational criterion}

Fix an index $k$.
Inside the hyperplane $I(r_k)_{=0}\subset I$ we consider subsets of the form
\[
I(r)_{\ge 0}\cap I(r_k)_{=0}\qquad (r\in\mathbb{Z}^{7}_{(0)},\ r\not\parallel r_k).
\]
Such a subset corresponds to a codimension-one face of the cone $I(r_k)_{\ge 0}$
precisely when it is \emph{inclusion-maximal} among all subsets of this form.
Following the terminology used in the accompanying Sage computations, we call these inclusion-maximal
subsets the \emph{maximal slices} in $I(r_k)_{=0}$.

\smallskip
\noindent\textbf{Directed slice-matching relation.}
For bookkeeping, we introduce a directed relation on $\{1,\dots,21\}$:
we write $k\rightsquigarrow j$ if \emph{some} maximal slice in $I(r_k)_{=0}$ equals
\[
I(\pm r_j)_{\ge 0}\cap I(r_k)_{=0}
\]
for at least one choice of sign.
(Here $\pm r_j$ indicates that we match against both orientations of $r_j$, as explained above.)

\smallskip
\noindent
By Definition~\ref{def:adjacency}, wall adjacency is symmetric. In our combinatorial model this means:
\[
\Phi_i\text{ and }\Phi_j\text{ are adjacent}
\quad\Longleftrightarrow\quad
(i\rightsquigarrow j)\ \text{and}\ (j\rightsquigarrow i).
\]
In other words, the adjacency graph is obtained from the directed slice-matching digraph by retaining
only the mutual pairs.

\smallskip
\noindent\textbf{Implementation.}
We compute maximal slices by the same linear-algebraic method as in Section~2, adapted to the
codimension-one setting: we enumerate $4$-tuples of points in $I(r_k)_{=0}$, extract candidate normals
from the right kernel (modulo $r_k$), build the corresponding slices, and keep only inclusion-maximal
ones. We then match each maximal slice against the list $\{\pm r_1,\dots,\pm r_{21}\}$.
(See the accompanying Sage scripts and logs for the full output.)

\subsection{Codimension-one adjacency graph}

The main output of the codimension-one computation is the following.

\begin{thm}[Codimension-one wall adjacency]\label{thm:adjacency}
Among the closed strata $\Phi_k$ in the strictly semistable locus, the wall-adjacency graph
(codimension-one adjacency in the sense of Definition~\ref{def:adjacency}) has $21$ vertices and
$56$ edges. Equivalently, for each $k$ the set
\[
N(k)\ :=\ \{\,j\neq k \mid \Phi_k\text{ is adjacent to }\Phi_j\,\}
\]
is given by Table~\ref{tab:adjacency-codim1}.
\end{thm}

\begin{table}[t]
\centering
\caption{Codimension-one wall adjacency: for each $k$, the neighbor set $N(k)$ of indices $j$ such that
$\Phi_k$ and $\Phi_j$ are adjacent (Definition~\ref{def:adjacency}).}
\label{tab:adjacency-codim1}
\renewcommand{\arraystretch}{1.15}
\begin{tabular}{c|p{0.36\textwidth} c|p{0.36\textwidth}}
$k$ & $N(k)$ & $k$ & $N(k)$\\\hline
$1$  & $3,5,7,10$
& $12$ & $3,4,5,6,7,15,17$\\
$2$  & $3,4,5,9$
& $13$ & $4,8,18,19$\\
$3$  & $1,2,4,5,9,10,12$
& $14$ & $9,10,11,18$\\
$4$  & $2,3,5,6,12,13,18$
& $15$ & $6,8,12,17,18,20$\\
$5$  & $1,2,3,4,6,7,12$
& $16$ & $8,18,19$\\
$6$  & $4,5,12,15$
& $17$ & $7,12,15,18,19,20,21$\\
$7$  & $1,5,12,17$
& $18$ & $4,9,10,13,14,15,16,17,20$\\
$8$  & $10,11,13,15,16,19$
& $19$ & $8,9,13,16,17,21$\\
$9$  & $2,3,10,14,18,19$
& $20$ & $15,17,18,21$\\
$10$ & $1,3,8,9,11,14,18$
& $21$ & $17,19,20$\\
$11$ & $8,10,14$
&  & \\
\end{tabular}
\end{table}

\begin{figure}[t]
\centering
\includegraphics[width=0.95\textwidth]{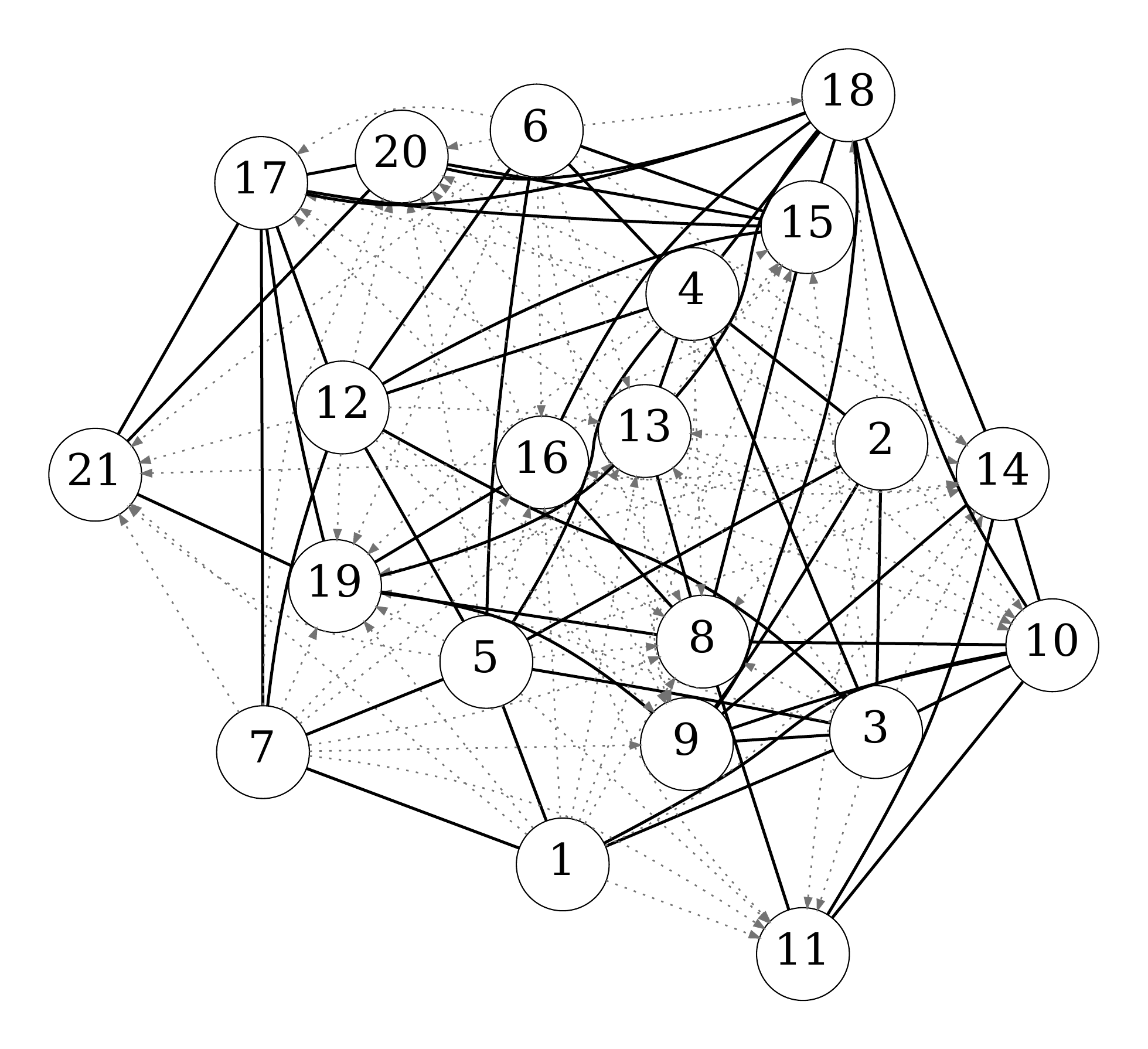}
\caption{Codimension-one slice-matching graph for the strata $\Phi_1,\dots,\Phi_{21}$.
Vertices are indexed by $k$ and represent $\Phi_k$.
A dotted arrow $i\to j$ indicates a \emph{one-sided} slice match: there exists an inclusion-maximal
slice $S\subset I(r_i)_{=0}$ of the form $S=I(\pm r_j)_{\ge 0}\cap I(r_i)_{=0}$, while no such match
occurs in the reverse direction.
A solid (undirected) edge between $i$ and $j$ indicates that slice matches occur in both directions
($i\to j$ and $j\to i$); equivalently, $\Phi_i$ and $\Phi_j$ are adjacent in the sense of
Definition~\ref{def:adjacency}.}
\label{fig:adjacency-codim1}
\end{figure}

\begin{remark}[One-sided slice matches and higher codimension]\label{rem:one-sided}
In the directed slice-matching digraph one also encounters \emph{one-sided} arrows
$k\rightsquigarrow j$ without the reverse arrow $j\rightsquigarrow k$.
These do \emph{not} represent codimension-one wall adjacency in the sense of
Definition~\ref{def:adjacency}; rather, they indicate that a facet of $I(r_k)_{\ge 0}$
is detected by $\pm r_j$ when restricted to $I(r_k)_{=0}$, while from the $j$-side the same incidence
is not codimension-one. Empirically, such one-sided phenomena are accounted for by higher-codimension
incidence among faces (codimension $\ge 2$), as can be explored by extending the same slice computation
to intersections of multiple hyperplanes $I(r_{k_1})_{=0}\cap\cdots\cap I(r_{k_c})_{=0}$.
\end{remark}

\FloatBarrier

\section{Non-inclusions and Filters}\label{sec:non-inclusion}

In this section, we prove that there are no inclusion relations among the
$21$ boundary families of strictly semistable cubic fivefolds constructed in the
previous sections. Concretely, for any distinct indices $k,\ell\in\{1,\dots,21\}$ we show
\[
\Phi_k \not\subseteq \Phi_\ell.
\]
Equivalently, we must rule out all $21\cdot 20 = 420$ ordered pairs $(k,\ell)$ with $k\neq \ell$.

Rather than checking these pairs directly, we use a \emph{sieve} (or ``filter'') method:
we impose a sequence of \emph{necessary} conditions for an inclusion
$\Phi_k\subseteq \Phi_\ell$.
Each condition is monotone under specialization (typically by upper semicontinuity),
so any pair violating the condition is discarded.
Applying these filters successively shrinks the candidate list until it becomes empty.

The filters are applied in the following order.

\begin{enumerate}
  \item \textbf{Filter 1: Dimension} (\S8.1).
  If $\Phi_k\subseteq \Phi_\ell$ and $k\neq \ell$, then necessarily
  \[
    \dim \Phi_k \;<\; \dim \Phi_\ell.
  \]

  \item \textbf{Filter 2: Apolar Betti numbers} (\S8.2).
  Within a fixed Hilbert-function stratum, the graded Betti numbers of the apolar algebra
  can only increase under specialization; hence an inclusion forces the entrywise Betti-table
  inequalities
  \[
    \beta_{i,j}\!\bigl(A_{f_k}\bigr)\ \ge\ \beta_{i,j}\!\bigl(A_{f_\ell}\bigr)
    \qquad\text{for all }(i,j).
  \]

  \item \textbf{Filter 3: Singular-locus Hilbert functions} (\S8.3).
  For a cubic $f$, the Hilbert function of the singular scheme $\Sing(X_f)$ is upper semicontinuous
  in families, so an inclusion forces the degreewise inequalities
  \[
    \HF_{\Sing(X_{f_k})}(d)\ \ge\ \HF_{\Sing(X_{f_\ell})}(d)
    \qquad\text{for all }d\ge 0.
  \]

  \item \textbf{Filter 4: Singular-locus degrees (leading term)} (\S8.4).
  The scheme-theoretic degree of the one-dimensional part of the singular locus is monotonic under specialization. 
  An inclusion therefore forces the inequality of the leading coefficients of the Hilbert polynomials:
  \[
    \deg_1\Sing(X_{f_k})\ \ge\ \deg_1\Sing(X_{f_\ell}).
  \]

  \item \textbf{Filter 5: Minimal exponents and limits of 1-cycles} (\S8.5).
  Under a flat specialization $X_\ell \rightsquigarrow X_k$, the fundamental $1$-cycle of the positive-dimensional singular locus specializes to an effective $1$-cycle supported on $\Sing(X_k)$. By the lower semicontinuity of the local minimal exponent $\widetilde{\alpha}$, this limit cycle must be supported on components having a generic minimal exponent no greater than that of the original curve.

  \item \textbf{Filter 6: Spectra of isolated singularities and generic stabilizers} (\S8.5.1).
  To eliminate the final remaining candidate pairs, we apply two fine obstructions. First, the Steenbrink spectrum of an isolated hypersurface singularity satisfies a strong semicontinuity property over half-open intervals, which restricts possible specializations. Second, an inclusion $\Phi_k\subseteq \Phi_\ell$ requires that the connected component of the generic stabilizer of $\Phi_\ell$ must be conjugate to a subgroup of the generic stabilizer of $\Phi_k$ in $G=\SL(7)$.
\end{enumerate}

\subsection{Filter 1: Dimension}
\label{subsec:filter1-dimension}

We begin with the most basic necessary condition coming from dimension
monotonicity.

\begin{lemma}[Dimension monotonicity]
\label{lem:dim-monotonicity}
Let $Y \subset Z$ be irreducible closed subsets of an algebraic variety.
If the inclusion is proper ($Y \subsetneq Z$), then $\dim Y < \dim Z$.
Equivalently, if $\dim Y \ge \dim Z$, then $Y \not\subsetneq Z$.
\end{lemma}

\begin{proof}
If $Y \subsetneq Z$ are irreducible, then any maximal chain of irreducible
closed subsets in $Y$ extends to a strictly longer chain in $Z$, so the Krull
dimension strictly increases.  (See any standard reference on dimension theory
in algebraic geometry.)
\end{proof}

In Section~4 (Table~2) we computed the dimensions of the $21$ families
$\Phi_k$.  Since we are only interested in \emph{proper} inclusions
$\Phi_k \subsetneq \Phi_\ell$ with $k \neq \ell$, Lemma~\ref{lem:dim-monotonicity}
implies the strict inequality
\begin{equation}
\label{eq:filter1-dim-ineq}
\Phi_k \subsetneq \Phi_\ell \quad \Longrightarrow \quad \dim \Phi_k < \dim \Phi_\ell .
\end{equation}
Thus, any ordered pair $(k,\ell)$ with $\dim \Phi_k \ge \dim \Phi_\ell$ is
discarded immediately.

For bookkeeping, set
\[
\mathcal{D}_d \;:=\; \{\,k \in \{1,\dots,21\} \mid \dim \Phi_k = d\,\}.
\]
From Table~2 we have the partition
\begin{align*}
\mathcal{D}_6 &= \{16,18\},\\
\mathcal{D}_4 &= \{11,21\},\\
\mathcal{D}_2 &= \{4,5,8,9,10,13,15,17,19\},\\
\mathcal{D}_1 &= \{14,20\},\\
\mathcal{D}_0 &= \{1,2,3,6,7,12\}.
\end{align*}

There are $21\cdot 20 = 420$ ordered pairs $(k,\ell)$ with $k\neq \ell$.
After imposing~\eqref{eq:filter1-dim-ineq}, the \emph{only} dimension jumps
that can survive are
\[
0\to 1,2,4,6;\qquad 1\to 2,4,6;\qquad 2\to 4,6;\qquad 4\to 6.
\]
Hence the number of pairs surviving Filter~1 is
\begin{align*}
|\mathcal{D}_0|\,(|\mathcal{D}_1|+|\mathcal{D}_2|+|\mathcal{D}_4|+|\mathcal{D}_6|)
&\;+\;
|\mathcal{D}_1|\,(|\mathcal{D}_2|+|\mathcal{D}_4|+|\mathcal{D}_6|)\\
&\;+\;
|\mathcal{D}_2|\,(|\mathcal{D}_4|+|\mathcal{D}_6|)
\;+\;
|\mathcal{D}_4|\,|\mathcal{D}_6|\\[2pt]
&= 6\cdot 15 \;+\; 2\cdot 13 \;+\; 9\cdot 4 \;+\; 2\cdot 2\\
&= 156.
\end{align*}
Therefore Filter~1 eliminates $420-156=264$ pairs, and the remaining $156$
candidates are passed to Filter~2.

\subsection{Filter 2: Apolar Betti numbers}\label{subsec:filter2-apolar}

Let
\[
S := \mathbb{C}[x_0,\dots,x_6],\qquad
R := \mathbb{C}[\partial_0,\dots,\partial_6],
\]
where $R$ acts on $S$ by differentiation $\partial_i=\frac{\partial}{\partial x_i}$.
For a cubic form $f\in S_3$, the apolar ideal is
\[
\Ann(f) := \{D\in R \mid D\cdot f = 0\}\subset R,
\]
and the corresponding apolar algebra is the graded Artinian Gorenstein algebra
\[
A_f := R/\Ann(f).
\]
Write $\beta_{i,j}(A_f)$ for the graded Betti numbers of the minimal graded free resolution
of $A_f$ over $R$.

\begin{lemma}[Upper semicontinuity of graded Betti numbers]\label{lem:usc-betti}
Let $\mathcal{A}$ be a flat family of finitely generated graded $R$-modules over an integral base $B$.
Then for every pair $(i,j)$ the function
\[
b\longmapsto \beta_{i,j}(\mathcal{A}_b)\qquad (b\in B)
\]
is upper semicontinuous. In particular, if $b_0$ is a specialization of the generic point $b_\eta$,
then $\beta_{i,j}(\mathcal{A}_{b_0})\ge \beta_{i,j}(\mathcal{A}_{b_\eta})$ for all $(i,j)$.
\end{lemma}

\begin{proof}
For each $(i,j)$ one has $\beta_{i,j}(M)=\dim_{\mathbb{C}}\Tor_i^R(M,\mathbb{C})_j$.
Moreover $\Tor_i^R(\mathcal{A},\mathbb{C})_j$ is the fiber of a coherent sheaf on $B$
obtained from a graded free resolution of $\mathcal{A}$; fiber dimensions of coherent sheaves
are upper semicontinuous, and flatness ensures compatibility with specialization.
\end{proof}

\begin{filter}[Apolar Betti numbers]\label{filter:apolar-betti}
Fix $k\neq \ell$. Suppose $\Phi_k\subseteq \Phi_\ell$.
Choose generic points $f_k\in \Phi_k$ and $f_\ell\in \Phi_\ell$ such that $A_{f_k}$ and $A_{f_\ell}$
lie in the same (Artinian) Hilbert-function stratum. Then
\[
\beta_{i,j}(A_{f_k})\ge \beta_{i,j}(A_{f_\ell})\qquad \text{for all }(i,j).
\]
Consequently, if there exists $(i,j)$ with $\beta_{i,j}(A_{f_k})<\beta_{i,j}(A_{f_\ell})$, then
$\Phi_k\not\subseteq \Phi_\ell$.
\end{filter}

\paragraph{Computed generic apolar Betti tables.}
For each family $\Phi_k$ we computed the graded Betti table of $A_f$ for a numerically generic element
$f\in \Phi_k$ (by fixing the weight--$0$ normal form and adding random positive-weight terms; in each case
we verified that $\HF(A_f)=(1,7,7,1)$). The resulting Betti tables fall into the following types.

For $m\in\{0,4,6,8,10,12\}$ set
\[
B(m):=
\begin{array}{c|cccccccc}
      & 0&1&2&3&4&5&6&7\\\hline
\text{total} & 1&21&64&70+m&70+m&64&21&1\\
0 & 1&.&.&.&.&.&.&.\\
1 & .&21&64&70&m&.&.&.\\
2 & .&.&.&m&70&64&21&.\\
3 & .&.&.&.&.&.&.&1
\end{array}.
\]
Then we obtain:
\[
\begin{aligned}
&B(0)\ \text{occurs for } k\in\{4,13,18\},\\
&B(4)\ \text{occurs for } k\in\{2,5,6,7,12,15,16,17,19\},\\
&B(6)\ \text{occurs for } k\in\{1,3\},\\
&B(8)\ \text{occurs for } k\in\{9,10\},\\
&B(10)\ \text{occurs for } k=8,\\
&B(12)\ \text{occurs for } k\in\{14,20\}.
\end{aligned}
\]
The remaining two cases are:
\[
B_{21}:=
\begin{array}{c|cccccccc}
      & 0&1&2&3&4&5&6&7\\\hline
\text{total} & 1&21&65&87&87&65&21&1\\
0 & 1&.&.&.&.&.&.&.\\
1 & .&21&64&71&16&1&.&.\\
2 & .&.&1&16&71&64&21&.\\
3 & .&.&.&.&.&.&.&1
\end{array},
\qquad
B_{11}:=
\begin{array}{c|cccccccc}
      & 0&1&2&3&4&5&6&7\\\hline
\text{total} & 1&21&67&99&99&67&21&1\\
0 & 1&.&.&.&.&.&.&.\\
1 & .&21&64&73&26&3&.&.\\
2 & .&.&3&26&73&64&21&.\\
3 & .&.&.&.&.&.&.&1
\end{array}.
\]
In particular, these Betti tables are totally ordered entrywise:
\[
B(0)\ \le\ B(4)\ \le\ B(6)\ \le\ B(8)\ \le\ B(10)\ \le\ B(12)\ \le\ B_{21}\ \le\ B_{11}.
\]
Therefore, for any pair $(k,\ell)$ we can discard $\Phi_k\subseteq \Phi_\ell$ as soon as the Betti type
of $k$ is strictly smaller than the Betti type of $\ell$.

\begin{proposition}[Effect of Filter 2]\label{prop:filter2-effect}
Applying Filter~\ref{filter:apolar-betti} to the $156$ ordered pairs surviving Filter~1 eliminates $66$
further pairs, leaving $90$ candidates for the subsequent filters. More precisely, the following $66$
pairs are discarded by the apolar Betti criterion:
\[
\begin{aligned}
&k\in \{1,2,3,7,12\},\ \ \ell\in \{8,9,10,11,14,20,21\};\\
&k=6,\ \ \ell\in \{8,9,10,11,14,20,21\};\\
&k\in \{4,13\},\ \ \ell\in \{11,16,21\};\\
&k\in \{5,8,9,10,14,15,17,19,20\},\ \ \ell\in \{11,21\}.
\end{aligned}
\]
\end{proposition}

\begin{proof}
By Lemma~\ref{lem:usc-betti}, in a specialization within a fixed Hilbert-function stratum the graded
Betti numbers can only increase. Since $\HF(A_f)=(1,7,7,1)$ for all computed generic representatives,
the comparison is valid across all families. Using the explicit Betti types listed above (which form a
single chain), any candidate pair $(k,\ell)$ with Betti type$(k) <$ Betti type$(\ell)$ is excluded.
Intersecting this condition with the $156$ pairs from Filter~1 yields exactly the list displayed, consisting
of $66$ pairs.
\end{proof}

\subsection{Filter 3: Hilbert functions of singular loci}\label{subsec:filter3-sing}

Keep the notation $S=\mathbb{C}[x_0,\dots,x_6]$. For a cubic form $f\in S_3$, set
$X_f:=V(f)\subset \mathbb{P}^6$ and recall that the (projective) singular scheme is cut out by the
saturated Jacobian ideal
\[
J(f):=\Bigl(\frac{\partial f}{\partial x_0},\dots,\frac{\partial f}{\partial x_6}\Bigr):(x_0,\dots,x_6)^\infty
\subset S,
\qquad
\Sing(X_f)=V\!\bigl(J(f)\bigr)\subset \mathbb{P}^6.
\]
For $d\ge 0$ define the singular-locus Hilbert function in degree $d$ by
\[
\HF_{\Sing(X_f)}(d):=\dim_{\mathbb{C}}(S/J(f))_d.
\]

\begin{lemma}[Upper semicontinuity of Hilbert functions in each degree]\label{lem:usc-hf}
Let $B$ be an integral variety and let $Z\subset \mathbb{P}^6\times B$ be a family of closed subschemes
with ideal sheaf $\mathcal{I}_Z$. Then for every $d\ge 0$ the function
\[
b\longmapsto \dim_{\mathbb{C}}(S/I_b)_d\qquad (b\in B),
\]
where $I_b\subset S$ is the saturated homogeneous ideal of the fiber $Z_b\subset \mathbb{P}^6$,
is upper semicontinuous. In particular, if $b_0$ is a specialization of the generic point $b_\eta$,
then $\HF_{Z_{b_0}}(d)\ge \HF_{Z_{b_\eta}}(d)$ for all $d\ge 0$.
\end{lemma}

\begin{proof}
Fix $d\ge 0$. The vector space $(S/I_b)_d$ is the fiber in degree $d$ of a coherent quotient of the
trivial bundle $S_d\otimes \mathcal{O}_B$, hence its fiber dimension is upper semicontinuous.
The stated inequality follows for specializations.
\end{proof}

\begin{filter}[Singular-locus Hilbert functions]\label{filter:sing-hf}
Fix $k\neq \ell$. If $\Phi_k\subseteq \Phi_\ell$, then for generic choices $f_k\in \Phi_k$ and
$f_\ell\in \Phi_\ell$ one has
\[
\HF_{\Sing(X_{f_k})}(d)\ge \HF_{\Sing(X_{f_\ell})}(d)\qquad \text{for all } d\ge 0.
\]
Consequently, if there exists $d$ with $\HF_{\Sing(X_{f_k})}(d)<\HF_{\Sing(X_{f_\ell})}(d)$, then
$\Phi_k\not\subseteq \Phi_\ell$.
\end{filter}

\paragraph{Computed (minimal) singular-locus Hilbert functions.}
For each $k=1,\dots,21$ we generated $5$ random perturbations of the weight--$0$ normal form $\nf_{k}$
(as in Filter~2) and computed $\HF_{\Sing(X_f)}(d)$ up to degree $40$.
For each $d$ we record the componentwise minimum
\[
h_k(d):=\min\{\HF_{\Sing(X_f)}(d): f \text{ among the samples in }\Phi_k\}.
\]
In the comparison below we apply Filter~\ref{filter:sing-hf} using these values $h_k(d)$.

\begin{proposition}[Effect of Filter 3]\label{prop:filter3-effect}
Applying Filter~\ref{filter:sing-hf} to the $90$ ordered pairs surviving Filter~2 eliminates $67$ further pairs,
leaving $23$ candidates for the subsequent filters. More precisely, the following $67$ pairs are discarded by the
singular-locus Hilbert-function criterion:
\[
\begin{aligned}
&k=1,\ \ell\in \{4,5,16\};\\
&k=2,\ \ell\in \{4,5,13,16,19\};\\
&k=3,\ \ell\in \{4,5,13,16,19\};\\
&k=4,\ \ell=18;\\
&k=5,\ \ell=18;\\
&k=6,\ \ell\in \{4,13,18\};\\
&k=7,\ \ell\in \{4,5,13,15,16,17,18\};\\
&k=8,\ \ell\in \{16,18\};\\
&k=9,\ \ell=18;\\
&k=10,\ \ell\in \{16,18\};\\
&k=11,\ \ell\in \{16,18\};\\
&k=12,\ \ell\in \{4,5,13,15,16,18,19\};\\
&k=14,\ \ell\in \{4,5,8,9,10,13,15,16,17,18,19\};\\
&k=15,\ \ell=16;\\
&k=17,\ \ell\in \{16,18\};\\
&k=19,\ \ell\in \{16,18\};\\
&k=20,\ \ell\in \{4,5,8,9,10,13,15,16,17,19\};\\
&k=21,\ \ell\in \{16,18\}.
\end{aligned}
\]
\end{proposition}

Hence the only ordered pairs $(k, l)$ not ruled out by Filters 1--3 are:
\begin{center}
(1, 13), (1, 15), (1, 17), (1, 18), (1, 19), \\
(2, 15), (2, 17), (2, 18), (3, 15), (3, 17), (3, 18), \\
(5, 16), (6, 5), (6, 15), (6, 16), (6, 17), (6, 19), \\
(7, 19), (9, 16), (12, 17), (13, 18), (15, 18), (20, 18).
\end{center}

\subsection{8.4 Filter 4: Singular-locus degrees (leading term)}

In this filter we extract a coarse but robust numerical invariant from the
singular-locus Hilbert function, namely the \emph{scheme-theoretic degree of the
one-dimensional part}. This ignores any isolated (possibly nonreduced) points,
since such components contribute only to the constant term of the Hilbert
polynomial.

\medskip

\noindent\textbf{(A) Scheme-theoretic degree via Hilbert polynomial.}
Let $Z \subset \PP^6$ be a closed subscheme with $\dim Z \le 1$, and write
\[
HF_Z(d) := \dim_{\mathbb{C}}(S/I_Z)_d
\]
for its Hilbert function. For $d\gg 0$, $HF_Z(d)$ agrees with a linear Hilbert
polynomial
\[
P_Z(d) \;=\; \deg_1(Z)\, d + \chi(\mathcal{O}_Z),
\]
where $\deg_1(Z) \in \mathbb{Z}_{\ge 0}$ is the coefficient of $d$.
Equivalently,
\[
\deg_1(Z)
\;=\;
\lim_{d\to\infty}\big(HF_Z(d)-HF_Z(d-1)\big).
\]
By construction, $\deg_1(Z)$ is \emph{scheme-theoretic}: it counts multiplicities
along $1$-dimensional components, and it is insensitive to embedded/isolated
$0$-dimensional structure.

\medskip

\noindent\textbf{(B) Practical extraction from the computed $h_k(d)$.}
For a cubic $f$ we apply this to $Z=\Sing(X_f)$ and write
\[
\deg_1\Sing(X_f)
\;:=\;
\lim_{d\to\infty}\Big(HF_{\Sing(X_f)}(d)-HF_{\Sing(X_f)}(d-1)\Big).
\]
In our computations of \S8.3 we recorded the componentwise minima
$h_k(d)$ among random perturbations in $\Phi_k$; the corresponding degree can be
read off from the stabilized first difference $\Delta h_k(d)=h_k(d)-h_k(d-1)$
for $d\gg 0$.

\begin{lemma}[Leading-coefficient monotonicity]\label{lem:degree-monotone}
Let $Z,Z'\subset \PP^6$ be closed subschemes with $\dim Z,\dim Z' \le 1$.
If $HF_Z(d)\ge HF_{Z'}(d)$ for all $d\ge 0$, then $\deg_1(Z)\ge \deg_1(Z')$.
\end{lemma}

\begin{proof}
For $d\gg 0$ we have $HF_Z(d)=P_Z(d)$ and $HF_{Z'}(d)=P_{Z'}(d)$, with
$P_Z(d)=\deg_1(Z)\,d+\chi(\mathcal{O}_Z)$ and $P_{Z'}(d)=\deg_1(Z')\,d+\chi(\mathcal{O}_{Z'})$.
If $\deg_1(Z)<\deg_1(Z')$, then $P_Z(d)<P_{Z'}(d)$ for all $d\gg 0$, contradicting
the assumed inequality $HF_Z(d)\ge HF_{Z'}(d)$ in large degree.
\end{proof}

\begin{filter}[Singular-locus degrees]\label{filter:sing-degree}
Fix $k\ne \ell$. If $\Phi_k \subseteq \Phi_\ell$, then for generic choices
$f_k\in \Phi_k$ and $f_\ell\in \Phi_\ell$ one has
\[
\deg_1\Sing(X_{f_k}) \;\ge\; \deg_1\Sing(X_{f_\ell}).
\]
Consequently, if $\deg_1\Sing(X_{f_k})<\deg_1\Sing(X_{f_\ell})$, then
$\Phi_k \not\subseteq \Phi_\ell$.
\end{filter}

\begin{cor}[Six immediate non-inclusions]\label{cor:six-degree-pairs}
The following six ordered pairs are ruled out by Filter~\ref{filter:sing-degree}:
\[
(1,13),\ (1,15),\ (1,17),\ (2,15),\ (3,15),\ (6,15).
\]
\end{cor}

\begin{proof}
By Table~3, the degrees of the top-dimensional singular-locus components are
\begin{itemize}
\item[] $\deg_1\Sing(X_1)=2$, $\deg_1\Sing(X_{13})=4$,
\item[] $\deg_1\Sing(X_{15})=9$, $\deg_1\Sing(X_{17})=3$,
\item[] $\deg_1\Sing(X_2)=\deg_1\Sing(X_3)=\deg_1\Sing(X_6)=5$.
\end{itemize}
Each of the listed pairs has $\deg_1$ strictly increasing from $k$ to $\ell$,
so $\Phi_k\subseteq \Phi_\ell$ is impossible by
Filter~\ref{filter:sing-degree}.
\end{proof}

\begin{proposition}[Updated candidate list after Filter 4]\label{prop:after-filter4}
Applying Filter~\ref{filter:sing-degree} to the $23$ pairs remaining after
Proposition~8.7 eliminates $6$ further pairs. Hence the only ordered pairs
not ruled out by Filters~1--4 are:
\begin{itemize}
\item[] $(1,18),\ (1,19),\ (2,17),\ (2,18),\ (3,17),\ (3,18),\ (5,16),\ (6,5),\ (6,16),\ (6,17),$
\item[] $ (6,19),\ (7,19),\ (9,16),\ (12,17),\ (13,18),\ (15,18),\ (20,18)$.
\end{itemize}
\end{proposition}

\subsection{Filter 5: Minimal Exponents and Limits of 1-Cycles}
\label{subsec:filter5-minexp}

In Filters 2--4, we utilized numerically generic points $f \in \Phi$ to extract stable algebraic invariants (such as Betti numbers), which tend to jump and become noisy at the highly symmetric $1$-PS limits. However, to efficiently rule out the majority of the remaining candidate pairs, we directly compare the geometric invariants of the polystable normal forms $\nf$. This is theoretically justified by the following lemma.

\begin{lemma}[Semicontinuity on singular $1$-cycles]\label{lem:1cycle-semi}
Let $X_\ell \rightsquigarrow X_k$ be a flat specialization of cubic fivefolds in $\PP^6$,
realized by a flat family $\mathcal{X}\to \Delta$ over a smooth pointed curve $(\Delta,0)$
with generic fiber $X_\ell$ and special fiber $X_k$.
Let $\Gamma_\ell\subset \Sing(X_\ell)$ be an irreducible curve component and define
\[
\widetilde{\alpha}_{\mathrm{gen}}(\Gamma_\ell)
\;:=\;
\widetilde{\alpha}_{p}(X_\ell)\qquad\text{for a general closed point }p\in \Gamma_\ell .
\]

Let $\Sing(\mathcal{X}/\Delta)\subset \mathcal{X}$ be the relative singular scheme, and let
$\overline{\Gamma}\subset \Sing(\mathcal{X}/\Delta)$ be the scheme-theoretic closure of $\Gamma_\ell$.
Write $(\overline{\Gamma}\cap X_k)^{(1)}$ for the union of the $1$-dimensional components of
$\overline{\Gamma}\cap X_k$ with their scheme structures, and set
\[
Z_k \;:=\; \bigl[(\overline{\Gamma}\cap X_k)^{(1)}\bigr]
\]
for its fundamental $1$-cycle.

Then:
\begin{enumerate}
\item $Z_k$ is an effective $1$-cycle supported on $\Sing(X_k)$, and
      $\deg Z_k=\deg \Gamma_\ell$.
\item There exists an irreducible curve component $\Gamma_k\subset \Supp(Z_k)$ such that
\[
\widetilde{\alpha}_{\mathrm{gen}}(\Gamma_k)\;\le\;\widetilde{\alpha}_{\mathrm{gen}}(\Gamma_\ell).
\]
\item If $\Gamma_k$ is a curve component of $\Sing(X_k)$ with multiplicity
      $m_k(\Gamma_k)$ in the fundamental $1$-cycle $[\Sing(X_k)]^{(1)}$,
      then the coefficient of $\Gamma_k$ in $Z_k$ is at most $m_k(\Gamma_k)$.
\end{enumerate}
\end{lemma}

\begin{proof}
Since $\overline{\Gamma}\subset \Sing(\mathcal{X}/\Delta)$, its special fiber
$\overline{\Gamma}\cap X_k$ is a closed subscheme of $\Sing(X_k)$.
Hence $Z_k$ is effective and supported on $\Sing(X_k)$.

The cycle $Z_k$ is the specialization of the cycle $[\Gamma_\ell]$ inside the proper family
$\overline{\Gamma}\to \Delta$. In particular, its degree (equivalently, its class in
$A_1(\PP^6)\cong \mathbb{Z}$) is preserved under specialization, so $\deg Z_k=\deg \Gamma_\ell$.

Choose an irreducible component $\Gamma_k$ of $\Supp(Z_k)$ that is dominated by $\overline{\Gamma}$.
Let $p_0\in \Gamma_k$ be a general closed point. After possibly shrinking $\Delta$,
we may choose a curve of points $p(t)\in \overline{\Gamma}\cap X_t$ specializing to $p_0$.
By lower semicontinuity of minimal exponents in flat families (see, e.g.,
\cite[Theorem 2.3]{Par25}),
\[
\widetilde{\alpha}_{p_0}(X_k)\;\le\;\liminf_{t\to 0}\widetilde{\alpha}_{p(t)}(X_t).
\]
For $t\neq 0$ general, $p(t)$ is a general point of $\Gamma_\ell$, hence the right-hand side equals
$\widetilde{\alpha}_{\mathrm{gen}}(\Gamma_\ell)$. Since $p_0$ is general on $\Gamma_k$,
the left-hand side equals $\widetilde{\alpha}_{\mathrm{gen}}(\Gamma_k)$, proving (2).

Finally, (3) follows from the surjection of local rings at the generic point $\eta_{\Gamma_k}$:
\[
\mathcal{O}_{\Sing(X_k),\eta_{\Gamma_k}}
\twoheadrightarrow
\mathcal{O}_{(\overline{\Gamma}\cap X_k)^{(1)},\eta_{\Gamma_k}},
\]
so lengths (hence cycle multiplicities) can only decrease.
\end{proof}

\begin{filter}[Minimal exponent and cycle degree]\label{filter:minexp-cycle}
Fix $k\neq \ell$ and suppose $\Phi_k\subseteq \Phi_\ell$.
Let $X_\ell=V(\nf_\ell)$ and $X_k=V(\nf_k)$ be the closed-orbit representatives.

Let
\[
Z_\ell^{BC}
\;:=\;
\sum_{\substack{\Gamma\subset \Sing(X_\ell)\\ \Gamma\ \text{of Type B or C}}}
m_\ell(\Gamma)\,[\Gamma]
\]
be the effective $1$-cycle supported on the Type~B/C curve components of $\Sing(X_\ell)$,
where $m_\ell(\Gamma)$ denotes the multiplicity of $\Gamma$ in the fundamental $1$-cycle
$[\Sing(X_\ell)]^{(1)}$. Set $D_\ell:=\deg Z_\ell^{BC}$.

Then any flat degeneration $X_\ell\rightsquigarrow X_k$ forces the existence of an effective
$1$-cycle
\[
Z_k^{BC}
\;=\;
\sum_{\substack{\Gamma\subset \Sing(X_k)\\ \Gamma\ \text{of Type B or C}}}
a(\Gamma)\,[\Gamma]
\]
supported on the Type~B/C components of $\Sing(X_k)$ such that
\[
\deg Z_k^{BC}=D_\ell,
\qquad
0\le a(\Gamma)\le m_k(\Gamma)\ \ \text{for all Type~B/C components }\Gamma\subset \Sing(X_k),
\]
where $m_k(\Gamma)$ is the multiplicity of $\Gamma$ in $[\Sing(X_k)]^{(1)}$.
In particular, if no such choice of coefficients $a(\Gamma)$ exists (e.g.\ if $D_\ell>0$
but $\Sing(X_k)$ has no Type~B/C component, or if $D_\ell$ cannot be realized as a degree
combination within the available multiplicities), then $\Phi_k\not\subseteq \Phi_\ell$.
\end{filter}

\begin{proposition}[Effect of Filter 5]\label{prop:filter5-effect}
Applying Filter~\ref{filter:minexp-cycle} to the $17$ candidate pairs remaining after Filter~4 eliminates $13$ pairs. 
Specifically, the following pairs are discarded:
\begin{itemize}
\item[] $(1, 18),\ (1, 19),\ (2, 17),\ (2, 18),\ (3, 17),\ (3, 18),\ (6, 17)$,
\item[] $(6, 19),\ (7, 19),\ (12, 17),\ (13, 18),\ (15, 18),\ (20, 18)$.
\end{itemize}
\end{proposition}

\begin{proof}
We analyze the limit of the $1$-dimensional components based on the data in Table~\ref{tab:section5-sing} and Table~\ref{tab:me-pattern-index-noA0}.

\noindent $\bullet$ \textbf{Pairs with $\ell=18$:} $(1, 18), (2, 18), (3, 18), (13, 18), (15, 18), (20, 18)$. \\
$X_{18}$ has two disjoint lines $L_1, L_2$ of Type B ($\widetilde{\alpha}_{\mathrm{gen}}=7/3$). Thus it carries a degree $2$ cycle with $\widetilde{\alpha}_{\mathrm{gen}}=7/3$.
\begin{itemize}
\item For $k \in \{1, 2, 3, 13\}$, all curve components of $X_k$ are of Type A1 ($\widetilde{\alpha}_{\mathrm{gen}}=5/2 > 7/3$). They cannot support the limit.
\item For $k=15$, the only component with $\widetilde{\alpha}_{\mathrm{gen}} \le 7/3$ is $L$ (Type C, degree $1$). However, $L$ has scheme-theoretic degree $1$ in $\Sing(X_{15})$ (as $\deg_1 \Sing(X_{15}) = 9 = 4+4+1$), which cannot absorb the limit of two lines (degree $2$).
\item For $k=20$, the components are $C$ (elliptic quartic, degree $4$) and $L$ (line, degree $1$). A limit of two lines must be supported on lines, but $X_{20}$ only has a degree $1$ line, which is insufficient.
\end{itemize}
Thus all these pairs are excluded.

\medskip
\noindent $\bullet$ \textbf{Pairs with $\ell=17$:} $(2, 17), (3, 17), (6, 17), (12, 17)$. \\
$X_{17}$ contains a smooth conic $C$ of Type B ($\widetilde{\alpha}_{\mathrm{gen}}=7/3$, degree $2$).
For $k \in \{2, 3, 6, 12\}$, all curve components of $X_k$ are of Type A1 ($\widetilde{\alpha}_{\mathrm{gen}}=5/2 > 7/3$). They cannot support the limit of the conic. Thus these are excluded.

\medskip
\noindent $\bullet$ \textbf{Pairs with $\ell=19$:} $(1, 19), (6, 19), (7, 19)$. \\
$X_{19}$ contains a smooth conic $C$ of Type C ($\widetilde{\alpha}_{\mathrm{gen}}=7/3$, degree $2$).
For $k \in \{1, 6, 7\}$, all curve components of $X_k$ are of Type A1 ($\widetilde{\alpha}_{\mathrm{gen}}=5/2 > 7/3$). Thus these are excluded.
\end{proof}

\begin{cor}[Updated candidate list after Filter 5]\label{cor:after-filter5}
The only ordered pairs not ruled out by Filters 1--5 are the following $4$ pairs:
\[
(5, 16),\ (6, 5),\ (6, 16),\ (9, 16).
\]
\end{cor}

\subsubsection{Filter 6: spectra of isolated singularities (and the last three pairs)}\label{subsubsec:filter6}

We now eliminate the remaining four pairs.
The key new ingredient is the Steenbrink spectrum of an \emph{isolated} hypersurface singularity, which satisfies a strong semicontinuity property.
This will rule out $(5,16)$ using only the quasi-homogeneous data recorded in Table~\ref{tab:sing-recap}.
The remaining three pairs are then excluded by a generic-stabilizer obstruction.

\medskip

\paragraph{(A) Spectrum semicontinuity for quasi-homogeneous isolated singularities.}

Recall that for an isolated hypersurface singularity $(f,0)$ the Steenbrink spectrum
\[ \Sp(f)=\sum_{\alpha\in\Q} n_\alpha\, t^\alpha \]
is a finite multiset of rational numbers (the spectral numbers) with multiplicities $n_\alpha\in\Z_{\ge 0}$.
We will use the following semicontinuity theorem.

\begin{thm}[Steenbrink]\label{thm:steenbrink-semicont}
Let $\{f_u\}$ be a holomorphic family of isolated hypersurface singularities.
For every half-open interval $(c,c+1]\subset\R$, the function
\[ u\longmapsto \#\bigl\{\alpha\in\Sp(f_u)\,\big|\, c<\alpha\le c+1\bigr\} \]
is upper semicontinuous.
Equivalently, under specialization $f_u\rightsquigarrow f_0$ one has
\[ \#\bigl(\Sp(f_0)\cap(c,c+1]\bigr)\ \ge\ \#\bigl(\Sp(f_u)\cap(c,c+1]\bigr)\qquad \text{for all }c\in\R. \]
\end{thm}

\begin{proof}
This is \cite[Thm.~(2.4)]{Ste85}.
\end{proof}

For quasi-homogeneous isolated singularities the spectrum can be computed from the weights.
We will use the following standard recipe (see e.g. \cite{Sai71}).

\begin{lemma}[Spectrum from the weighted Milnor algebra]\label{lem:spectrum-from-hilbert}
Let $g\in\C[x_1,\dots,x_c]$ be quasi-homogeneous of weighted degree $D$ with weights $\wt(x_i)=w_i\in\Z_{>0}$, and assume that $g$ has an isolated singularity at the origin.
Let $M(g)=\C[x_1,\dots,x_c]/(\partial g)$ be the Milnor algebra, graded by the weighted degree, and write its Hilbert series as
\[ H_{M(g)}(t)=\sum_{m\ge 0} h_m t^m. \]
Then the spectrum of $g$ is the multiset
\[ \Sp(g)=\bigl\{\ \alpha(m):=\tfrac{m+\sum_i w_i}{D}-1\ \text{ with multiplicity }h_m\ \bigr\}. \]
\end{lemma}

\begin{proof}
Since $g$ is weighted homogeneous with an isolated singularity, the partial derivatives form a graded regular sequence of degrees $D-w_i$.
Hence
\[ H_{M(g)}(t)=\prod_{i=1}^c \frac{1-t^{D-w_i}}{1-t^{w_i}}, \]
and the spectral numbers are obtained by the usual shift from the weighted degrees of a homogeneous basis of $M(g)$ (cf. \cite{Sai71}).
\end{proof}

We now apply this to the two isolated singularity types appearing on $X_5$ and $X_{16}$ (Table~\ref{tab:sing-recap}).
The singularities $E_{20}$ and $E_{21}$ are quasi-homogeneous with
\[ E_{20}:\ \mathrm{QH}(1,2,2;6)_{20},\qquad E_{21}:\ \mathrm{QH}(2,3,5;12)_{21}. \]
By Lemma~\ref{lem:spectrum-from-hilbert} one finds
\begin{align*}
\Sp(E_{20})&=\bigl\{-\tfrac{1}{6},\ 0,\ (\tfrac{1}{6})^{\times 3},\ (\tfrac{1}{3})^{\times 3},\ (\tfrac{1}{2})^{\times 4},\ (\tfrac{2}{3})^{\times 3},\ (\tfrac{5}{6})^{\times 3},\ 1,\ \tfrac{7}{6}\bigr\},\\
\Sp(E_{21})&=\bigl\{-\tfrac{1}{6},\ 0,\ \tfrac{1}{12},\ \tfrac{1}{6},\ (\tfrac{1}{4})^{\times 2},\ (\tfrac{1}{3})^{\times 2},\ \tfrac{5}{12},\ (\tfrac{1}{2})^{\times 3},\ \tfrac{7}{12},\ (\tfrac{2}{3})^{\times 2},\ (\tfrac{3}{4})^{\times 2},\ \tfrac{5}{6},\ \tfrac{11}{12},\ 1,\ \tfrac{7}{6}\bigr\}.
\end{align*}

Set $c:=\tfrac{19}{24}$.
Then
\[ \#\bigl(\Sp(E_{20})\cap(c,c+1]\bigr)=5 \qquad\text{but}\qquad \#\bigl(\Sp(E_{21})\cap(c,c+1]\bigr)=4, \]
because in $(c,c+1]$ the spectrum of $E_{20}$ contributes
$\{(\tfrac{5}{6})^{\times 3},1,\tfrac{7}{6}\}$ whereas the spectrum of $E_{21}$ contributes
$\{\tfrac{5}{6},\tfrac{11}{12},1,\tfrac{7}{6}\}$.
By Theorem~\ref{thm:steenbrink-semicont} it follows that $E_{21}$ cannot occur as a specialization of $E_{20}$.

\begin{proposition}\label{prop:filter6-516}
The inclusion $\Phi_5\subset\Phi_{16}$ is impossible.
\end{proposition}

\begin{proof}
Assume $\Phi_5\subset\Phi_{16}$.
Then the closed orbit representative $X_5$ lies in the closure of $\Phi_{16}$.
Choose a one-parameter degeneration $X_t\in\Phi_{16}$ with limit $X_0\cong X_5$.
For $t\ne 0$, the variety $X_t$ has two isolated singular points of type $E_{20}$ (Proposition~\ref{prop:case-k16}); the limit $X_0$ has two isolated singular points of type $E_{21}$.
Working in disjoint analytic neighborhoods of the singular points, we obtain families of isolated hypersurface singularities and can apply Theorem~\ref{thm:steenbrink-semicont}.
This contradicts the strict inequality
$\#\bigl(\Sp(E_{21})\cap(c,c+1]\bigr)<\#\bigl(\Sp(E_{20})\cap(c,c+1]\bigr)$.
\end{proof}

\paragraph{(B) The remaining three pairs.}

After Proposition~\ref{prop:filter6-516}, the only remaining candidate inclusions are
\[ \Phi_6\subset\Phi_{5},\qquad \Phi_6\subset\Phi_{16},\qquad \Phi_9\subset\Phi_{16}. \]
We exclude these by comparing the generic connected stabilizer tori.

\begin{lemma}[Generic stabilizer type obstruction]\label{lem:generic-stabilizer-obstruction}
Let $G$ act on a projective variety $X$ and let $S\subset X$ be a $G$-stable irreducible subset.
Assume that a general point of $S$ has connected stabilizer a one-dimensional torus $T\cong\C^*$.
If $S\subset S'$ for another $G$-stable irreducible subset $S'\subset X$ whose general point has connected stabilizer a one-dimensional torus $T'$, then $T'$ is conjugate to $T$ in $G$.
In particular, if both $T$ and $T'$ are diagonal one-parameter subgroups of $\SL_7$ then their weight multisets must agree.
\end{lemma}

\begin{proof}
Under specialization the connected stabilizer can only increase.
Hence for $S\subset S'$ the generic torus stabilizer of $S'$ embeds into the generic torus stabilizer of $S$.
Since both are one-dimensional tori, the embedding is an isomorphism, i.e. the two tori are conjugate.
For diagonal one-parameter subgroups of $\SL_7$, conjugacy is equivalent to equality of the weight multiset.
\end{proof}

In our setting, the closed orbit representatives $X_k$ are fixed by the diagonal one-parameter subgroups $\lambda_k$ listed in Section~\ref{algorithm}.
For $k\in\{5,6,9,16\}$ these are
\begin{align*}
\lambda_5(t)&=\diag(t^{4},t^{2},t^{1},t^{0},t^{-1},t^{-2},t^{-4}),\\
\lambda_6(t)&=\diag(t^{5},t^{3},t^{2},t^{1},t^{-1},t^{-4},t^{-6}),\\
\lambda_9(t)&=\diag(t^{2},t^{2},t^{0},t^{0},t^{-1},t^{-1},t^{-2}),\\
\lambda_{16}(t)&=\diag(t^{2},t^{1},t^{0},t^{0},t^{0},t^{-1},t^{-2}).
\end{align*}
Their weight multisets are pairwise distinct, so Lemma~\ref{lem:generic-stabilizer-obstruction} rules out the three remaining inclusions.

\begin{cor}\label{cor:all-pairs-eliminated}
After applying Filters~1--6, there are no remaining candidate inclusions $\Phi_k\subset\Phi_\ell$ with $k\ne\ell$.
In particular, Theorem~\ref{non-iclusion} holds.
\end{cor}

\section{Concluding Remarks and Future Directions}
\label{sec:future}

The results of this paper turn the strictly semistable GIT boundary for cubic fivefolds into an explicit and tractable set of test degenerations.
Concretely, we obtain:
the list of $21$ $\SL(7)$-inequivalent boundary families (Theorem~\ref{21-list}),
explicit polystable normal forms and component dimensions (Section~\ref{sec:closed-orbit} and Table~\ref{tab:case-summary}),
a complete description of the singular loci of the closed-orbit representatives (Section~\ref{sec:5} and Table~\ref{tab:section5-sing}),
the six isolated analytic types that occur on the boundary (Table~\ref{tab:sing-recap}),
the sharp minimal-exponent equalities at the critical threshold (Theorems~\ref{thm:minexp-73} and~\ref{thm:global-minimal-exponents}),
and the sparse Kirwan wall-crossing adjacency graph among boundary components (Section~\ref{sec:adjacency} and Theorem~\ref{thm:adjacency}).
Together with the non-inclusion statement (Theorem~\ref{non-iclusion}),
this provides a concrete ``dictionary'' for the boundary geometry of the moduli space.

We close by outlining several directions where such explicit boundary models should be useful.
A long-term motivation is to leverage these degenerations in Hodge-theoretic and integrable-systems approaches to the
intermediate Jacobians of cubic fivefolds (Section~\ref{subsec:schottky}).
To make that program viable, however, one needs a substantially refined understanding of
(i) the deformation and moduli theory of the extremal isolated singularities (Section~\ref{subsec:extremal-analysis}),
(ii) the resulting limit mixed Hodge structures for boundary degenerations (Section~\ref{subsec:lmhs}),
and (iii) the extent to which the present methodology scales to higher-dimensional cubic hypersurfaces (Section~\ref{subsec:higher-dim}).
We also record a conjectural structural principle suggested by our computations (Section~\ref{subsec:conjecture}).

\subsection{Detailed analysis of extremal cubic fivefold singularities}
\label{subsec:extremal-analysis}

Isolated singular points occur only for a proper subset of the boundary families (Table~\ref{tab:section5-sing}),
and, on general closed-orbit representatives, every isolated boundary germ belongs to one of the six quasi-homogeneous analytic types listed in
Table~\ref{tab:sing-recap}.
A key outcome of Section~\ref{sec:minimal-exponents} is that each of these isolated types realizes the equality case in Park's threshold:
Theorem~\ref{thm:minexp-73} shows
\[
\widetilde{\alpha}_x(X)=\frac{7}{3}=\frac{n+1}{d}\qquad\text{for }(n,d)=(6,3)
\]
at every isolated boundary point.
Motivated by this sharp borderline behavior, we singled out these six analytic types as
\emph{extremal} (Definition~\ref{def:extremal-sing}).

Beyond their role in the GIT classification, these singularities appear to be a natural and rigid collection of
``boundary'' local models for cubic fivefolds, and they invite a more intrinsic analysis.
Several concrete problems seem especially compelling:

\begin{itemize}
  \item \textbf{Local deformation theory and comparison with wall-crossing.}
  A systematic study of the miniversal deformation spaces (semi-universal unfoldings) of the six extremal types
  should clarify the local analytic structure of the moduli space near boundary points.
  Since the singularities are quasi-homogeneous, one expects the Milnor and Tjurina numbers to coincide, and
  the deformation space can be approached via explicit Tjurina algebras.
  A natural goal is to understand how (semi)stability conditions cut out chambers inside these bases and how the resulting local pictures reflect
  Kirwan wall-crossing (Definition~\ref{def:adjacency}, Theorem~\ref{thm:adjacency}).

  \item \textbf{Connections to $K$-stability and $K$-moduli.}
  With the rapid development of $K$-moduli for Fano varieties, it is natural to ask how the classical GIT compactification compares to the
  $K$-moduli compactification for cubic fivefolds.
  While minimal exponents are Hodge-theoretic invariants and do not directly encode discrepancies,
  the uniform inequality $\widetilde{\alpha}_x(X)=\frac{7}{3}>1$ for the extremal isolated types
  suggests that these singularities should be relatively mild from several birational and Hodge-theoretic perspectives.
  It would therefore be very interesting to determine for each explicit extremal local model whether it is klt,
  and more broadly to test local $K$-stability by computing normalized volumes and related stability thresholds.

  \item \textbf{Finer Hodge-theoretic and topological invariants.}
  Since these extremal germs govern the most degenerate isolated limits on the GIT boundary,
  their refined invariants (Hodge spectrum, monodromy, Milnor lattice, Brieskorn lattice, etc.)
  should feed directly into computations of limiting Hodge data for global degenerations (cf.\ Section~\ref{subsec:lmhs}).
\end{itemize}

\subsection{Limit mixed Hodge structures and moduli compactifications}
\label{subsec:lmhs}

For cubic threefolds and fourfolds, a recurring theme is that explicit GIT compactifications interact in a meaningful way with
Hodge-theoretic compactifications of period domains (e.g.\ \cite{CG72,Laz09}).
For cubic fivefolds, understanding the degeneration of periods and intermediate Jacobians \cite{Col86}
requires an analysis of the associated limit mixed Hodge structures (LMHS).

From this perspective, two kinds of input are indispensable:
one needs explicit local models of singular fibers, and one needs control over how boundary strata meet.
The present paper provides precisely this data in a concrete form:
explicit closed-orbit boundary representatives (Table~\ref{tab:case-summary}),
their singular loci (Section~\ref{sec:5}, Table~\ref{tab:section5-sing}),
and the adjacency relations among boundary components under Kirwan wall-crossing (Section~\ref{sec:adjacency}, Theorem~\ref{thm:adjacency}).
Combined with the local invariants of the extremal isolated types (Section~\ref{subsec:extremal-analysis}),
this should make it feasible to construct transverse slices to boundary strata and to compute local Picard--Lefschetz transformations
and monodromy cones in explicit families.
A natural medium-term objective is therefore a systematic computation of LMHS for selected degenerations built from our $21$ boundary models,
with an eye toward comparisons between the GIT boundary and Hodge-theoretic compactifications.

\subsection{Towards the Schottky problem for cubic fivefolds via integrable systems}
\label{subsec:schottky}

Let $X\subset \PP^{6}$ be a smooth cubic fivefold.
Its intermediate Jacobian $J(X)$ is a principally polarized abelian variety (PPAV) of dimension $21$,
and a Torelli-type statement is available in this setting (see \cite{Voi86}).
Thus the period map may be viewed as embedding the $35$-dimensional moduli space of smooth cubic fivefolds
into the $231$-dimensional PPAV moduli space $\mathcal{A}_{21}$.
A basic and widely open problem is to characterize, or even effectively constrain, this $35$-dimensional locus inside $\mathcal{A}_{21}$.

In the classical Schottky problem for Jacobians of curves, integrable systems play a decisive role:
Shiota's theorem characterizes Jacobians among PPAVs via the KP hierarchy \cite{Shiota86}.
This motivates the question of whether the locus of intermediate Jacobians of cubic fivefolds admits
an analogous characterization in terms of modular/differential constraints on theta functions.
At present, even formulating a plausible candidate hierarchy is nontrivial, and any integrable-systems approach must confront degenerations.

Here the boundary classification becomes relevant.
Theta-functional constraints are typically most informative when one understands their behavior under degeneration,
where period matrices and theta functions develop controlled asymptotics.
Our $21$ explicit boundary families (Theorem~\ref{21-list} and Table~\ref{tab:case-summary}),
together with the specialization relations encoded by Theorem~\ref{thm:adjacency},
provide a concrete set of test degenerations and a combinatorial skeleton for how boundary strata meet.
A natural open-ended program is therefore:
use these explicit degenerations to compute LMHS and limiting period data (as in Section~\ref{subsec:lmhs}),
and then investigate whether the resulting locus in $\mathcal{A}_{21}$ satisfies distinctive modular or differential constraints
of an integrable-systems flavor.

\subsection{Extremal singularities on cubic sixfolds and higher dimensions}
\label{subsec:higher-dim}

The computational and structural pipeline used in this paper
(maximal strictly semistable supports via convex geometry, Luna's centralizer reduction, and polystability checks via
the convex-hull criterion or the Casimiro--Florentino criterion)
relies on general features of reductive group actions and is therefore, in principle, portable to higher-dimensional cubic hypersurfaces,
albeit with rapidly growing combinatorial complexity.

An immediate target is the GIT compactification for cubic sixfolds ($X\subset \PP^{7}$).
Park's stability criterion \cite[Theorem~A]{Par25} identifies the critical minimal exponent threshold
as $\frac{n+1}{d}$ for degree $d$ hypersurfaces in $\PP^{n}$; in particular, for cubic sixfolds $(n,d)=(7,3)$ the critical value is $\frac{8}{3}$.
Our results for cubic fivefolds show that the isolated boundary singularities realize the equality case
$\widetilde{\alpha}_x(X)=\frac{7}{3}$ (Theorem~\ref{thm:minexp-73}),
suggesting a broader principle: in higher dimensions, the strictly semistable boundary may again be governed by a distinguished collection
of quasi-homogeneous isolated singularities sitting exactly at the critical threshold.
Identifying such ``extremal'' types for cubic sixfolds would be a natural next step toward a higher-dimensional boundary dictionary.

\subsection{A hypersurface version of maximal polystability}
\label{subsec:conjecture}

One striking uniform phenomenon emerged throughout our explicit classification:
for every maximal $T$-strictly semistable support $I(\mathbf{r}_k)_{\ge 0}$ produced by the convex-geometric algorithm,
the associated $1$-PS degeneration
\[
\phi_k \;=\; \lim_{t\to 0}\lambda_{\mathbf{r}_k}(t)\cdot f_k
\]
of a general polynomial $f_k$ with $\Supp(f_k)=I(\mathbf{r}_k)_{\ge 0}$
landed on a \emph{polystable} point for the residual action of the corresponding centralizer.
In other words, once the support is maximal among strictly semistable ones,
the weight-$0$ initial form produced by the limiting process appears to have closed orbit automatically.

The fact that this automatic polystability holds flawlessly without a single exception across all 21 highly diverse boundary families computed in Section 4—despite the wide variation in their residual centralizers, ranging from tori to complicated non-abelian groups—provides overwhelming empirical evidence.

Heuristically, the convex-geometric maximality of the support dictates a combinatorial rigidity: one cannot add further monomials without making the support unstable. Intuitively, this extremality forces the barycenter to be strictly "trapped" in the relative interior of the weight polytope for the residual action, which exactly translates to polystability.

Motivated by this evidence, we record the following hypersurface-level formulation,
tailored to the natural $\SL(n{+}1)$-action on spaces of homogeneous polynomials.

\begin{conjecture}[Hypersurface Maximal Polystability]
\label{conj:maximal-polystable}
Fix integers $n\ge 1$ and $d\ge 2$, and let
\[
G=\SL(n{+}1), \qquad
W=\Sym^{d}(\C^{n{+}1})\cong \C[x_0,\dots,x_n]_d
\]
with the natural linear $G$-action.
Fix the diagonal maximal torus $T\subset G$.
Let
\[
I=\mathbb{Z}^{n+1}_{(d)}\cap \mathbb{Z}^{n+1}_{\ge 0}, \qquad
\eta=\Bigl(\tfrac{d}{n+1},\dots,\tfrac{d}{n+1}\Bigr)
\]
be the exponent simplex and its barycenter.
For $\mathbf{r}\in \mathbb{Z}^{n+1}_{(0)}$ define
\[
I(\mathbf{r})_{\ge 0}=\{\mathbf{i}\in I\mid \mathbf{r}\cdot \mathbf{i}\ge 0\}, \qquad
I(\mathbf{r})_{=0}=\{\mathbf{i}\in I\mid \mathbf{r}\cdot \mathbf{i}=0\}.
\]
We say that $I(\mathbf{r})_{\ge 0}$ is a \emph{maximal $T$-strictly semistable support} if
$\eta\in \mathrm{Conv}\bigl(I(\mathbf{r})_{\ge 0}\bigr)$, $\eta\in \partial\mathrm{Conv}\bigl(I(\mathbf{r})_{\ge 0}\bigr)$
(relative boundary inside the affine hyperplane $\mathrm{wt}=d$),
and $I(\mathbf{r})_{\ge 0}$ is maximal for these properties with respect to inclusion among subsets of $I$ of the form
$I(\mathbf{r}')_{\ge 0}$.
Let $\lambda_{\mathbf{r}}:\mathbb{G}_m\to T$ be the $1$-PS given by
$\lambda_{\mathbf{r}}(t)=\mathrm{diag}(t^{r_0},\dots,t^{r_n})$.
For a general form
\[
f=\sum_{\mathbf{i}\in I} a_{\mathbf{i}}x^{\mathbf{i}}\in W
\quad\text{with}\quad
\Supp(f)=I(\mathbf{r})_{\ge 0},
\]
consider the $1$-PS limit
\[
f_0 \;:=\; \lim_{t\to 0}\lambda_{\mathbf{r}}(t)\cdot f
\;=\;
\sum_{\mathbf{i}\in I(\mathbf{r})_{=0}} a_{\mathbf{i}}x^{\mathbf{i}}
\;\in\;
W^{\lambda_{\mathbf{r}}(\mathbb{G}_m)}.
\]
Then $f_0$ is polystable.
\end{conjecture}

\medskip
\noindent\textbf{Possible strategy for a proof.}
One may hope to prove Conjecture~\ref{conj:maximal-polystable} by combining Luna's centralizer reduction \cite{Lun75}
with the Casimiro--Florentino criterion \cite{CF12} (or, in toric cases, the convex-hull criterion,
Theorem~\ref{thm:convex-hull}).
Luna reduction translates closedness of the $G$-orbit of the $1$-PS limit $f_0$
into polystability for the residual action of the centralizer
$C_G(\lambda(\mathbb{G}_m))$ on the fixed locus $W^{\lambda(\mathbb{G}_m)}$.
After choosing a maximal torus in the centralizer, the Casimiro--Florentino/convex-hull criteria
reduce polystability to the existence of a strictly positive ``balancing'' linear identity in the real character space.

In the hypersurface case $W=\Sym^{d}\C^{n+1}$, this can be phrased as the existence of positive coefficients
$a_\alpha>0$ (summing to $1$) such that
\[
\sum_{\alpha\in \Supp(f_0)} a_\alpha\,\alpha
\;=\;
\frac{d}{n+1}(1,\ldots,1).
\]
The conjecture asserts that the maximal-support hypothesis forces such a strictly positive solution,
equivalently that the relevant weight polytope already contains the barycenter in its interior for the residual action.
If true, this would streamline the search for closed orbits in many hypersurface GIT problems:
once maximal strictly semistable supports are enumerated, the residual polystability check would become essentially automatic.

\end{document}